\documentclass[12pt]{article}
\usepackage{amsmath, latexsym, amsfonts, amssymb, amsthm, amscd}
\usepackage[left=1.5cm, right=1.5cm, top=2.5cm, bottom=2.5cm]{geometry}
\usepackage{upquote}
\usepackage{color}
\usepackage{setspace}
\usepackage[labelfont=bf]{caption}

\usepackage[utf8]{inputenc}
\usepackage[T1]{fontenc}
\usepackage[english]{babel}

\usepackage{mathrsfs}
\usepackage{dsfont}

\usepackage{hyperref}
\usepackage{bookmark}

\allowdisplaybreaks



\newtheorem{corollary}{Corollary}
\newtheorem{proposition}{Proposition}
\newtheorem{lemma}{Lemma}
\newtheorem{remark}{Remark}
\newtheorem{definition}{Definition}
\newtheorem{theorem}{Theorem}

\newtheorem{assumption}{Assumption}

\newtheorem{notation}{Notation}

\newcommand\mycom[2]{\genfrac{}{}{0pt}{}{#1}{#2}}

\begin{document}
\title{Trait-dependent branching particle systems with competition and multiple offspring}
\author{G. Berzunza\footnote{ {\sc Institut f\"ur Mathematische Stochastik, Georg-August-Universit\"at G\"ottingen. Goldschmidtstrasse 7, 37077 G\"ottingen, Germany.} E-mail: gabriel.berzunza-ojeda@uni-goettingen.de},\,\, A. Sturm \footnote{ {\sc Institut f\"ur Mathematische Stochastik, Georg-August-Universit\"at G\"ottingen. Goldschmidtstrasse 7, 37077 G\"ottingen, Germany.} E-mail: asturm@math.uni-goettingen.de} \, \, and \, \, A. Winter\footnote{ {\sc Fakult\"at f\"ur Mathematik, Universit\"at Duisburg-Essen. Thea-Leymann-Strasse 9, 45127 Essen, Germany.} E-mail: anita.winter@uni-due.de}}
\maketitle


\begin{abstract}
In this work we model the dynamics of a population that evolves as a continuous time branching process with a trait structure and ecological interactions in form of mutations and competition between individuals. We generalize existing microscopic models by allowing individuals to have multiple offspring at a reproduction event. Furthermore, we allow the reproduction law to be influenced both by the trait type of the parent as well as by the mutant trait type.

We look for tractable large population approximations. More precisely, under some natural assumption on the branching and mutation mechanisms,  we establish a superprocess limit as the solution of a well-posed martingale problem.
Standard approaches do not apply in our case due to the lack of the branching property, which is a consequence of the dependency created by the competition between individuals. 
For showing uniqueness we therefore had to develop a generalization of Dawson's Girsanov Theorem that may be of independent interest.
\end{abstract}

\bigskip

\noindent {\sc Key words and phrases}: Darwinian evolution; branching process; competition-mutation dynamics; interacting particle system; nonlinear superprocesses; adaptive dynamics; limit theorem.

\noindent {\sc Subject Classes}: 60J80; 60J68; 60K35.

\section{Introduction}

The study of interactions between organisms and their environment which influence their reproductive success and contribute to genotype and phenotype variation 
 is one of the main questions in evolutionary ecology and population genetics.  In this paper, we are interested in modelling the dynamics of populations by emphasizing the ecological interactions, namely the competition between individuals for limited resources, where each individual is characterized by a quantitative trait which remains constant during the individual's life and which is passed on to offspring unless a mutation occurs. Motivated by the work of
Bolker-Pacala \cite{Bo1997} and Dieckmann-Law \cite{DiU1999}, several models have been rigorously developed in this context. Firstly, Fournier and M\'el\'eard \cite{Four2004} considered spatial seed models. Secondly, Champagnat, et al. \cite{Cha2008}, Jourdain, et al. \cite{Jou}, M\'el\'eard and Viet Chi
\cite{MeTra2012} studied phenotypic trait structured populations when the mutation kernel behaves essentially as a Gaussian law or belongs to the domain of attraction of a stable law. Finally, M\'el\'eard and Viet Chi \cite{Me2012} considered also structured populations whose dynamics depends on the past. In these works, the population is essentially modelled by a continuous time pure birth-death process with mutation. The birth and death rates of this Markov process may depend on each individual's trait and the interactions between them. While traits are normally hereditarily transmitted from a parent to its offspring with a (small) probability a mutation may occur. In this case, the offspring makes an instantaneous mutation step at birth to a new trait value. This mutation step is driven by a mutation kernel (a probability kernel) that depends only on the parent trait. The authors then pass from the microscopic description of the population on the level of individuals to a macroscopic description on the level of population mass distribution in the trait space.

It is important to point out that whereas most organisms rely on binary reproduction for propagation, many other use alternative mechanisms which include multiple offspring in order to reproduce and remain competitive; see for example \cite{Bio1}, \cite{Bio2} and \cite{Bio3}.
Thus, in this work, we are interested in generalizing this microscopic model in that we allow individuals to have multiple offspring at a reproduction event.
More precisely, we consider a general offspring distribution where the number of children produced by each individual depends on its trait as well as on the new trait that appears in case a mutation occurs. We have a number of scenarios in mind in which such a dependence may occur. One such scenario is modelling so-called "jackpot" events, introduced in a seminal paper of Luria and Delbr\"uck \cite{LuriaDelbrueck1943}, in which particular mutants rapidly create a sizeable mutant subpopulation - in the original famous Luria and Delbr\"uck experiment because they are more resistant to detrimental effects of the environment.

In our model, this mutant subpopulation is created instantaneously and we refer to them simply as mutant offspring.
Let us consider a particular scenario, namely the evolution of different strains of virus populations or other microparasites with fast adaptation, 
in order to motivate a dependence of the offspring distribution on the parent as well as on the mutant strain.

Virus populations evolve as subpopulations within hosts that in turn infect other hosts and thus create new evolving subpopulations. Within each host subpopulations of different strains of the virus will generally be present
(due to the initial infection but in particular due to mutation during the infection)
and their evolution is affected by the immune system, that reacts to the presence of particular strains that it has already recognized. This leads to an increased death rate of a prevalent subpopulation, which we model by a competition term, effectively a death rate that depends on the size and proximity (in type space) of the entire virus population.

A mutant type -sufficiently different from the parent type say- may on the other hand quickly establish a sizeable subpopulation, whose size could also depend on the intrinsic fitness of their trait type, before they are targeted by the immune system ("immune escape"). Admittedly, a dependence on the size and proximity of the entire virus population (as in the competition term), which is shaping the current immune system and its response,  could be even more desirable. But we view the dependence on the parent type as a first step in this direction, which is in particular realistic if a mutation to an epitope site results in a completely new phenotype of the virus' antibody-binding sites. (We note that our general type space and set-up may also be used to explicitly model the interplay between frequent epitope mutations and relatively rare non-epitope mutations affecting the fitness, see Strelkowa and
L\"assig \cite{StrelkowaLaessig12} for a discussion.)

On the level of the hosts a similar dynamic is at play. The infection of a new host (with a particular virus type) is affected by the (local) availability of hosts that are yet uninfected by a particular strain. Thus again new mutants may have an initial advantage that could depend on a strain's intrinsic fitness as well as on how different it is from previous strains.
We refer to Remark \ref{rem1}  (c) for a choice of mutation dependent offspring distributions that could be suitable in modelling the situation described above.

We note that there is an extensive research literature on analysing the spread of different virus strains (and their genealogies), see for example \cite{StrelkowaLaessig12, Desaietal2013, Neher2013, NeherHallatschek2013}  
and references therein. Mathematically rigorous results for models with fixed total population size and particular type spaces
can be found in recent work by Schweinsberg \cite{SchweinsbergI2017,SchweinsbergII2017}, see Dawson and Greven \cite{DawsonGreven2013}  for a general treatment of such models. 


Apart from an interpretation of the type space as a space of genetic or phenotypic traits we could also go back to the interpretation of the type space as a spatial location of individuals (or a combination of the two). In a spatial setting Fournier and M\'el\'eard \cite{Four2004} interpreted individuals as plants and the production of new individuals in the type space as a result of seed dispersal (with immediate maturation). But unlike in the model of \cite{Four2004}  seeds are not always dispersed individually but may be dispersed in groups, in particular when the seeds within fruit are consumed by animals and carried over larger distances, see \cite{CortesUriarte2012, Schuppetal2010} for some recent biological literature highlighting the importance of these dispersal mechanisms. How many of these seeds establish themselves at their new location may depend on the parent location, which can influence how many viable seeds were produced, as well as on the new location, which may be more or less favorable.

Finally, we point out that our model can easily be adjusted to also include mutation of individuals during their lifetime. In this case,
the "birth" of a (or multiple) individuals at a new location in the type space would happen at the same time as the "death" of the individual at the original location. (If we think again of a geographical space then this would be migration.) Details are left to the interested reader. 


As in previous work, the main goal of our work is to look for macroscopic approximations, namely for tractable large population approximations of the individual-based models 
when the size of the population tends to infinity, combined with frequent mutation and accelerated birth and death. The latter is known as allometric demographies or allometric effects (larger populations made up of smaller individuals who reproduce and die faster); see for example \cite[Section 4.2]{Cha2008} and reference therein for background. Basically, this leads to systems in which organisms have short lives and reproduce fast while their colonies or populations grow or decline on a slow timescale.
We proceed with tightness-uniqueness arguments inspired by the classical theory of superprocesses \cite{Dawson1991} and \cite{Li2011} without interaction. Clearly,  difficulties arise due to the lack of the branching property which is a consequence of the dependency created by the competition between individuals. Nevertheless, following ideas of M\'el\'eard \cite{Four2004} and Champagnat, et al. \cite{Cha2008} we introduce a new infinite dimensional martingale problem. In the limit, we obtain a measure-valued process defined as the solution of this nonlinear martingale problem. The proof of uniqueness of such a martingale problem requires substantial work. We develop a new Girsanov type theorem which allows us to get rid of the non-linearities caused by the competition. This Girsanov theorem may be viewed as a generalization of Dawson's Girsanov Theorem \cite{Da1978} and may also be of independent interest.  The effect of multiple branching makes the analysis more complicated due to the loss of some moments. Therefore, we adapt the localization procedure introduced by Stroock \cite{Stroock1975} and generalized by He \cite{He2009} to the measure-valued context.

It is important to point out that the nonlinear superprocess obtained as the limit generalizes, for instance, the work of \cite{Elk1991}, \cite{Fitz1992}, \cite{DawDoGo2002} or \cite{Perkins2002} by incorporating interaction. On the other hand, the general reproduction law of the approximating population system yields a limiting process with a general branching mechanism, which extends the models proposed by M\'el\'eard \cite{Four2004},  Champagnat, et al. \cite{Cha2008}, Jourdain, et al. \cite{Jou} and Etheridge \cite{Eth2004} to study spatially interactive structured populations. Let us remark that our model allows the description of massive reproduction events which translate into discontinuities of the limiting process. This can be the first step to analyzings superprocesses with interactions that possess a jump structure.

The outline of the remainder of this paper is as follows. Section \ref{model} is devoted to the introduction of the individual-based model we are interested in. Here, we also prove some useful properties of the model. The main
convergence result based on a large population limit is stated in Section \ref{mainresult}. In Section  \ref{profteo3}, we prove tightness of the laws of the particle processes and we identify the limiting values as solutions of a nonlinear martingale problem. The uniqueness of such a martingale problem is attended to in Sections \ref{Uni} and \ref{lastT}.

\section{The individual-based model} \label{model}

In this section, we formally introduce our interacting particle Markov process for Darwinian evolution in an asexual population with non-constant population size in which each individual
is characterized by hereditary types. Our model's construction starts with a microscopic description of a population in which the adaptive traits influence the birth rate, the mutation process, the death rate, and how the individuals interact with each other and their external environment. More precisely, we assume that the phenotype of each individual is described by a quantitative trait. Throughout the paper, we will assume that
\begin{center}
the trait space $\mathcal{X}$ is a Polish space that is locally compact.
\end{center}

In the following we should consistently refer to $x \in  \mathcal{X}$ as either a "trait" or a "type". We consider a parameter $K \in \mathbb{N}$ that scales the resources or area available. It is called the ``system size'' by Metz et al. \cite{Met1996}. It will become apparent later that this parameter is linked to the size of the population: large $K$ means a large population (provided that the initial condition is proportional to $K$). We have the following definition of the stochastic interacting individual system where individuals behave independently:
\begin{enumerate}
\item {\bf Birth and mutation:} An individual of trait type $x \in \mathcal{X}$ gives birth at rate $b_{K}(x) \in \mathbb{R}_{+}$. The number of offspring born at each birth time is controlled by a Markov kernel $\pi_K$ on $\mathcal{X}^2\times \mathbb{N}$ i.e. by a family of offspring distributions
    indexed by $\mathcal{X}\times\mathcal{X}$, say
    \begin{equation}
    \label{e:pi}
       \pi^{K}=\big(\pi^{K}(x,h)=(\pi^{K}(x,h,k),\,k\geq 1),\, x,h\in\mathcal{X}\big)
    \end{equation}
       such that  $(x,h) \mapsto \pi^{K}(x,h,\cdot)$ is measurable and $\sum_{k=1}^{\infty} \pi^{K}(x,h,k) = 1$ for all $x,h \in \mathcal{X}$. More precisely, each individual of type $x$ gives birth independently to $k$ clonal individuals with probability $\pi^{K}(x,x,k)(1-p(x))$, where $p(x) \in [0,1]$ is the mutation probability of an individual with trait $x \in \mathcal{X}$. Otherwise, it produces $k$ individuals of type $h$ with probability $p(x) \pi^{K}(x,h,k) m_{K}(x, {\rm d}h)$, where $m_{K}(x, \cdot)$ is a probability measure on $\mathcal{X}$ called the mutation kernel or mutation step law. Note here that the new type $h$ only depends on $x$ while the number of individuals produced depends on $x$ and $h$.

\item {\bf Natural death:} An individual of type $x \in \mathcal{X}$ dies naturally at rate $d_{K}(x) \in \mathbb{R}_{+}$.

\item {\bf Competition:} We let $c_{K}(x, y) \in \mathbb{R}_{+}$ be the competition kernel which models the competition pressure felt by an individual with trait $x \in \mathcal{X}$ from an individual with type $y \in \mathcal{X}$. We then add extra death due to competition. Specifically, each individual of type $y$ points independent exponential clocks of parameter $c_{K}(x, y)$ on each individual of type $x$. Then, the death of an individual of type $x$ occurs as soon as a clock pointed at this individual rings.
\end{enumerate}

Let $\mathcal{M}(\mathcal{X})$ denote the set of finite Borel measures on $\mathcal{X}$ equipped with the weak topology, and define the subset $\mathcal{M}^{K}(\mathcal{X})$ of $\mathcal{M}(\mathcal{X})$ by
\begin{eqnarray*}
\mathcal{M}^{K}(\mathcal{X}) = \left\{
\frac{1}{K} \sum_{i =1}^{n} \delta_{x_{i}}: n \geq 0, \, x_{1}, \dots, x_{n} \in \mathcal{X} \right \},
\end{eqnarray*}

\noindent where $\delta_{x}$ is the Dirac measure at $x$. For any $\mu \in \mathcal{M}(\mathcal{X})$ and any measurable
function $f$ on $\mathcal{X}$, we set
$\langle \mu, f  \rangle = \int_{\mathcal{X}} f {\rm d} \mu$. 

At any time $t \geq 0$, we let $N_{t}$ be the finite number of individuals alive, each of which is assigned a trait type in $\mathcal{X}$. Let us denote by $x_{1}, \dots, x_{N_{t}}$ the trait types of these individuals. The state of the population at time $t \geq 0$, rescaled by $K$, can be described by the finite point measure $\nu_{t}^{K}$ on $\mathcal{X}$
defined by
\begin{eqnarray*}
\nu_{t}^{K} = \frac{1}{K} \sum_{i=1}^{N_{t}} \delta_{x_{i}}.
\end{eqnarray*}

We let $\mathds{1}_{A}$ be the indicator function of a set $A \subset \mathcal{X}.$
For simplicity, we denote by $\mathbf{1}:=\mathds{1}_{\mathcal{X}}$ the indicator function on  the whole space. We observe that $\langle \nu_{t}^{K}, \mathbf{1}  \rangle = N_{t} K^{-1}$. For any $x \in \mathcal{X}$, the positive number  $\langle \nu_{t}^{K}, \mathds{1}_{\{ x\}}  \rangle$ is called the density of the trait $x$ at time $t$.

In the next section, we are going to  construct under suitable assumptions a
$\mathcal{M}^{K}(\mathcal{X})$-valued Markov process with
infinitesimal generator, $\mathscr{L}^{K}$, defined for a convergence determining subspace of bounded measurable functions $f$ from $\mathcal{M}^{K}(\mathcal{X})$ to $\mathbb{R}$ and for all $\mu^{K} \in \mathcal{M}^{K}(\mathcal{X})$ by
\begin{align} \label{eq1}
\mathscr{L}^{K} f(\mu^{K}) & =  K \int_{\mathcal{X}}  \sum_{k=1}^{\infty} b_{K}(x)(1-p(x)) \pi^{K}(x,x, k)
\left(  f \left(   \mu^{K} + k \frac{\delta_{x}}{K} \right) - f(\mu^{K}) \right) \mu^{K}({\rm d} x) \nonumber \\
& \hspace*{5mm} +  K  \int_{\mathcal{X}} \sum_{k=1}^{\infty} b_{K}(x) p(x)   \int_{\mathcal{X}} \pi^{K}(x,h,k) \left(  f \left(   \mu^{K} + k\frac{\delta_{h}}{K} \right) - f(\mu^{K}) \right) m_{K}(x, {\rm d} h) \mu^{K}({\rm d} x) \nonumber \\
& \hspace*{5mm} +  K \int_{\mathcal{X}}  \left( d_{K}(x) + K \int_{\mathcal{X}} c_{K}(x,y) \mu^{K}({\rm d} y) \right)  \left(  f \left(   \mu^{K} - \frac{\delta_{x}}{K} \right) - f(\mu^{K}) \right) \mu^{K}({\rm d} x).
\end{align}

\noindent The construction is inspired by \cite{Cha2008} and \cite{Four2004}, who consider the case of binary reproduction and an offspring distribution that is independent of the trait type. In this more general setting to the best of our knowledge, the construction has not been made rigorous before. Therefore, we present the details 
in order to make this work self-contained.

\begin{remark} In our model we assume that in case of mutation all offspring will have the same mutant trait. One could consider more general dynamics in which at a mutation event each new offspring could mutate into a different trait independently of its sibling. This clearly will make the model more realistic but mathematically more involved and complicated. Thus, we leave it as an open problem. On the other hand, we recall that we are primarily interested in studying the (potentially fast) rise in numbers of individuals of new traits, and this is what the proposed model is trying to capture.
\end{remark}

In the present paper we use the following notation. Given a topological space $V$, let $\mathbb{B}(V)$  denote the Borel $\sigma$-algebra on $V$. Let $W$ be another topological space with its respective $\sigma$-algebra $\mathbb{B}(W).$ Then we denote by  $B(V, W)$  the set of bounded measurable functions from $V$ to $W$. Let $T >0$ and $\mathbb{D}([0,T], V)$ (resp. $\mathbb{D}([0,\infty), V)$) denote the space of c\`adl\`ag paths from $[0,T]$ (resp. from $[0, \infty)$)  to $V$ furnished with the Skorokhod topology.
For a metric space $\bar{V}$ let $\mathcal{P}(\bar{V})$ be the family of Borel probability measures on $\bar{V}$ equipped with the Prohorov metric. Let $B(V, \mathbb{R})$ be furnished with the supremum norm (i.e. for $f \in B(V, \mathbb{R})$, we write $\parallel f \parallel_{\infty} = \sup_{x \in V} |f(x)|$)  and $B(V, \mathbb{R}_{+})$ denote the subset of $B(V, \mathbb{R})$ of positive elements. We use $C_{b}(V, \mathbb{R})$ (resp. $C_{b}(V, \mathbb{R}_{+})$) to denote the set of bounded continuous functions from $V$ to $\mathbb{R}$ (resp. from $V$ to $\mathbb{R}_{+}$). For any integer $n \geq 1$, let
$C_{b}^{n}(\mathbb{R}, \mathbb{R})$ (resp. $C_{b}^{n}(\mathbb{R}, \mathbb{R}_{+})$, $C_{b}^{n}(\mathbb{R}_{+}, \mathbb{R}_{+})$)  be the subset of $C_{b}(\mathbb{R}, \mathbb{R})$ (resp. $C_{b}(\mathbb{R}, \mathbb{R}_{+})$,  $C_{b}(\mathbb{R}_{+}, \mathbb{R}_{+})$) of functions with bounded continuous derivatives up to the $n$-th order. If $V$ is locally compact and separable, we write $C_{0}(V, \mathbb{R})$ for the space of continuous functions from $V$ to $\mathbb{R}$ which vanish at infinity (i.e.\ $f \in C_{0}(V, \mathbb{R})$ if for all $\varepsilon >0$ there exists $E \subset V$ compact such that for all $x \in V \setminus E$ one has that $|f(x)| \leq \varepsilon)$. Let $\hat{\mathcal{X}} = \mathcal{X} \cup \{ \partial \}$ be the one-point compactification of $\mathcal{X}$, with $\hat{\mathcal{X}} = \mathcal{X}$ whenever $\mathcal{X}$ is compact. Let $C_{\partial}(\mathcal{X}, \mathbb{R})$ (resp. $C_{\partial}(\mathcal{X}, \mathbb{R}_{+})$, $C_{\partial}(\mathcal{X} \times \mathcal{X}, \mathbb{R}_{+})$) be the set of functions in $C_{b}(\mathcal{X}, \mathbb{R})$ (resp. $C_{b}(\mathcal{X}, \mathbb{R}_{+})$, $C_{b}(\mathcal{X} \times \mathcal{X}, \mathbb{R}_{+})$) that can be extended continuously to $\hat{\mathcal{X}}$ (resp. $\hat{\mathcal{X}} \times \hat{\mathcal{X}}$). Furthermore, in the case that $\mathcal{X}$ is a subset of $\mathbb{R}^{l}$ ($l \geq 1$), we let $C_{\partial}^{n}(\mathcal{X}, \mathbb{R})$ (resp. $C_{\partial}^{n}(\mathcal{X}, \mathbb{R}_{+})$, $C_{\partial}^{n}(\mathcal{X} \times \mathcal{X}, \mathbb{R}_{+})$) be the set of functions in $C_{b}^{n}(\mathcal{X}, \mathbb{R})$ (resp. $C_{b}^{n}(\mathcal{X}, \mathbb{R}_{+})$, $C_{b}^{n}(\mathcal{X} \times \mathcal{X}, \mathbb{R}_{+})$) which together with their derivatives up to the $n$-th order can be extended continuously to $\hat{\mathcal{X}}$ (resp. $\hat{\mathcal{X}} \times \hat{\mathcal{X}}$). Finally, we use the superscript ``$+$'' to denote the subsets of non-negative elements bounded away from zero e.g., $B(V, \mathbb{R}_{+})^{+}$,  $C_{b}(V, \mathbb{R}_{+})^{+}$, etc. That is, for $f \in B(V, \mathbb{R}_{+})^{+}$ (or $C_{b}(V, \mathbb{R}_{+})^{+}$) there exists an $\varepsilon >0$ such that $f(x) \geq \varepsilon$ for all $x \in V$.

\subsection{Poissonian construction} \label{poissonc}

We provide a path-wise description of the stochastic process $(\nu^{K}_{t}, t \geq 0)$. For this we will use the following:

\begin{assumption} \label{assum1} Assumptions on the population parameters of the model:
\begin{itemize}
\item[(i)] The birth and natural death rate belong to $B(\mathcal{X}, \mathbb{R}_{+})$. So, there exist $0 < \overline{b}, \overline{d}< + \infty$ (that may depend on $K$) such that $b_{K}(\cdot) \leq \overline{b}$ and $d_{K}(\cdot) \leq \overline{d}$.

\item[(ii)] The competition kernel belongs to $B(\mathcal{X} \times \mathcal{X}, \mathbb{R}_{+})$. So, there exists $0 <  \overline{c} < + \infty$ (that may depend on $K$) such that $ c_{K}(\cdot, \cdot) \leq \overline{c}$.

\item[(iii)] The mutation kernel $m_{K}(x, {\rm d} h)$ is absolutely continuous with respect to a $\sigma$-finite probability measure $\bar{m}$ on $\mathcal{X}$ with density $m_{K}(x, h)$.
\end{itemize}
\end{assumption}

We need the following notation:

\begin{notation} \label{not1}
Fix $0 \in \mathcal{X}$ (an arbitrary element), let $\mathbf{H} = (H_{1}, H_{2}, \dots, H_{k}, \dots): \mathcal{M}^{K}(\mathcal{X}) \mapsto \mathcal{X}^{\mathbb{N}}$ be defined by
\begin{eqnarray*}
\mathbf{H}\left(\frac{1}{K} \sum_{i =1}^{n} \delta_{x_{i}} \right) = (x_{\theta(1)}, \dots, x_{\theta(n)}, 0, \dots, 0), \hspace*{5mm} \text{for} \hspace*{5mm} n \geq 1,
\end{eqnarray*}

\noindent where $x_{\theta(1)} \preceq \dots \preceq x_{\theta(n)}$ for some arbitrary (but fixed) order $\preceq$ on $\mathcal{X}$. 
\end{notation}

To avoid confusion, it is important to notice that the function $\mathbf{H}$ is listing all the $x_{i}$'s in some order and that there may be  
repetitions. The function $\mathbf{H}$ allows us to label the individuals in a population described by a measure in $\mathcal{M}^{K}(\mathcal{X})$ in an arbitrary way (here depending on their types). The vector that is given by $\mathbf{H}$ will be useful later on when we want to attach Poisson processes to all individuals and want them to interact (at the jump times of these Poisson processes) with the rest of the population according to their trait type.

\begin{notation}
We consider the space
$\mathcal{C}^{K} \subseteq \mathbb{D}([0, \infty), \mathcal{M}(\mathcal{X}))$ of piecewise constant cadlag paths, i.e.
\begin{eqnarray*}
    \mathcal{C}^{K} := \left\{ (\nu_{t}, t \geq 0) \, \Big| \, \mycom{\forall \, t \geq 0, \, \nu_{t} \in \mathcal{M}^{K}(\mathcal{X}), \, \, \text{and} \, \, \exists \, 0 = t_{0} < t_{1} < t_{2} < \cdots,}{\lim_{n \rightarrow \infty} t_{n}= \infty \, \, \text{with} \, \, \nu_{t} = \nu_{t_{i}} \, \forall \, t \in [t_{i}, t_{i+1} ) }  \right \}.
\end{eqnarray*}
\end{notation}

Observe that for $(\nu_{t}, t \geq 0) \in \mathcal{C}^{K}$, and $t >0$ we can define $\nu_{t-}$ in the following way: If $ t \not \in \bigcup_{i\geq 0} \{ t_{i} \}$, $\nu_{t-} = \nu_{t}$, while if $t=t_{i}$, for some $i \geq 1$, $\nu_{t-}=\nu_{t_{i-1}}$.

We now introduce some Poisson point processes that we need. {We will write $\lambda$ for the Lebesgue measure on $\mathbb{R}_+$ and $n$ for the counting measure on $\mathbb{N}$.}

\begin{definition}
Let $(\Omega, \mathcal{F}, \mathbb{P})$ be a (sufficiently large) probability space. On this space, we consider the following four independent random elements:
\begin{enumerate}
\item {\bf Initial distribution:} Let $\nu_{0}^{K}$ be a $\mathcal{M}^{K}(\mathcal{X})$-valued random variable.

\item {\bf Clonal birth:} Let $N_{\text{c}}$
be a Poisson point measure on
$\mathbb{R}_{+}\times \mathbb{N}  \times \mathbb{N} \times \mathbb{R}_{+}$, \\with intensity measure
$\lambda\otimes n\otimes n\otimes\lambda$.

\item {\bf Mutation:} Let $N_{\text{m}}$
be a Poisson point measure on
$\mathbb{R}_{+} \times \mathbb{N}  \times \mathcal{X} \times \mathbb{N} \times \mathbb{R}_{+}$, \\with intensity measure
$\lambda\otimes n\otimes\bar{m}\otimes n\otimes\lambda$.

\item {\bf Natural death and competition:} Let $N_{\text{d}}$
be a Poisson point measure on
$\mathbb{R}_{+} \times \mathbb{N} \times \mathbb{R}_{+}$, \\ with intensity measure
$\lambda\otimes n\otimes\lambda$.
\end{enumerate}

Let us denote by $(\mathcal{F}_{t})_{t \geq 0}$ the canonical filtration generated by these processes.
\end{definition}

Finally, we define the population process in terms of the previous Poisson measures.

\begin{definition} \label{def1}
An $(\mathcal{F}_{t})_{t \geq 0}$-adapted stochastic process $\nu^{K} = (\nu_{t}^{K}, t \geq 0)$ that belongs a.s. to $\mathcal{C}^{K}$ will be called the population process if a.s., for all $t \geq 0$,
\begin{align*}
\nu_{t}^{K} & = \nu_{0}^{K} + \int_{[0,t]} \int_{\mathbb{N}} \int_{\mathbb{N}} \int_{ \mathbb{R}_{+}} k \frac{\delta_{H_{i}(\nu^{K}_{s-})}}{K}  \mathds{1}_{\{ i \leq K \langle \nu_{s-}^{K}, \mathbf{1}\rangle \}} \\
& \hspace*{20mm} \mathds{1}_{ \{z \leq b_{K}(H_{i}(\nu_{s-}^{K})) (1-p(H_{i}(\nu_{s-}^{K})))\pi^{K}(H_{i}(\nu_{s-}^{K}),H_{i}(\nu_{s-}^{K}) , k)  \}} N_{{\rm c}}({\rm d}s, {\rm d}i, {\rm d}k, {\rm d}z )
\\
& \hspace*{5mm}  + \int_{[0,t]} \int_{\mathbb{N}} \int_{\mathcal{X} \times \mathbb{N}} \int_{ \mathbb{R}_{+}} k \frac{\delta_{ h}}{K} \mathds{1}_{\{ i \leq K \langle \nu_{s-}^{K}, \mathbf{1}\rangle \}} \\
& \hspace*{20mm} \mathds{1}_{ \{z \leq b_{K}(H_{i}(\nu_{s-}^{K})) p(H_{i}(\nu_{s-}^{K})) \pi^{K}(H_{i}(\nu_{s-}^{K}), h, k)m_{K}(H_{i}(\nu_{s-}^{K}),h)\}} N_{{\rm m}}({\rm d}s, {\rm d}i, {\rm d}h, {\rm d} k, {\rm d}z )
\\
& \hspace*{5mm} - \int_{[0,t]} \int_{\mathbb{N}} \int_{ \mathbb{R}_{+}} \frac{\delta_{H_{i}(\nu^{K}_{s-})}}{K} \mathds{1}_{\{ i \leq K \langle \nu_{s-}^{K}, \mathbf{1}\rangle \}} \\
& \hspace*{20mm} \mathds{1}_{ \{z \leq  d_{K}(H_{i}(\nu_{s-}^{K})) + K\int_{\mathcal{X}} c_{K}(H_{i}(\nu_{s-}^{K}),y) \nu_{s-}^{K}({\rm d} y) \}} N_{{\rm d}}({\rm d}s, {\rm d}i, {\rm d}z ).
\end{align*}
\end{definition}
\noindent
In order to show existence and some moment properties for the population process in Definition \ref{def1} we need another assumption.
\begin{assumption} \label{assum2}
We consider the following moment conditions:
\begin{itemize}
\item[(i)] The offspring distribution $\pi^{K}$ has finite mean, namely
\begin{eqnarray*}
\kappa_{1} = \sup_{x,h \in \mathcal{X}} \sum_{k = 1}^{\infty} k\pi^{K}(x,h,k) < + \infty.
\end{eqnarray*}

\item[(ii)] The measure $\nu_{0}^{K}$ has finite mean, that is
\begin{eqnarray*}
 \mathbb{E}[\langle \nu_{0}^{K}, \mathbf{1}  \rangle  ] < + \infty.
\end{eqnarray*}
\end{itemize}
\end{assumption}

We now show that under Assumptions \ref{assum1} and \ref{assum2} the stochastic process $\nu^{K} = (\nu_{t}^{K}, t \geq 0)$ from Definition \ref{def1} is well-defined. Observe that the total jump rate of $\nu^{K}_{t}$ is bounded by a polynomial in the total mass at time $t \geq 0$ by Assumption \ref{assum1}. Therefore, the process is well-defined on the interval $[0, \tau_{n}]$, where for $n \geq 1$
\begin{equation}
\label{e:taun}
   \tau_{n} := \inf \{t \geq 0: \langle \nu_{t}^{K}, \mathbf{1} \rangle  \geq n/K \}.
\end{equation}
Moreover, the process is shown to be well-defined if we can exclude explosion of the total mass.
Thus the goal then is to show that $\tau_{n} \rightarrow \infty$ almost surely as $n \rightarrow \infty$.

\begin{theorem} \label{teo1}
Suppose that Assumptions \ref{assum1} and \ref{assum2} are fulfilled. Then the following hold:
\begin{itemize}
\item[(a)] The stochastic process $\nu^{K} = (\nu_{t}^{K}, t \geq 0)$ from Definition \ref{def1} is well-defined and it is not explosive.
\item[(b)] Moreover, assume that for some $q \geq 1$, 
\begin{eqnarray*}
\kappa_{q} = \sup_{x,h \in \mathcal{X}} \sum_{k = 1}^{\infty} k^{q}\pi^{K}(x,h,k) < + \infty \hspace*{5mm} \text{and} \hspace*{5mm} \mathbb{E}[\langle \nu_{0}^{K}, \mathbf{1}  \rangle^{q}  ] < + \infty.
\end{eqnarray*}
Then for any $0 < T < + \infty$,
\begin{eqnarray} \label{eq3}
\mathbb{E} \big[  \sup_{t \in [0,T]} \langle \nu_{t}^{K}, \mathbf{1} \rangle^{q}\big]  < + \infty.
\end{eqnarray}
\end{itemize}
\end{theorem}

\begin{proof}
Claim (a) is a consequence of claim (b). Indeed, we can build the solution $\nu^{K} = (\nu^{K}_{t}, t \geq 0)$ step by step using Definition \ref{def1} (see for instance \cite[Section 2.3]{Four2004} in a similar setting). We have to check only that the sequence of (effective or fictitious) jump instants $(T_{n}, n \geq 0)$ goes a.s. to infinity as $n \rightarrow \infty$ (i.e. there is no explosion in finite time), and this follows from (b) with $q = 1$ due to the uniform (in $\mathcal{X}$) boundedness of the rates by Assumption \ref{assum1}.

We now prove (b). Recall $\tau_n$ from (\ref{e:taun}).
Then, a simple computation using Assumption \ref{assum1} shows that, dropping the non-positive death terms on Definition \ref{def1} yields
\begin{align*}
\sup_{s \in [0, t \wedge \tau_{n}]} \langle \nu_{s}^{K}, \mathbf{1} \rangle^{q} & \leq \langle \nu_{0}^{K}, \mathbf{1} \rangle^{q} + \int_{[0,t \wedge \tau_{n}]} \int_{\mathbb{N}} \int_{\mathbb{N}} \int_{ \mathbb{R}_{+}} \left( \left(\langle \nu_{s-}^{K}, \mathbf{1} \rangle + \frac{k}{K} \right)^{q} - \langle \nu_{s-}^{K}, \mathbf{1} \rangle^{q} \right)  \mathds{1}_{\{ i \leq K \langle \nu_{s-}^{K}, \mathbf{1}\rangle \}} \\
& \hspace*{20mm} \mathds{1}_{ \{z \leq b_{K}(H_{i}(\nu_{s-}^{K})) (1-p(H_{i}(\nu_{s-}^{K})))\pi^{K}(H_{i}(\nu_{s-}^{K}), H_{i}(\nu_{s-}^{K}), k)  \}} N_{{\rm c}}({\rm d}s, {\rm d}i, {\rm d}k, {\rm d}z )
\\
& \hspace*{5mm}  + \int_{[0,t \wedge \tau_{n}]} \int_{\mathbb{N}} \int_{\mathcal{X} \times \mathbb{N}} \int_{ \mathbb{R}_{+}} \left( \left(\langle \nu_{s-}^{K}, \mathbf{1} \rangle + \frac{k}{K} \right)^{q} - \langle \nu_{s-}^{K}, \mathbf{1} \rangle^{q} \right) \mathds{1}_{\{ i \leq K \langle \nu_{s-}^{K}, \mathbf{1}\rangle \}} \\
& \hspace*{20mm} \mathds{1}_{ \{z \leq b_{K}(H_{i}(\nu_{s-}^{K})) p(H_{i}(\nu_{s-}^{K})) \pi^{K}(H_{i}(\nu_{s-}^{K}), h, k)m_{K}(H_{i}(\nu_{s-}^{K}),h)\}} N_{{\rm m}}({\rm d}s, {\rm d}i, {\rm d}h, {\rm d} k, {\rm d}z ).
\end{align*}

By taking expectations and recalling Assumption \ref{assum1}, we obtain that
\begin{align*}
\mathbb{E} \big [ \sup_{s \in [0, t \wedge \tau_{n}]} \langle \nu_{s}^{K}, \mathbf{1} \rangle^{q} \big] & \leq \mathbb{E} \left [ \langle \nu_{0}^{K}, \mathbf{1} \rangle^{q} \right] \\
& \hspace*{3mm}+ K \overline{b} \, \mathbb{E} \Bigg [ \int_{0}^{t \wedge \tau_{n}}  \int_{ \mathcal{X}} \bigg( \sum_{k=1}^{\infty} \pi^{K}(x,x,k) \left( \left(\langle \nu_{s}^{K}, \mathbf{1} \rangle + \frac{k}{K} \right)^{q} - \langle \nu_{s}^{K}, \mathbf{1} \rangle^{q} \right)
\\
& \hspace*{3mm}+ \sum_{k=1}^{\infty} \int_{ \mathcal{X}} \pi^{K}(x,h,k) m_{K}(x,h) \left( \left(\langle \nu_{s}^{K}, \mathbf{1} \rangle + \frac{k}{K} \right)^{q} - \langle \nu_{s}^{K}, \mathbf{1} \rangle^{q} \right) \bar{m}({\rm d}h)
\bigg) \nu_{s}^{K}({\rm d}x ) {\rm d}s \Bigg].
\end{align*}

Next, we recall the inequality
\begin{eqnarray*}
(x+k)^{q}-x^{q} \leq  C_{q}(k^{q} + k^{q-1} x^{q-1}), \hspace*{5mm} \text{for} \hspace*{5mm}  k,x \in \mathbb{N} \cup \{0\},
\end{eqnarray*}

\noindent and some positive constant $C_{q}$ depending only on $q$. We thus obtain
\begin{eqnarray*}
\mathbb{E} \big[ \sup_{s \in [0, t \wedge \tau_{n}]} \langle \nu_{s}^{K}, \mathbf{1} \rangle^{q} \big] & \leq & \mathbb{E}[ \langle \nu_{0}^{K}, \mathbf{1} \rangle^{q}] +  2  C_{q} \, \overline{b} \, \max\left( 1, K^{1-q} \right) \mathbb{E} \left[ \int_{0}^{t \wedge \tau_{n}} (\kappa_{q} \langle \nu_{s-}^{K}, \mathbf{1} \rangle + \kappa_{q-1} \langle \nu_{s-}^{K}, \mathbf{1} \rangle^{q}) {\rm d} s \right] \\
& \leq & C_{q,K} \left( 1+
\mathbb{E} \left[  \int_{0}^{t} ( 1 + \langle \nu_{s \wedge \tau_{n}}^{K}, \mathbf{1} \rangle^{q}) {\rm d} s \right] 
 \right),
\end{eqnarray*}

\noindent where $C_{q, K}$ is a positive constant depending only on $q$ and $K$ (for the last inequality, we used that $x \leq 1 + x^{q}$, for $x \geq 0$). The Gronwall Lemma allows us to conclude that for any $T < \infty$, there exists a constant $C_{q, T}$ (not depending on $n$) such that
\begin{eqnarray} \label{eq2}
\mathbb{E} \big[ \sup_{s \in [0, T \wedge \tau_{n}]} \langle \nu_{s}^{K}, \mathbf{1} \rangle^{q} \big] & \leq & C_{q, T}.
\end{eqnarray}

Finally, we only need to deduce that $\tau_{n}$ tends a.s. to infinity in order to finish the proof. Indeed, if not, we may find a $T_{0} < \infty$ such that $\epsilon_{T_{0}} = \mathbb{P}(\sup_{n \geq 1} \tau_{n} < T_{0}) > 0$. This would imply that
$\mathbb{E}[\sup_{t \in [0, T_{0} \wedge \tau_{n}]} \langle \nu_{t}^{K}, \mathbf{1} \rangle^{q}] \geq \epsilon_{T_{0}} (n/K)^{q}$ for all $n$ which contradicts our last inequality. Therefore, we may let $n \rightarrow \infty$ in (\ref{eq2}) thanks to Fatou's Lemma and get (\ref{eq3}).
\end{proof}

Let us now show that if the population parameters satisfy Assumptions \ref{assum1} and \ref{assum2} and $\nu^{K} = (\nu^{K}_{t}, t \geq 0)$ solves the stochastic equation in Definition \ref{def1}, then it follows the dynamic we are interested in, i.e. that it has infinitesimal generator given by (\ref{eq1}).

Let  $Z = (Z_{t}, t \geq 0)$ be a Markov process $Z = (Z_{t}, t \geq 0)$ with state space $(E,d)$ defined on some probability space $(\tilde{\Omega}, \tilde{\mathcal{F}}, \tilde{\mathbb{P}})$. Recall that the infinitesimal generator $\mathscr{G}$ of a Markov process $Z$ is given by
\begin{eqnarray*}
\mathscr{G}(f(z)) := \frac{{\rm d}}{{\rm d} t} \tilde{\mathbb{E}}[f(Z_{t})] \Big|_{t=0} = \lim_{t \rightarrow 0} \frac{\tilde{\mathbb{E}}[f(Z_{t})] - f(z)}{t}, 
\end{eqnarray*} 
\noindent for $z \in E$ and $f \in B(E, \mathbb{R})$ for which this limit exists, see \cite[equation (1.10) in Section 1.1]{Ethier1986}.

\begin{proposition} \label{Prop1}
Assume Assumption \ref{assum1} and  \ref{assum2} and consider $\nu^{K} = (\nu_{t}^{K}, t \geq 0)$ as in Definition \ref{def1}. 
Then $\nu^{K}$ is a Markov process. Its infinitesimal generator $\mathscr{L}^{K}$ is given by (\ref{eq1}), and it is defined for all functions $f \in B(\mathcal{M}^{K}(\mathcal{X}), \mathbb{R})$ such that for $u \in [0,1]$ and $\mu \in \mathcal{M}^{K}(\mathcal{X})$
\begin{eqnarray} \label{eq8}
\int_{\mathcal{X}}|f(\mu - u \delta_{x}) - f(\mu)| \mu({\rm d} x) < C u,
\end{eqnarray}

\noindent where $C$ is a positive constant that does not depend on $\mu$. In particular, the law of $\nu^{K}$ does not depend on the chosen order (see Notation \ref{not1}).
\end{proposition}

\begin{proof}
The fact that $\nu^{K} = (\nu_{t}^{K}, t \geq 0)$ is a Markov process follows from its definition by classical results from the theory of Poisson random measures, see \cite[Section VI.6 and IX.3]{Erhan2011} for background and examples. Let us now prove that the infinitesimal generator of the process $\nu^{K}$ has the desired form. Consider a function $f$ as in the statement and recall that in our notation $\nu_{0}^{K} = \frac{1}{K} \sum_{i = 1}^{K \langle \nu_{0}^{K}, \mathbf{1} \rangle} \delta_{H_{i}(\nu_{0}^{K})}$. We notice that a.s.
\begin{eqnarray*}
f(\nu_{t}^{K}) = f(\nu_{0}^{K}) + \sum_{s \leq t} (f(\nu_{s-}^{K} + (\nu_{s}^{K} - \nu_{s-}^{K})) - f(\nu_{s-}^{K}) ),
\end{eqnarray*}

\noindent for $t \geq 0$. Then,
\begin{align*}
f(\nu_{t}^{K}) & = f(\nu_{0}^{K}) + \int_{[0,t]} \int_{\mathbb{N}} \int_{\mathbb{N}} \int_{ \mathbb{R}_{+}}
\left(f \left(\nu_{s-}^{K} + k \frac{\delta_{H_{i}(\nu_{s-}^{K})}}{K} \right) - f(\nu_{
s-}^{K}) \right) \mathds{1}_{\{ i \leq K \langle \nu_{s-}^{K}, \mathbf{1}\rangle \}} \\
& \hspace*{25mm} \mathds{1}_{ \{z \leq b_{K}(H_{i}(\nu_{s-}^{K})) (1-p(H_{i}(\nu_{s-}^{K})))\pi^{K}(H_{i}(\nu_{s-}^{K}), H_{i}(\nu_{s-}^{K}), k)  \}} N_{{\rm c}}({\rm d}s, {\rm d}i, {\rm d}k, {\rm d}z )
\\
& \hspace*{5mm}  + \int_{[0,t]} \int_{\mathbb{N}} \int_{\mathcal{X} \times \mathbb{N}} \int_{ \mathbb{R}_{+}}
\left(f \left(\nu_{s-}^{K} + k \frac{\delta_{ h}}{K} \right) - f(\nu_{
s-}^{K}) \right)
\mathds{1}_{\{ i \leq K \langle \nu_{s-}^{K}, \mathbf{1}\rangle \}} \\
& \hspace*{25mm} \mathds{1}_{ \{z \leq b_{K}(H_{i}(\nu_{s-}^{K})) p(H_{i}(\nu_{s-}^{K})) \pi^{K}(H_{i}(\nu_{s-}^{K}), h, k)m_{K}(H_{i}(\nu_{s-}^{K}),h)\}} N_{{\rm m}}({\rm d}s, {\rm d}i, {\rm d}h, {\rm d} k, {\rm d}z )
\\
& \hspace*{5mm} + \int_{[0,t]} \int_{\mathbb{N}} \int_{ \mathbb{R}_{+}}
 \left(f \left(\nu_{s-}^{K} - \frac{\delta_{H_{i}(\nu_{s-}^{K})}}{K} \right) - f(\nu_{
s-}^{K}) \right)
  \mathds{1}_{\{ i \leq K \langle \nu_{s-}^{K}, \mathbf{1}\rangle \}} \\
& \hspace*{25mm} \mathds{1}_{ \{z \leq d_{K}(H_{i}(\nu_{s-}^{K})) + K\int_{\mathcal{X}} c_{K}(H_{i}(\nu_{s-}^{K}),y) \nu_{s-}^{K}({\rm d} y) \}} N_{{\rm d}}({\rm d}s, {\rm d}i, {\rm d}s ).
\end{align*}

\noindent Taking expectations, we obtain that
\begin{align*}
\mathbb{E}[f(\nu_{t}^{K})] &= \mathbb{E}[f(\nu_{0}^{K})]  \\
& \hspace*{5mm} +
 \int_{0}^{t}
\mathbb{E}\left[ \sum_{i =1}^{K \langle \nu_{s}^{K}, \mathbf{1} \rangle } \left \{ \sum_{k=1}^{\infty}
 b_{K}(H_{i}(\nu_{s-}^{K})) (1-p(H_{i}(\nu_{s-}^{K}))) \pi^{K}(H_{i}(\nu_{s-}^{K}), H_{i}(\nu_{s-}^{K}), k) \right. \right. \\
 & \hspace*{25mm} \times \left(f \left(\nu_{s-}^{K} + k \frac{\delta_{H_{i}(\nu_{s-}^{K})}}{K} \right) - f(\nu_{
s-}^{K}) \right)
\\
& \hspace*{5mm}  \phantom{AAA}+
\sum_{k =1}^{\infty}b_{K}(H_{i}(\nu_{s-}^{K})) p(H_{i}(\nu_{s-}^{K})) \int_{\mathcal{X}} \pi^{K}(H_{i}(\nu_{s-}^{K}), h, k)m_{K}(H_{i}(\nu_{s-}^{K}),h)  \\
& \hspace*{25mm} \times  \left(f \left(\nu_{s-}^{K} + k \frac{\delta_{ h}}{K} \right) - f(\nu_{
s-}^{K}) \right) \bar{m}({\rm d}h)
\\
& \hspace*{5mm} \left. \left. \phantom{AAA}+
 \left( d_{K}(H_{i}(\nu_{s-}^{K})) + K\int_{\mathcal{X}} c_{K}(H_{i}(\nu_{s-}^{K}),y) \nu_{s-}^{K}({\rm d} y) \right) \left(f \left(\nu_{s-}^{K} - \frac{\delta_{H_{i}(\nu_{s-}^{K})}}{K} \right) - f(\nu_{
s-}^{K}) \right) \right \} \right ] {\rm d}s.
\end{align*}

Recalling Notation \ref{not1} and that we are integrating with respect the Lebesgue measure, we have that
\begin{align*}
\mathbb{E}[f(\nu_{t}^{K})] & = \mathbb{E}[f(\nu_{0}^{K})]  \\
& \hspace*{5mm} +
 \int_{0}^{t}
\mathbb{E}\bigg[ K \int_{\mathcal{X}}  \sum_{k=1}^{\infty}
 b_{K}(x) (1-p(x)) \pi^{K}(x, x, k)    \left(f \left(\nu_{s}^{K} + k \frac{\delta_{x}}{K} \right) - f(\nu_{
s}^{K}) \right) \nu_{s}^{K}({\rm d} x) 
\\
& \hspace*{5mm}  +
 K \int_{\mathcal{X}} \sum_{k =1}^{\infty}b_{K}(x) p(x) \int_{\mathcal{X}} \pi^{K}(x, h, k)m_{K}(x,h)  \left(f \left(\nu_{s}^{K} + k \frac{\delta_{ h}}{K} \right) - f(\nu_{
s}^{K}) \right) \bar{m}({\rm d}h) \nu_{s}^{K}({\rm d} x)
\\
& \hspace*{5mm}  +
 K \int_{\mathcal{X}} \left( d_{K}(x) + K\int_{\mathcal{X}} c_{K}(x,y) \nu_{s}^{K}({\rm d} y) \right) \left(f \left(\nu_{s}^{K} - \frac{\delta_{x}}{K} \right) - f(\nu_{s}^{K}) \right) \nu_{s}^{K}({\rm d} x)  \bigg]{\rm d}s.
\end{align*}

Since $\mathscr{L}^{K}(f(\nu_{0}^{K})) =\frac{{\rm d}}{{\rm d} t} \mathbb{E}[f(\nu_{t}^{K})] \big|_{t=0}$, Assumption \ref{assum1} as well as the conditions  (\ref{eq3}) for $q=1$, which is a consequence of Assumption \ref{assum2},  and (\ref{eq8})  lead to (\ref{eq1}) by differentiating the previous expression. Moreover, it should be now clear that the law of $\nu^{K}$ does not depend on the chosen order.
\end{proof}

\begin{remark} \label{remfunc}
We point out that any function $f \in B(\mathcal{M}(\mathcal{X}), \mathbb{R})$ of the form
\begin{eqnarray*}
f(\mu) = \sum_{i=1}^{n} \theta_{i} e^{- \lambda_{i} \langle \mu, \phi_{i} \rangle}, \hspace*{5mm} n \in \mathbb{N},
\end{eqnarray*}

\noindent where $\mu \in \mathcal{M}(\mathcal{X})$, $\lambda_{i}  \in \mathbb{R}_{+}$, $\theta_{i} \in \mathbb{R}$ and $\phi_{i} \in B(\mathcal{X}, \mathbb{R}_{+})$ for $i = 1, \dots, n$, satisfies condition (\ref{eq8}). More precisely, for $u \in [0,1]$,
\begin{eqnarray} \label{keyineq}
\int_{\mathcal{X}}|f(\mu - u \delta_{x}) - f(\mu)| \mu({\rm d} x) & \leq & \sum_{i=1}^{n} |\theta_{i}| e^{-\lambda_{i}\langle \mu, \phi_{i} \rangle } \int_{\mathcal{X}} |e^{u\lambda \phi_{i}(x)}-1| \mu({\rm d}x ) \nonumber \\
& \leq &  u \sum_{i=1}^{n} C_{i}|\theta_{i}| \lambda_{i} \langle \mu, \phi_{i} \rangle e^{-\lambda_{i}\langle \mu, \phi_{i} \rangle } \nonumber \\
& \leq & u \sum_{i=1}^{n} C_{i}|\theta_{i}| \lambda_{i},
\end{eqnarray}

\noindent for some $C_{i} > 0$. We have used the inequality $|e^{x}-1|\leq |x|e^{|x|}$, for $x \in \mathbb{R}$, in order to obtain the second line. This class of functions clearly separates points in $\mathcal{M}(\mathcal{X})$.  It is not clear whether they are convergence determining or not. However, this will not be needed later on.

On the other hand, we point out that the class of functions which satisfy (\ref{eq8}) contains the functions $f  \in B(\mathcal{M}(\mathcal{X}), \mathbb{R})$ of the form
\begin{equation*}
f(\mu)= 
              e^{-\lambda \mu(\mathcal{X})} \int_{\mathcal{X}^{n}} \phi(x_{1}, \dots, x_{n}) \mu({\rm d}x_{1}) \cdots  \mu({\rm d}x_{n}), 
\end{equation*}
\noindent with $\lambda \in \mathbb{R}_{+}$, $n \in \mathbb{N} \cup \{0\}$, $\phi \in C_{b}(\mathcal{X}^{n}, \mathbb{R})$, and where $\mathcal{X}^{n}$ denotes the $n$-fold space of $\mathcal{X}$. These functions clearly separate points in $\mathcal{M}(\mathcal{X})$ and moreover, they are also convergence determining (observe that the class is closed under multiplication and apply, for example, \cite[Theorem 2.7]{Wolf2013}).
\end{remark}

\subsection{Martingale properties}

We finally give some martingale properties of the process $\nu^{K} = (\nu_{t}^{K}, t \geq 0)$ , which are the key point of our approach. Recall that $\pi^{K}=(\pi^{K}(x,h) = (\pi^{K}(x,h,k), k \geq 1), x,h \in \mathcal{X})$ is the offspring distribution associated to the model described in Section \ref{model}. Let
$g^{K}(x, h, \cdot)$ be the associated probability generating function for $x, h \in \mathcal{X}$, that is,
\begin{eqnarray} \label{eq22}
g^{K}(x,h,z) = \sum_{k =1}^{\infty} \pi^{K}(x,h,k) z^{k}, \hspace*{5mm} |z| \leq 1.
\end{eqnarray}
We consider the mean value of the offspring distribution $\pi^{K}$,
\begin{eqnarray} \label{eq25}
\kappa^{K}(x,h) = \sum_{k =1}^{\infty} k \pi^{K}(x,h, k), \hspace*{6mm} \text{for} \, \, x,h \in \mathcal{X}.
\end{eqnarray}

\begin{theorem} \label{teo2}
Suppose that Assumption \ref{assum1} and \ref{assum2} are fulfilled.
\begin{itemize}
\item[(a)] For all functions $f \in  B(\mathcal{M}^{K}(\mathcal{X}), \mathbb{R})$ that satisfy (\ref{eq8}) and such that for some constant $C \geq 0$ (that may depend on $K$) and $\nu \in \mathcal{M}^{K}(\mathcal{X})$ we have $|f(\nu)| + |\mathscr{L}^{K}(f(\nu))| \leq C (1+\langle \nu, \mathbf{1} \rangle)$, the process  $M^{K}(f) = (M_{t}^{K}(f),t \geq 0)$ given by
\begin{eqnarray*}
M^{K}_{t}(f) = f(\nu_{t}^{K}) - f(\nu_{0}^{K}) - \int_{0}^{t} \mathscr{L}^{K}(f(\nu_{s}^{K})) {\rm d} s
\end{eqnarray*}

is a c\`adl\`ag $(\mathcal{F}_{t})_{t \geq 0}$-martingale starting from $0$.

\item[(b)] For any function $\phi \in C_{b}(\mathcal{X}, \mathbb{R}_{+})$, $E^{K}(\phi) = (E_{t}^{K}(\phi),t \geq 0)$ given by
\begin{align} \label{eq30}
 E^{K}_{t}( \phi) & = \exp(-  \langle \nu_{t}^{K}, \phi \rangle) - \exp(-  \langle \nu_{0}^{K},  \phi \rangle) \nonumber
\\
& \hspace*{3mm} - K\int_{0}^{t} \int_{\hat{\mathcal{X}}}
  \left( b_{K}(x) \left( g^{K}\left(x,x, e^{-\frac{ \phi(x)}{K}} \right) -1 \right)
+   d_{K}(x)\left( e^{\frac{\phi(x)}{K}} -1 \right) \right)
  \exp(-  \langle \nu_{s}^{K},  \phi \rangle) \nu_{s}^{K}({\rm d} x)  {\rm d} s \nonumber  \\
& \hspace*{3mm}  +
K \int_{0}^{t} \int_{\hat{\mathcal{X}}}  b_{K}(x) p(x) \Bigg( \int_{\hat{\mathcal{X}}} \bigg( g^{K} \left(x,x, e^{-\frac{\phi(x)}{K}} \right) \nonumber \\
 & \hspace*{50mm} - g^{K} \left(x,h, e^{-\frac{\phi(h)}{K}} \right) \bigg) m_{K}(x,h) \bar{m}({\rm d}h)  \Bigg) \exp(- \langle \nu_{s}^{K},  \phi \rangle)   \nu_{s}^{K}({\rm d} x) {\rm d} s \nonumber
\\
& \hspace*{3mm} -
 K^{2} \int_{0}^{t} \int_{\hat{\mathcal{X}}}  \left( \int_{\hat{\mathcal{X}}} c_{K}(x,y) \nu_{s}^{K}({\rm d} y) \right) \left( e^{\frac{\phi(x)}{K}} -1 \right) \exp(-  \langle \nu_{s}^{K},  \phi \rangle) \nu_{s}^{K}({\rm d} x)  {\rm d} s,
\end{align}
is a c\`adl\`ag $(\mathcal{F}_{t})_{t \geq 0}$-martingale starting from $0$.
\end{itemize}
\end{theorem}

\begin{proof}
Point (a) follows from \cite[Theorem I.51]{Pro2005} by showing
\begin{eqnarray*}
\mathbb{E} \big[  \sup_{s \in [0, t]} \left |M_{s}^{K}(f) \right | \big]   <  \infty, \hspace*{3mm} \text{for every} \hspace*{2mm} t \geq 0.
\end{eqnarray*}

\noindent This is a consequence of the assumption on $f$, Theorem \ref{teo1} and Proposition \ref{Prop1}. The point (b) is a consequence of (a) with $f(\nu) = \exp(-  \langle  \nu, \phi \rangle)$ with $\nu \in \mathcal{M}^{K}(\mathcal{X})$.
\end{proof}

\section{The superprocess limit} \label{mainresult}

In this section, we investigate the limit, when the system size $K$ increases to $+ \infty$, of the interactive particle system described in Section \ref{model}. As we will see, this will lead to a random measure-valued process. In an obvious way, we regard the previous interactive particle system as a process with state space $\mathcal{M}^{K}(\hat{\mathcal{X}}) \subset \mathcal{M}(\hat{\mathcal{X}})$. We denote by $g^{K}(x,h, \cdot)$ and $\kappa^{K}(x,h)$ the probability generating function and mean, of the offspring distribution $\pi^K(x,h)$, for $x,h \in \hat{\mathcal{X}}$, defined as in (\ref{eq22}) and (\ref{eq25}), respectively. Similarly, we write $b_{K}(x)$ and $d_{K}(x)$, for $x \in \hat{\mathcal{X}}$. For $\alpha \in (0,1]$, we also consider the $1+ \alpha$ moment of the offspring distribution $\pi^{K}$,
\begin{eqnarray*} 
\kappa^{K}_{1+\alpha}(x,h) = \sum_{k =1}^{\infty} k^{1+\alpha} \pi^{K}(x,h, k), \hspace*{6mm} \text{for} \, \, x,h \in \hat{\mathcal{X}}.
\end{eqnarray*}

\noindent Note that if $\kappa^{K}_{1+\alpha}(x,h)< \infty$ for $x,h \in \hat{\mathcal{X}}$ then also $\kappa^{K}(x,h) < \infty$ for $x,h \in \hat{\mathcal{X}}$.

We consider the following hypotheses.

\begin{assumption} \label{assum3}
The population parameters satisfy:
\begin{itemize}

\item[(i)] $c_{K}(x,y) = K^{-1}c(x,y)$ for all $x,y \in \hat{\mathcal{X}}$ and $c \in C_{\partial}(\mathcal{X} \times \mathcal{X}, \mathbb{R}_{+})$.

\item[(ii)] There exists $\alpha \in (0,1]$ such that the $(1+ \alpha)$-th moment of the offspring distribution is uniformly bounded, i.e.,
\begin{eqnarray*}
 \sup_{x,h \in \hat{\mathcal{X}}}\kappa^{K}_{1+\alpha}(x,h) = \sup_{x,h \in \hat{\mathcal{X}}} \sum_{k =1}^{\infty} k^{1+\alpha} \pi^{K}(x, h, k) < + \infty.
\end{eqnarray*}

\item[(iii)] $\sup_{K}  \sup_{x \in \hat{\mathcal{X}}} \left( | \kappa^{K}(x,x) b_{K}(x) - d_{K}(x)| + K^{-\alpha}| \kappa^{K}_{1+\alpha}(x,x) b_{K}(x) + d_{K}(x)|\right) < + \infty$.


\item[(iv)] For $z \geq 0$ and $x \in \hat{\mathcal{X}}$, we define
\begin{eqnarray*}
\psi^{K}(x,z) = b_{K}(x)(g^{K}(x,x, e^{-z}) -1) + d_{K}(x)( e^{z} -1).
\end{eqnarray*}

The sequence $(K\psi^{K}(x,z/K))_{K}$ converges, uniformly on $\hat{\mathcal{X}} \times [0, a]$ for each $a \geq 0$, to
\begin{eqnarray*}
\psi(x, z) =   -b(x) z + \sigma(x)z^{2} + \int_{0}^{\infty} (e^{-zu} -1 +zu) \Pi(x, {\rm d}u),
\end{eqnarray*}

\noindent where $b \in C_{\partial}(\mathcal{X}, \mathbb{R})$, $\sigma \in C_{\partial}(\mathcal{X}, \mathbb{R}_{+})$  and $\Pi(x, \cdot)$ is a kernel from $\hat{\mathcal{X}}$ to $(0, \infty)$ such that
\begin{eqnarray*}
\sup_{x \in  \hat{\mathcal{X}}} \int_{0}^{\infty} (u \wedge u^{2}) \Pi(x, {\rm d}u) < + \infty, \hspace*{5mm} \text{and} \hspace*{5mm} \int_{B} (u \wedge u^{2}) \Pi(x, {\rm d}u) \in C_{\partial}(\mathcal{X}, \mathbb{R}_{+}),
\end{eqnarray*}

\noindent for each $B \in \mathbb{B}(\mathbb{R}_{+})$.

\item[(v)] $p \in C_{\partial}(\mathcal{X}, [0,1])$.

\item[(vi)] For $x \in \hat{\mathcal{X}}$, the mutation kernel $m_{K}(x, {\rm d} h)$ is absolutely continuous with respect to a $\sigma$-finite probability measure $\bar{m}$ on $\hat{\mathcal{X}}$ with density $m_{K}(x, h)$. Moreover, 
\begin{eqnarray*}
 \sup_{x \in \hat{\mathcal{X}}} \left|  b_{K}(x) \int_{ \hat{\mathcal{X}}}( \kappa^{K}(x,h) - \kappa^{K}(x,x)) m_{K}(x,h)  \bar{m}({\rm d}h) \right| < \infty,\\
\end{eqnarray*}
\noindent and
\begin{eqnarray*}
  \sup_{x \in \hat{\mathcal{X}}} \left| K^{-\alpha} b_{K}(x) \int_{ \hat{\mathcal{X}}}( \kappa^{K}_{1+\alpha}(x,h) - \kappa^{K}_{1+\alpha}(x,x)) m_{K}(x,h)  \bar{m}({\rm d}h) \right| < \infty.
\end{eqnarray*}

\item[(vii)] There is a bounded generator $A$ of a Feller semi-group on $C_{b}(\hat{\mathcal{X}}, \mathbb{R})$ with domain $\mathscr{D}(A)$, dense in $C_{b}(\hat{\mathcal{X}} , \mathbb{R})$, such that for all $\phi \in \mathscr{D}(A)$
\begin{eqnarray*}
\lim_{K \rightarrow \infty} \sup_{x \in \hat{\mathcal{X}}} \left|  K  b_{K}(x)  \int_{\hat{\mathcal{X}}} \left( g^{K}\left(x,x, e^{-  \frac{ \phi(x)}{K}} \right) -  g^{K}\left(x,h, e^{-\frac{ \phi(h)}{K}} \right)\right) m_{K}(x,h) \bar{m}({\rm d}h) - A \phi(x) \right | = 0.
\end{eqnarray*}
\end{itemize}
\end{assumption}

The motivation behind Assumption \ref{assum3} (iv) comes from the theory of superprocesses. More precisely, from the approximation of branching particle systems that leads to measure-valued branching processes with local branching mechanisms; see for example \cite[Section 4.4]{Dawson1991} or \cite[Proposition 4.3]{Li2011}.

\begin{remark}
A classical choice for the competition function in Assumption \ref{assum3} (i) is $c\equiv 1$. This corresponds to density dependence involving the total population size known as the ``mean field case'' or the ``logistic case''.
\end{remark}

\begin{remark} \label{rem1}
Typically, the choice of the functions $b_{K}$ and $d_{K}$ will depend on the offspring distribution. We illustrate this with two examples:
\begin{itemize}
\item[(a)] {\bf Single offspring distribution.} The reproduction law satisfies
\begin{eqnarray*}
g^{K}(x, h, z) = z, \hspace*{5mm} x,h \in \hat{\mathcal{X}} \hspace*{5mm} \text{and}  \hspace*{5mm} |z| \leq 1.
\end{eqnarray*}

In particular, choosing $b_{K}$ and $d_{K}$ proportional to $K$ yields the case studied by Champagnat, et al. \cite{Cha2008}. More precisely,
Champagnat, et al. \cite{Cha2008} considered
\begin{eqnarray*}
b_{K}(x) = K \sigma(x) + b(x) \hspace*{5mm} \text{and}  \hspace*{5mm} d_{K}(x) = K \sigma(x),
\end{eqnarray*}

\noindent for $x \in \hat{\mathcal{X}}$ and where $b,\sigma \in C_{\partial}(\mathcal{X}, \mathbb{R}_{+})$. In particular,
the sequence $(K\psi^{K}(x,z/K))_{K}$ converges, uniformly on $\hat{\mathcal{X}} \times [0, a]$ for each $a \geq 0$, to
\begin{eqnarray*}
\psi(x, z) =  - b(x) z + \sigma(x)z^{2}.
\end{eqnarray*}

Note also that the parameters in this example satisfy Assumption \ref{assum3} (ii), (iii) and (vi) with $\alpha =1$. This model has been also studied in \cite{Jou},\cite{Me2012} and \cite{MeTra2012} in a similar setting.
Finally, let us mention that Champagnat, et al. \cite{Cha2008} also studied the case of single offspring distribution where the natural birth and death rates are proportional to $K^{\eta}$, for some $\eta \in (0,1)$. In this scenario, they showed that the limit process is deterministic and described by a partial differential equation. This follows from the fact that the variance vanishes in the limit.

\item[(b)] {\bf $\beta$-stable offspring distribution.} The reproduction law satisfies
\begin{eqnarray*}
g^{K}(x, h, z) = \frac{1}{\beta}(1-z)^{1+\beta} + \frac{1+\beta}{\beta}z - \frac{1}{\beta}, \hspace*{5mm} x,h \in \hat{\mathcal{X}},  \hspace*{5mm} |z| \leq 1, \hspace*{5mm} \text{and} \hspace*{3mm} \beta \in (0, 1].
\end{eqnarray*}

This type of offspring distribution has been used in order to get convergence of branching particle systems to the so-called $(\alpha, d, \beta)$-superprocess (see for example \cite[Section 4.5]{Dawson1991}). In this case, in order to obtain a nontrivial limit we must choose $b_{K}$ and $d_{K}$ proportional to $K^{\beta}$. Moreover, Assumptions \ref{assum3} (ii), (iii), and (vi) have to hold with $\alpha =\beta$. Clearly, the variance of the offspring distribution is infinite and therefore the limiting process can no longer have finite second moments. On the other hand, as in the example (a), we expect that the limit process is deterministic and described by a partial differential equation if the offspring distribution is $\beta$-stable and $b_{K}$, $d_{K}$ are proportional to $K^{\eta^{\prime}}$, for some $\eta^{\prime} \in (0,\beta)$.

\item[(c)] Consider a function $\Lambda_{K} \in B(\hat{\mathcal{X}} \times \hat{\mathcal{X}}, \mathbb{R}_{+})$ such that $\Lambda_{K}(x,x) = 0$, for $x \in \hat{\mathcal{X}}$. Suppose that the reproduction law satisfies
\begin{eqnarray*}
g^{K}(x,h,z) = z \exp(- \Lambda_{K}(x,h) (1-z) ), \hspace*{5mm} x,h \in \hat{\mathcal{X}} \hspace*{5mm} \text{and}  \hspace*{5mm} |z| \leq 1.
\end{eqnarray*}
By considering $b_{K}(x) = K \sigma(x) + b(x)$ and $d_{K}(x) = K \sigma(x)$,  for $x \in \hat{\mathcal{X}}$ and where $b,\sigma \in C_{\partial}(\mathcal{X}, \mathbb{R}_{+})$, one can check that Assumption \ref{assum3} (iv) is fulfilled. This generating function corresponds to a random variable $X_{x,h} + 1$, where $X_{x,h}$ is distributed according to a Poisson random variable of parameter $\Lambda_{K}(x,h)$. For instance, we could take $\mathcal{X} = [x_{1}, x_{2}]$, with $-\infty < x_{1} < x_{2} < + \infty$, and $\Lambda_{K}(x,h) = |x-h|$. 
In this case Assumptions \ref{assum3} (ii) and (iii) are satisfied with $\alpha =1$. In general, this depends on the choice of $\Lambda_{K}$.

\item[(d)] Let $b_{K}, d_{K} \in C_{\partial}(\mathcal{X}, \mathbb{R}_{+})$ such that $b_{K} \rightarrow b$ and $d_{K} \rightarrow d$, as $K \rightarrow \infty$, (uniformly on $\hat{\mathcal{X}}$) where $b, d \in C_{\partial}(\mathcal{X}, \mathbb{R}_{+})$. Suppose that  Assumption \ref{assum3} (ii) is satisfied and that for each $a \geq 0$,
\begin{eqnarray*}
b_{K}(x)\left( g^{K}\left(x,x, e^{-\frac{ z}{K}} \right) -1 \right) = - \frac{z}{K} b_{K}(x) \kappa^{K}(x,x)  + o\left(\frac{1}{K} \right),
\end{eqnarray*}

\noindent for $x \in \hat{\mathcal{X}}$ and $z \in [ 0, a]$ (the small $o$ term is uniform on $\hat{\mathcal{X}} \times [0, a]$). Then $(K\psi^{K}(x,z/K))_{K}$ converges, uniformly on $\hat{\mathcal{X}} \times [0, a]$ for each $a \geq 0$,
\begin{eqnarray*}
\psi(x,z) =  (b(x) \kappa(x)-d(x)) z
\end{eqnarray*}
\noindent if $\kappa^{K}(x,x) \rightarrow \kappa(x) \in C_{\partial}( \mathcal{X}, \mathbb{R}_{+})$, as $K \rightarrow \infty$, uniformly on $\hat{\mathcal{X}}$. In this case, one may expect to obtain a deterministic limiting process described by a partial differential equation as \cite[Theorem 4.2]{Cha2008} or \cite[Theorem 5.3]{Four2004}.
\end{itemize}
\end{remark}

Let us now state our main theorem. For $0 < T < + \infty$ and $\mu^{K} \in \mathcal{M}^{K}(\mathcal{X})$, let us call $\mathbf{Q}^{K} = \mathbf{L}(\mu^{K})$ the law of the process $\nu^{K} = (\nu^{K}_{t}, t \in [0,T])$ such that $\mathbf{Q}^{K}(\nu^{K}_{0} = \mu^{K})=1$. We denote by $\mathbf{E}^{K}$ the expectation with respect $\mathbf{Q}^{K}$. We make a slight abuse of notation and denote by $\mathbf{1} = \mathds{1}_{\hat{\mathcal{X}}}$ the indicator function on the whole space $\hat{\mathcal{X}}$, unless we specify otherwise.

\begin{theorem} \label{teo3}
Suppose that Assumption \ref{assum3} is fulfilled. Assume also that there exists $\mu \in \mathcal{M}(\mathcal{X})$ (possibly random) such that
\begin{eqnarray*}
\lim_{K \rightarrow \infty} \nu_{0}^{K} = \mu
\end{eqnarray*}

\noindent in law for the weak topology on $\mathcal{M}(\mathcal{X})$ and that
\begin{eqnarray} \label{eq13}
\sup_{K} \mathbf{E}^{K}[\langle \nu_{0}^{K}, \mathbf{1} \rangle^{1+\alpha}] < \infty,
\end{eqnarray}

\noindent where $\alpha \in (0,1]$ is as in Assumption \ref{assum3} (ii). Then, for each $0 < T < + \infty$:
\begin{itemize}
\item[(a)] The sequence of
laws $(\mathbf{Q}^{K})_{K}$ is tight in $\mathcal{P}(\mathbb{D}([0,T], \mathcal{M}(\hat{\mathcal{X}})))$.

\item[(b)] Let $\mathbf{Q}_{\mu}$ be a limit point of $(\mathbf{Q}^{K})_{K}$. Then, the measure-valued process $\nu \in \mathbb{D}([0,T], \mathcal{M}(\hat{\mathcal{X}}))$, with law $\mathbf{Q}_{\mu}$ such that $\mathbf{Q}_{\mu}(\nu_{0} = \mu)=1$, satisfies the following conditions:
\begin{enumerate}
\item We have that
\begin{eqnarray*}
\sup_{t \in [0, T]} \mathbf{E}_{\mathbf{Q}_{\mu}}[\langle \nu_{t}, \mathbf{1} \rangle^{1+\alpha}] < + \infty.
\end{eqnarray*}

\item The measure-valued process $\nu \in \mathbb{D}([0,T], \mathcal{M}(\hat{\mathcal{X}}))$, or equivalently its law $\mathbf{Q}_{\mu}$, solves the following martingale problem: For any
\begin{eqnarray*}
\phi \in \mathscr {D}(\mathbf{M}) := C_{\partial}(\mathcal{X}, \mathbb{R}_{+})^{+} \cap \mathscr{D}(A)\end{eqnarray*}

\noindent the process $M(\phi) = (M_{t}(\phi), t \in [0,T])$ given by
\begin{align} \label{eq10}
M_{t}(\phi) & =  \langle \nu_{t}, \phi \rangle - \langle \nu_{0}, \phi \rangle - \int_{0}^{t} \int_{\hat{\mathcal{X}}} p(x) A \phi(x)  \nu_{s}({\rm d} x) {\rm d} s \nonumber \\
& \hspace*{10mm}  - \int_{0}^{t} \int_{\hat{\mathcal{X}}}  \phi(x)  \left( \ b(x) - \int_{\hat{\mathcal{X}}} c(x,y) \nu_{s}({\rm d}y) \right) \nu_{s}({\rm d}x){\rm d} s  \tag{$\mathbf{M}$}
\end{align}

\noindent is a $\mathbf{Q}_{\mu}$-martingale. Moreover, $M(\phi)$ admits the decomposition $M(\phi) = M^{\text{c}}(\phi) + M^{\text{d}}(\phi)$, where $M^{\text{c}}(\phi)$ is a continuous martingale with increasing process
\begin{eqnarray}
\label{eq:quadrativ-var}
 2 \int_{0}^{t} \int_{\hat{\mathcal{X}}} \sigma(x) \phi^{2}(x) \nu_{s}({\rm d}x) {\rm d}s,
\end{eqnarray}

\noindent and $M^{\rm d}(\phi)$ is a purely discontinuous martingale, i.e.
\begin{eqnarray}
\label{eq:discrete_mg}
M^{\rm d}_{t}(\phi) = \int_{0}^{t} \int_{\mathcal{M}(\hat{\mathcal{X}})}  \langle \mu , \phi \rangle \tilde{N}({\rm d}s, {\rm d}\mu),
\end{eqnarray}

\noindent where $\tilde{N}({\rm d}s, {\rm d}\mu)$ is the  compensated random measure of the optional random measure
$N({\rm d}s, {\rm d}\mu)$ on $[0, \infty) \times \mathcal{M}(\hat{\mathcal{X}})$ given by
\begin{eqnarray} \label{eq:N}
N({\rm d}s, {\rm d}\mu) = \sum_{s  > 0} \mathds{1}_{\{\Delta \nu_{s} \neq \mathbf{0}  \}} \delta_{(s, \Delta \nu_{s})} ({\rm d}s, {\rm d}\mu)
\end{eqnarray}
\noindent with $\Delta \nu_{s} = \nu_{s} - \nu_{s-} \in \mathcal{M}(\hat{\mathcal{X}})$ and $\mathbf{0}$ denoting the null measure.
We then have $\tilde{N}=N-\hat{N}$ with the compensator $\hat{N}({\rm d}s, {\rm d} \mu) = {\rm ds} \, \hat{n}(\nu_{s}, {\rm d} \mu)$ and  $\hat{n}(\nu_{s}, {\rm d} \mu)$ given by
\begin{eqnarray} \label{eq:n}
\int_{\mathcal{M}(\hat{\mathcal{X}})} f(\mu) \hat{n}(\nu_{s}, {\rm d} \mu) = \int_{\hat{\mathcal{X}}} \int_{0}^{\infty} f(u \delta_{x}) \Pi(x, {\rm d}u) \nu_{s}({\rm d}x) \hspace*{5mm} \text{for} \hspace*{5mm} f \in B(\mathcal{M}(\hat{\mathcal{X}}), \mathbb{R}),
\end{eqnarray}

\noindent and $\tilde{N}({\rm d}s, {\rm d}\mu)$ the corresponding compensated random measure.
\end{enumerate}
\end{itemize}
\end{theorem}

Let us provide some specific examples of Assumption \ref{assum3} (vii), in order to show that a large class of dynamics can be included. 
\begin{remark} \label{exrem1}
Suppose that there exists a sequence $(a_{K})_{K}$ of positive real numbers such that $a_{K} \rightarrow \infty$ and $b_{K}(\cdot)/a_{K} \rightarrow 1$ (uniformly on $\hat{\mathcal{X}}$), as $K \rightarrow \infty$. For $\phi \in B(\hat{\mathcal{X}}, \mathbb{R})$, suppose also that 
\begin{eqnarray*}
 g^{K}\left(x,x, e^{-  \frac{ \phi(x)}{K}} \right) -  g^{K}\left(x,h, e^{-\frac{ \phi(h)}{K}} \right) = \frac{\phi(h)}{K} - \frac{\phi(x)}{K} + o\left( \frac{1}{K a_{K}}\right),
 \end{eqnarray*}
 \noindent for $x,h \in \hat{\mathcal{X}}$ and where the small $o$ term is uniform on $\hat{\mathcal{X}} \times \hat{\mathcal{X}}$. 
\begin{itemize}
\item[(a)] Consider the case when $\mathcal{X} = [x_{1}, x_{2}]$ with $-\infty < x_{1} < x_{2} < + \infty$. Then the mutation kernel $m_{K}(x, {\rm d}h)$ can be a Gaussian distribution (conditioned to be in $[x_{1}, x_{2}]$) with mean $x \in [x_{1}, x_{2}]$ and variance $\theta^{2}/a_{K}$ such that $(\theta^{2}/a_{K})b_{K}( \cdot) \rightarrow \theta^{2} \in \mathbb{R}_{+}$ (uniformly on $\mathcal{X}$), as $K \rightarrow \infty$. In this case, we have $A\phi = \frac{\theta^{2}}{2} \phi^{\prime \prime}$ for $\phi \in C^{2}_{b}([x_{1}, x_{2}], \mathbb{R})$ such that $\phi^{\prime}(x_{1}) = \phi^{\prime}(x_{2}) = 0$. If in addition $m_{K}(x, \cdot)$ has mean $x+\gamma_{K}$ such that $\gamma_{K} b_{K}(\cdot) \rightarrow \gamma \in \mathbb{R}$ (uniformly on $\mathcal{X}$), as $K \rightarrow \infty$, corresponding to a mutational drift, then $A\phi = \frac{\theta^{2}}{2} \phi^{\prime \prime} + \gamma \phi^{\prime}$.
\item[(b)] Consider the case when $\mathcal{X} = \mathbb{R}^{l}$, $l \geq 1$. For $x \in \hat{\mathcal{X}}$, let the mutation kernel $m_{K}(x, {\rm d}h)$ be the density of a random variable with mean vector $x$ and covariance matrix $\Sigma(x)/a_{K} = (\Sigma_{ij}(x)/a_{K}, 1 \leq i,j \leq l)$. Moreover, assume that the function $\Sigma$ is bounded and that the third moment of $m_{K}(x, {\rm d}h)$ is of order $1/a^{2}_{K}$ uniformly on $x \in \hat{\mathcal{X}}$. Then, the generator $A$ is  for $\phi \in C_{b}^{2}(\mathbb{R}^{l}, \mathbb{R})$ given by $A\phi(x) = \frac{1}{2} \sum_{i,j = 1}^{l} \Sigma_{ij}(x) \frac{\partial^{2} \phi(x)}{\partial x_{i} \partial x_{j}} $. For instance,
\begin{eqnarray*}
m_{K}(x, {\rm d}h) = \left( \frac{a_{K} }{2 \pi \theta^{2}(x)} \right)^{l/2} \exp \left( - \frac{ a_{K} | h -x |^{2}}{2 \theta^{2}(x)} \right) \mathds{1}_{\{ h \in \mathbb{R}^{l} \}} \, {\rm d} h,
\end{eqnarray*}
for $x \in \mathbb{R}^{l}$ and $\theta^{2}(x)$ positive and bounded.
\item[(c)] Take $\mathcal{X} = \mathbb{R}$, and for $x \in \hat{\mathcal{X}}$, let the mutation kernel $m_{K}(x, {\rm d}h)$ be the distribution of a Pareto random variable with index $\beta \in (1,2)$ divided by $K^{\eta/\beta}$, for $\eta \in (0,1]$. It has been proved by Jourdain et al. \cite{Jou} that for $\phi \in C_{b}^{2}(\mathbb{R}, \mathbb{R})$,
\begin{eqnarray*}
\lim_{K \rightarrow \infty} \sup_{x \in \hat{\mathcal{X}}}
\left | K^{\eta} \int_{\mathcal{X}}(\phi( h) - \phi(x))  m_{K}(x, {\rm d}h) - \frac{\beta}{2} D^{\beta}\phi(x) \right | = 0,
\end{eqnarray*}
\noindent where
\begin{eqnarray*}
D^{\beta}\phi(x) = \int_{\mathcal{X}} (\phi(x+h) - \phi(x)-h \phi^{\prime}(x) \mathds{1}_{\{ |h| \leq 1 \}}) \frac{{\rm d} h}{|h|^{1+\beta}}, \hspace*{6mm} x \in \hat{\mathcal{X}},
\end{eqnarray*}
\noindent is the fractional Laplacian of index $\beta$. Thus, in this example, Assumption \ref{assum3} (vii) is satisfied with $A = D^{\beta}$ as long as we take $a_{K} = K^{\eta}$.
\item[(d)] An interesting example is when $\mathcal{X} = \{x_{1}, x_{2} \}$ is a set of two traits. Consider the mutation kernel 
\begin{eqnarray*}
m_{K}(x,{\rm d}h) = \mathds{1}_{\{x=x_{1}\}} \left(q_{x_{1}}^{K} \delta_{x_{2}}({\rm d}h) 
+ (1-q_{x_{1}}^{K}) \delta_{x_{1}}({\rm d}h) \right)
+ \mathds{1}_{\{x=x_{2}\}} \left( q_{x_{2}}^{K} \delta_{x_{1}}({\rm d}h) + (1-q_{x_{2}}^{K}) \delta_{x_{2}}({\rm d}h) \right), 
\end{eqnarray*}
\noindent where $q_{x_{1}}^{K}, q_{x_{2}}^{K} \in (0,1)$ such that $a_{K}q_{x_{1}}^{K} \rightarrow q_{x_{1}}$ and $a_{K}q_{x_{2}}^{K} \rightarrow q_{x_{2}}$, as $K \rightarrow \infty$, with $q_{x_{1}}, q_{x_{2}} \in (0,1)$. Thus, Assumption \ref{assum3} (vii) is satisfied with
\begin{eqnarray*}
A \phi(x)= \mathds{1}_{\{x=x_{1}\}} q_{x_{1}} (\phi(x_{2}) - \phi(x_{1})) + \mathds{1}_{\{x=x_{2}\}} q_{x_{2}}(\phi(x_{1}) - \phi(x_{2})), 
\end{eqnarray*}
\noindent where $\phi \in C_{b}(\mathcal{X}, \mathbb{R})$. 
\end{itemize}
\end{remark}

\begin{remark}
 Let $\mathcal{X} = [x_{1}, x_{2}]$ with $-\infty < x_{1} < x_{2} < + \infty$ and suppose that that there exists a sequence $(a_{K})_{K}$ of positive real numbers such that $a_{K} \rightarrow \infty$, $a_{K} = O(K)$ and $b_{K}(\cdot)/a_{K} \rightarrow 1$ (uniformly on $\hat{\mathcal{X}}$), as $K \rightarrow \infty$. We consider a reproduction law that satisfies 
\begin{eqnarray*}
g^{K}(x,h,z) = z \exp(- a_{K}^{-1/2}|x-h| (1-z) ), \hspace*{5mm} x,h \in [x_{1}, x_{2}] \hspace*{5mm} \text{and}  \hspace*{5mm} |z| \leq 1;
\end{eqnarray*}
\noindent recall the example in Remark \ref{rem1} (c). For $\phi \in B(\hat{\mathcal{X}}, \mathbb{R})$, note that 
\begin{align*}
 g^{K}\left(x,x, e^{-  \frac{ \phi(x)}{K}} \right) -  g^{K}\left(x,h, e^{-\frac{ \phi(h)}{K}} \right)  & = \frac{\phi(h)}{K} - \frac{\phi(x)}{K} + \frac{\phi(x)^{2}}{2K^{2}} - \frac{\phi(h)^{2}}{2K^{2}}  \\
 & \hspace*{5mm} + \frac{\phi(h)}{K}a_{K}^{-1/2}|x-h|+ o\left( \frac{1}{K a_{K}}\right),
 \end{align*}
 \noindent for $x,h \in  [x_{1}, x_{2}]$ with a  big $O$ term that  is uniform on $[x_{1}, x_{2}] \times [x_{1}, x_{2}]$. Assume that the mutation kernel $m_{K}(x, {\rm d}h)$ is the Gaussian distribution (conditioned to be in $[x_{1}, x_{2}]$) with mean $x \in [x_{1}, x_{2}]$ and variance $\theta^{2}/a_{K}$ as in Remark \ref{exrem1} (a). Then an elementary computation shows that $A\phi = \frac{\theta^{2}}{2} \phi^{\prime \prime} +  \sqrt{\frac{2}{\pi}} \theta \phi$, for $\phi \in C^{2}_{b}([x_{1}, x_{2}], \mathbb{R})$ such that $\phi^{\prime}(x_{1}) = \phi^{\prime}(x_{2}) = 0$.
\end{remark}


We state the following result on uniqueness of the limiting process of Theorem \ref{teo3}. Recall that $C_{\partial}(\mathcal{X}, \mathbb{R}_{+})^{+}$ denotes the subset of functions in $C_{b}(\mathcal{X}, \mathbb{R}_{+})$  that can be extended continuously to $\hat{\mathcal{X}}$ and that are bounded away from zero. 

\begin{theorem} \label{teo4}
Let $\mu \in \mathcal{M}(\mathcal{X})$ be non random and suppose that the coefficients in the martingale problem $($\ref{eq10}$)$ are as in Assumption \ref{assum3}  and such that in addition $\sigma \in 	C_{\partial}(\mathcal{X}, \mathbb{R}_{+})^{+}.$ Then uniqueness holds for the martingale problem $($\ref{eq10}$)$ and it is thus well posed. In particular, 
$(\mathbf{Q}^{K})_{K}$ in Theorem \ref{teo3} converge to this unique solution of the martingale problem.
\end{theorem}

Finally, we provide a criterion to check that no mass escapes.

\begin{theorem} \label{teo9}
Assume that $\mathcal{X}$ is not compact and that Assumption \ref{assum3} is fulfilled with  $\sigma \in 	C_{\partial}(\mathcal{X}, \mathbb{R}_{+})^{+}$.
Moreover, suppose that there exists a sequence $(\phi_{n})_{n \geq 1} \subset \mathscr{D}(\mathbf{M})$ such that
\begin{itemize}
\item[(i)] for every $n \geq 1$, $\lim_{d(x, \partial) \rightarrow 0} \phi_{n}(x) = 1$, $\lim_{d(x, \partial) \rightarrow 0}  A\phi_{n}(x) = 0$ (where $d$ is some metric in the Polish space $\mathcal{X}$),
\item[(ii)] $\sup_{n \geq 1} \sup_{x \in \mathcal{X}} \phi_{n}(x) < \infty$ and $\sup_{n \geq 1} \sup_{x \in \mathcal{X}} A\phi_{n}(x) < \infty$. Furthermore, $ \phi_{n} \rightarrow \mathds{1}_{\{ \partial\}}$ and $ A\phi_{n} \rightarrow 0$, as $n \rightarrow \infty$, pointwise.
\end{itemize}

\noindent For a non random $\mu \in \mathcal{M}(\mathcal{X})$, let $\mathbf{Q}_{\mu}$ be the unique solution to the martingale problem $($\ref{eq10}$)$. Then, $\mathbf{Q}_{\mu}$ is actually the law of a measure-valued process in $\mathbb{D}([0,T], \mathcal{M}(\mathcal{X}))$.
\end{theorem}

\begin{remark}
We give an example that satisfies the conditions of Theorem \ref{teo9}. Consider the framework of Remark \ref{exrem1} (b) where $\mathcal{X} = \mathbb{R}^{l}$, $l \geq 1$. In this case, $A$ is given for $\phi \in C_{\partial}^{2}(\mathbb{R}^{l}, \mathbb{R})$ by $A\phi(x) = \frac{1}{2} \sum_{i,j = 1}^{l} \Sigma_{ij}(x) \frac{\partial^{2} \phi(x)}{\partial x_{i} \partial x_{j}} $. For $n \geq 1$, let
\begin{eqnarray*}
\phi_{n}(x) = \left\{ \begin{array}{lcl}
              \exp \left(- \frac{1}{|x|^{2} - n^{2}} \right) & \mbox{  if } & |x| > n, \\
              0  & \mbox{  if } & |x| \leq n. \\
              \end{array}
    \right.
\end{eqnarray*}

\noindent Then, one can easily check that $(\phi_{n})_{n \geq 1} \subset \mathscr{D}(\mathbf{M})$ and that the conditions of Theorem \ref{teo9} are satisfied.
\end{remark}

\begin{remark}
We notice the following:
\begin{itemize}
\item[(a)] Following the terminology of the theory of superprocesses (see for example \cite{Dawson1991}, \cite{Fitz1992} and \cite{Li2011}), one could refer to the solution of the
martingale problem (\ref{eq10}) as the $( A, \psi, c)$-superprocess with competition, i.e.\ a superprocess with spatial motion governed by the infinitesimal generator $A$, reproduction mechanism or branching mechanism $\psi$, and with competition kernel $c$.

\item[(b)] In the non-spatial setting (i.e.\ one type setting without mutation), Lambert \cite{Lam2005} introduced a class of general branching processes with logistic growth, abbreviated $LB$-processes. These processes may be viewed as continuous-state branching processes (CSBP's) with an extra negative interaction between each pair of individuals in the population. Then the solution to the martingale problem (\ref{eq10}) can also be thought-out as a generalization of $LB$-processes to model spatially structured populations.

\item[(c)] The solution of the martingale problem (\ref{eq10}) generalizes the models proposed by M\'el\'eard \cite{Four2004},  Champagnat, et al.\ \cite{Cha2008}, Jourdain, et al.\ \cite{Jou}. These can be recovered by considering $\psi(x,z) =  -b(x)z + \sigma(x)z^{2}$ for $x \in \mathcal{X}$ and $z \geq 0$, and $A$ the Laplacian or the fractional Laplacian. The solution to (\ref{eq10}) can also be seen as an extension of the model of Etheridge \cite{Eth2004} by taking $\psi(x,z) =  -bz + \sigma z^{2}$ for $x \in \mathcal{X}$, $z \geq 0$, $b, \sigma$ constants, $A$ the Laplacian and $c(x,y) = h(|x-y|)$, for $x,y \in \mathcal{X}$, with a nonnegative decreasing function $h$ on $\mathbb{R}_{+}$ that satisfies $\int_{0}^{\infty} h(r) r^{d-1} {\rm d}r < \infty$.

\item[(d)] It is important to point out that we were not able to show uniqueness in general for the martingale problem (\ref{eq10}). More precisely, the case when the diffusion part in the branching mechanism $\psi$, i.e.\ $\sigma$ in Assumption \ref{assum3} (iv), is no bounded away from zero is not covered in Theorem \ref{teo4}. We could not obtain a useful Girsanov type theorem in this case (see Section \ref{Uni2}) to get rid of the nonlinearity problems, which prevent us to use Laplace-transform techniques as in the classical theory of superprocess. It seems that this is a really hard problem.
\end{itemize}
\end{remark}

Section \ref{profteo3} is devoted to the proof of Theorem \ref{teo3}. We firstly establish in Section \ref{tightness} the tightness of the sequence $(\nu^{K})_{K}$, i.e., Theorem \ref{teo3} (a). In Section \ref{Ident}, we identify its limiting values $\nu$, and we show that they satisfy the properties of Theorem \ref{teo3} (b). In Section \ref{Uni}, we prove Theorem \ref{teo4} about uniqueness of the limiting process and the convergence of the sequence $(\nu^{K})_{K}$. Finally, we show that there is no escape of mass for the limiting process, i.e. Theorem \ref{teo9}, in Section \ref{lastT}.

\section{Proof of Theorem \ref{teo3}} \label{profteo3}
Throughout this section we assume that the assumptions of Theorem \ref{teo3} are fulfilled.
\subsection{Tightness} \label{tightness}

For $0 < T < + \infty$ and $\mu^{K} \in \mathcal{M}(\mathcal{X})$, recall that $\mathbf{Q}^{K} = \mathbf{L}(\mu^{K})$ is the law of the process $\nu^{K} = (\nu^{K}, t \in [0,T])$ such that $\mathbf{Q}^{K}(\nu_{0}^{K} = \mu^{K})=1$. We denote by $\mathbf{E}^{K}$ the expectation with respect to $\mathbf{Q}^{K}$.

We shall prove that:

\begin{proposition} \label{prop2}
The sequence of laws $(\mathbf{Q}^{K})_{K}$ in $\mathcal{P}(\mathbb{D}([0,T], \mathcal{M}(\hat{\mathcal{X}})))$ is tight.
\end{proposition}

First, we obtain the following moment estimate.

\begin{lemma} \label{lemma1}
For all $0 < T < \infty$, we have that
\begin{eqnarray*}
\sup_{K} \mathbf{E}^{K}\big[  \sup_{t \in [0, T]} \langle \nu_{t}^{K}, \mathbf{1} \rangle^{1+\alpha} \big]  < + \infty,
\end{eqnarray*}

\noindent where $\alpha \in (0,1]$ is as in Assumption \ref{assum3} (ii). 
\end{lemma}

\begin{proof}
Following a similar construction as in Section \ref{poissonc}, one can define a $\mathcal{M}^{K}(\hat{\mathcal{X}})$-valued Markov process $\tilde{\nu} = (\tilde{\nu}_{t}^{K}, t \geq 0)$ with infinitesimal generator $\tilde{\mathscr{L}}^{K}$ given by
\begin{align*} 
\tilde{\mathscr{L}}^{K} f(\mu^{K}) & =  K \int_{\hat{\mathcal{X}}}  \sum_{k=1}^{\infty} b_{K}(x)(1-p(x)) \pi^{K}(x,x, k)
\left(  f \left(   \mu^{K} + k \frac{\delta_{x}}{K} \right) - f(\mu^{K}) \right) \mu^{K}({\rm d} x) \nonumber \\
& \hspace*{5mm} +  K  \int_{\hat{\mathcal{X}}} \sum_{k=1}^{\infty} b_{K}(x) p(x)   \int_{\hat{\mathcal{X}}} \pi^{K}(x,h,k) \left(  f \left(   \mu^{K} + k\frac{\delta_{h}}{K} \right) - f(\mu^{K}) \right) m_{K}(x, h) \bar{m}({\rm d}h) \mu^{K}({\rm d} x) \nonumber \\
& \hspace*{5mm} +  K \int_{\hat{\mathcal{X}}}   d_{K}(x)   \left(  f \left(   \mu^{K} - \frac{\delta_{x}}{K} \right) - f(\mu^{K}) \right) \mu^{K}({\rm d} x),
\end{align*}

\noindent for all $\mu^{K} \in \mathcal{M}^{K}(\hat{\mathcal{X}})$ and $f$ a measurable and bounded function from $\mathcal{M}^{K}(\hat{\mathcal{X}})$ to $\mathbb{R}$. This process can be associated with a stochastic interacting individual system where individuals behave as in the model described in Section \ref{model} but with the difference that there is no death due to competition.

Note also that the processes $( \langle \nu_{t}^{K}, \mathbf{1} \rangle, t \geq 0)$ and $( \langle \tilde{\nu}_{t}^{K}, \mathbf{1} \rangle, t \geq 0)$ can be coupled together such that $\langle \nu_{0}^{K}, \mathbf{1} \rangle = \langle \tilde{\nu}_{0}^{K}, \mathbf{1} \rangle$ and $\langle \nu_{t}^{K}, \mathbf{1} \rangle \leq \langle \tilde{\nu}_{t}^{K}, \mathbf{1} \rangle$ for all $t \geq 0$. Therefore, it is enough to prove the following moment bound to finish the proof,
\begin{eqnarray} \label{eq1Rev2}
\sup_{K} \mathbf{E}^{K}\big[  \sup_{t \in [0, T]} \langle \tilde{\nu}_{t}^{K}, \mathbf{1} \rangle^{1+\alpha} \big]  < + \infty.
\end{eqnarray}

Following a similar argument as in the proof of Theorem \ref{teo1} (i.e., a Poissonian construction for the system associated to the operator $\tilde{\mathscr{L}}^{K}$), one deduce that
\begin{align*}
& \mathbb{E} \big [  \langle \tilde{\nu}_{t}^{K}, \mathbf{1} \rangle^{1+\alpha} \big]  \leq \mathbb{E} \left [ \langle \nu_{0}^{K}, \mathbf{1} \rangle^{1+\alpha} \right] \\
& \hspace*{3mm}+ K  \, \mathbb{E} \Bigg [ \int_{0}^{t}  \int_{ \hat{\mathcal{X}}} \bigg( \sum_{k=1}^{\infty} b_{K}(x)(1-p(x))\pi^{K}(x,x,k) \left( \left(\langle \tilde{\nu}_{s}^{K}, \mathbf{1} \rangle + \frac{k}{K} \right)^{1+\alpha} - \langle \tilde{\nu}_{s}^{K}, \mathbf{1} \rangle^{1+\alpha} \right)
\\
& \hspace*{3mm}+ \sum_{k=1}^{\infty} b_{K}(x) p(x) \int_{ \hat{\mathcal{X}}} \pi^{K}(x,h,k) m_{K}(x,h) \left( \left(\langle \tilde{\nu}_{s}^{K}, \mathbf{1} \rangle + \frac{k}{K} \right)^{1+\alpha} - \langle \tilde{\nu}_{s}^{K}, \mathbf{1} \rangle^{1+\alpha} \right) \bar{m}({\rm d}h) \\
& \hspace*{3mm}+ d_{K}(x) \left( \left(\langle \tilde{\nu}_{s}^{K}, \mathbf{1} \rangle - \frac{1}{K} \right)^{1+\alpha} - \langle \tilde{\nu}_{s}^{K}, \mathbf{1} \rangle^{1+\alpha} \right) 
\bigg) \tilde{\nu}_{s}^{K}({\rm d}x ) {\rm d}s \Bigg],
\end{align*}

\noindent for $t \geq 0$. Note the following inequalities
\begin{eqnarray*}
(x+k)^{1+ \alpha} - x^{1+\alpha} \leq (1+\alpha) (x^{\alpha}k + k^{1+\alpha}) \hspace*{2mm} \text{and} \hspace*{2mm} (x-1)^{1+ \alpha} - x^{1+\alpha} \leq (1+\alpha)(1-x^{\alpha}), 
\end{eqnarray*}

\noindent for $x,k \in \mathbb{N}$ and $\alpha \in (0,1]$. Then, 
\begin{align*}
\mathbb{E} \big [  \langle \tilde{\nu}_{t}^{K}, \mathbf{1} \rangle^{1+\alpha} \big]  & \leq \mathbb{E} \left [ \langle \nu_{0}^{K}, \mathbf{1} \rangle^{1+\alpha} \right] + (1+\alpha)  \, \mathbb{E} \Bigg [ \int_{0}^{t} \int_{\hat{\mathcal{X}}} \bigg( (b_{K}(x) \kappa^{K}(x,x) -d_{K}(x))\langle \tilde{\nu}_{s}^{K}, \mathbf{1} \rangle^{\alpha}  \\
& \hspace*{4mm} + K^{-\alpha}(b_{K}(x) \kappa^{K}_{1+\alpha}(x,x) +d_{K}(x))
\\
& \hspace*{4mm}+ \langle \tilde{\nu}_{s}^{K}, \mathbf{1} \rangle^{\alpha} b_{K}(x) p(x) \int_{ \hat{\mathcal{X}}}( \kappa^{K}(x,h) - \kappa^{K}(x,x)) m_{K}(x,h)  \bar{m}({\rm d}h) \\
& \hspace*{4mm} +  K^{-\alpha} b_{K}(x) p(x) \int_{ \hat{\mathcal{X}}}( \kappa^{K}_{1+\alpha}(x,h) - \kappa^{K}_{1+\alpha}(x,x)) m_{K}(x,h)  \bar{m}({\rm d}h) \bigg) \tilde{\nu}_{s}^{K}({\rm d}x ) {\rm d}s \Bigg].
\end{align*}

\noindent It follows from Assumption \ref{assum3}  (ii), (iii), (v) and (vi)  that
\begin{align*}
\mathbf{E}^{K} \big[ \langle \tilde{\nu}_{t}^{K}, \mathbf{1} \rangle^{1+\alpha} \big] & \leq \mathbf{E}^{K} [\langle \nu_{0}^{K}, \mathbf{1} \rangle^{1+\alpha}]  + C    \int_{0}^{t} \mathbf{E}^{K} [ 1+ 
  \langle \tilde{\nu}^{K}_{s}, \mathbf{1} \rangle^{1+\alpha} ]
 {\rm d}s  ,
\end{align*}

\noindent for some positive constant $C$ (that does not depend on $K$). For all $0 < T < \infty$, Gronwall's Lemma and Condition (\ref{eq13}) allow us to conclude that there exists another constant $C_{T} > 0$, not depending on $K$, such that
\begin{eqnarray} \label{eq2Rev2}
\mathbf{E}^{K}\big[  \langle \tilde{\nu}_{t}^{K}, \mathbf{1} \rangle^{1+\alpha} \big]  \leq C_{T}
 \hspace*{4mm} \text{for all} \hspace*{2mm} t \in [0,T].
\end{eqnarray}

On the other hand, it is not difficult to see that the process $\tilde{M}^{K} = (\tilde{M}_{t}^{K}, t \geq 0)$ given by
\begin{eqnarray*}
\tilde{M}_{t}^{K} = \langle \tilde{\nu}_{t}^{K}, \mathbf{1} \rangle -\langle \nu_{0}^{K}, \mathbf{1} \rangle - \int_{0}^{t} \tilde{\mathscr{L}}^{K} g(\tilde{\nu}_{s}^{K}) {\rm d} s,
\end{eqnarray*}

\noindent where $g(\tilde{\nu}_{s}^{K}) = \langle \tilde{\nu}_{s}^{K}, \mathbf{1} \rangle$, is a martingale. But then, Doob's maximal inequality implies that there is a constant $C^{\prime} >0$ (that does not depend on $K$) such that
 \begin{eqnarray*}
 \mathbf{E}^{K}\big[\sup_{t \in [0, T]}\langle \tilde{\nu}_{t}^{K}, \mathbf{1} \rangle^{1+\alpha} \big]  \leq   \mathbf{E}^{K} [ \langle \nu_{0}^{K}, \mathbf{1} \rangle] +  C^{\prime} \mathbf{E}^{K} [|\tilde{M}_{T}^{K}|^{1+\alpha}] + \mathbf{E}^{K} \left[ \left( \int_{0}^{T} |\tilde{\mathscr{L}}^{K} g(\tilde{\nu}_{s}^{K})| {\rm d} s \right)^{1+\alpha}\right].
 \end{eqnarray*}

\noindent Note that Assumptions \ref{assum3} (ii), (iii), (v) and (vi) imply that there is a constant $C^{\prime \prime} >0$, not depending on $K$, such that
\begin{align} \label{eq1Rev22}
| \tilde{\mathscr{L}}^{K} g(\tilde{\nu}_{s}^{K})| & \leq \int_{\hat{\mathcal{X}}} (b_{K}(x) \kappa^{K}(x,x) -d_{K}(x)) \tilde{\nu}_{s}^{K}({\rm d} x) \nonumber \\
& \hspace*{5mm} +   \int_{\hat{\mathcal{X}}} b_{K}(x) p(x) \left( \int_{ \hat{\mathcal{X}}}( \kappa^{K}(x,h) - \kappa^{K}(x,x)) m_{K}(x,h)  \bar{m}({\rm d}h) \right) \tilde{\nu}_{s}^{K}({\rm d} x), \nonumber \\
& \leq C^{\prime \prime} \langle \tilde{\nu}_{s}^{K}, \mathbf{1} \rangle,
\end{align}

\noindent for $s \in [0,T]$. Then the above inequality together with Jensen's inequality shows that
 \begin{eqnarray} \label{eq2Rev22}
 \mathbf{E}^{K}\big[\sup_{t \in [0, T]}\langle \tilde{\nu}_{t}^{K}, \mathbf{1} \rangle^{1+\alpha} \big]  \leq   \mathbf{E}^{K} [ \langle \nu_{0}^{K}, \mathbf{1} \rangle] +  C^{\prime} \mathbf{E}^{K} [|\tilde{M}_{T}^{K}|^{1+\alpha}] + C_{T}^{\prime} \sup_{t \in [0,T]}\mathbf{E}^{K} [ \langle \tilde{\nu}_{t}^{K}, \mathbf{1} \rangle^{1+\alpha}],
 \end{eqnarray}
 
 \noindent for some constant $C_{T}^{\prime} > 0$ that does not depend on $K$. Finally, the claim in (\ref{eq1Rev2}) follows by combining Condition (\ref{eq13}), the inequalities (\ref{eq2Rev2}), (\ref{eq1Rev22}), (\ref{eq2Rev22}) and the expression for the martingale $\tilde{M}^{K}$.
\end{proof}

\begin{proof}[Proof of Proposition \ref{prop2}]
Notice that $\mathcal{M}(\hat{\mathcal{X}})$ is a Polish space by \cite[Theorem 3.1.7]{Ethier1986} which implies that $\mathbb{D}([0,T], \mathcal{M}(\hat{\mathcal{X}}))$ is also a Polish space. We use Jakubowski's criterion for tightness \cite[Theorem 3.1]{Jaku1986} that gives necessary and sufficient conditions for tightness in $\mathbb{D}([0,T], \mathcal{M}(\hat{\mathcal{X}}))$. More precisely, let $\mathscr{E}$ be a family of continuous functions on $\mathcal{M}(\hat{\mathcal{X}})$ that separates points in $\mathcal{M}(\hat{\mathcal{X}})$ and is closed under addition. We show that
\begin{itemize}
\item[(a)] For each $\varepsilon > 0$ there exits $\Gamma > 0$ such that
\begin{eqnarray*}
\mathbf{Q}^{K}\left(  \sup_{0 \leq t \leq T} \langle  \nu^{K}_{t}, \mathbf{1} \rangle \leq \Gamma \right) \geq 1 - \varepsilon,  \hspace*{4mm} K \in \mathbb{N}.
\end{eqnarray*}

\item[(b)] The family $(\mathbf{Q}^{K})_{K}$ is $\mathscr{E}$-weakly tight, i.e. for each $f \in \mathscr{E}$ the laws of $(f( \nu_{t}^{K}): t \in [0,T])$ under $\mathbf{Q}^{K}$ are tight.
\end{itemize}

We define the family of functions
\begin{eqnarray*}
 \mathscr{E} := \bigcup_{n=1}^{\infty} \left \{ \sum_{i=1}^{n} \theta_{i} e^{- \lambda_{i} \langle \mu, \phi_{i} \rangle}: \mu \in \mathcal{M}(\hat{\mathcal{X}}), \hspace*{2mm} \lambda_{i} \in \mathbb{R}_{+}, \hspace*{2mm} \theta_{i} \in \mathbb{R} \hspace*{3mm}  \text{and} \hspace*{3mm} \phi_{i} \in C_{\partial}(\mathcal{X}, \mathbb{R}_{+})  \right \}.
\end{eqnarray*}

\noindent Observe that it separates points on $\mathcal{M}(\hat{\mathcal{X}})$ (it follows from Dynkin's $\pi$-$\lambda$ Theorem; see for example \cite[Theorem 1.3.2 and 1.3.3]{Billi2008}) and that it is closed under addition. On the other hand,  Lemma \ref{lemma1} implies (a). Therefore, it only remains to show (b). We consider  $f \in \mathscr{E}$, i.e.,
\begin{eqnarray*}
 f(\mu) =  \sum_{i=1}^{n} \theta_{i} e^{- \lambda_{i} \langle \mu, \phi_{i} \rangle}, \hspace*{5mm} n \in \mathbb{N},
 \end{eqnarray*}

\noindent where $\mu \in \mathcal{M}(\hat{\mathcal{X}})$, $\lambda_{i} \in \mathbb{R}_{+}$, $\theta_{i} \in \mathbb{R}$ and $\phi_{i} \in C_{\partial}(\mathcal{X}, \mathbb{R}_{+})$, for $i = 1, \dots, n$. It has been shown in Remark~\ref{remfunc} that $f \in \mathscr{E}$ satisfies
condition (\ref{eq8}) in Proposition \ref{Prop1}. Then,
\begin{align} \label{exeq1}
| \mathscr{L}^{K} f(\nu^{K}_{t}) | & = \bigg| K \int_{\hat{\mathcal{X}}}  \sum_{k=1}^{\infty} b_{K}(x)(1-p(x)) \pi^{K}(x,x, k)
\left(  f \left(   \nu^{K}_{t} + k \frac{\delta_{x}}{K} \right) - f(\nu^{K}_{t}) \right) \nu^{K}_{t}({\rm d} x) \nonumber \\
& \hspace*{5mm} +  K  \int_{\hat{\mathcal{X}}} \sum_{k=1}^{\infty} b_{K}(x) p(x)   \int_{\hat{\mathcal{X}}} \pi^{K}(x,h,k) \left(  f \left(   \nu^{K}_{t} + k\frac{\delta_{h}}{K} \right) - f(\nu^{K}_{t}) \right) m_{K}(x, {\rm d} h) \nu^{K}_{t}({\rm d} x) \nonumber \\
& \hspace*{5mm} +  K \int_{\hat{\mathcal{X}}}  \left( d_{K}(x) + K \int_{\hat{\mathcal{X}}} c_{K}(x,y) \nu^{K}_{t}({\rm d} y) \right)  \left(  f \left(   \nu^{K}_{t} - \frac{\delta_{x}}{K} \right) - f(\nu^{K}_{t}) \right) \nu^{K}_{t}({\rm d} x) \bigg| \nonumber \\
 & \leq \sum_{i=1}^{n} |\theta_{i}|e^{- \lambda_{i} \langle \nu_{t}^{K}, \phi_{i} \rangle} \bigg | \int_{\hat{\mathcal{X}}}  K  \left( b_{K}(x) \left( g\left(x,x,e^{-\frac{\lambda_{i}\phi_{i}(x)}{K}}\right) -1 \right)  +d_{K}(x)\left(e^{\frac{\lambda_{i}\phi_{i}(x)}{K}} - 1 \right) \right) \nu^{K}_{t}({\rm d} x) \nonumber \\
& \hspace*{5mm} -   \int_{\hat{\mathcal{X}}}  K  b_{K}(x) p(x) \left( g\left(x,x,e^{-\frac{\lambda_{i}\phi_{i}(x)}{K}}\right) -g\left(x,h,e^{-\frac{\lambda_{i}\phi_{i}(h)}{K}}\right) \right)  m_{K}(x,  h) \bar{m}({\rm d} h) \nu^{K}_{t}({\rm d} x) \nonumber \\
& \hspace*{5mm} +   \int_{\hat{\mathcal{X}}}   K^{2} \left( \int_{\hat{\mathcal{X}}} c_{K}(x,y) \nu^{K}_{t}({\rm d} y) \right)  \left( e^{\frac{\lambda_{i} \phi_{i}(x)}{K}}  -1 \right) \nu^{K}_{t}({\rm d} x) \bigg|.
\end{align}

We now use Assumption \ref{assum3}: We apply condition (iv) to the forth line, (v) and (vii) to the fifth line, and (i) to the last line together with (\ref{keyineq}) which implies that writing  $f( \nu^{K})= \sum_{i=1}^{n} \theta_{i} e^{- \lambda_{i}  \langle \nu^{K}, \phi_{i} \rangle}$ this last line can be bounded by
\begin{align*}
&  \left|\int_{\hat{\mathcal{X}}}   K \left( \int_{\hat{\mathcal{X}}} c(x,y) \nu^{K}_{t}({\rm d} y) \right)  \left( f \left( \nu_{t}^{K}-\frac{1}{K} \delta_x \right)-  f  ( \nu_{t}^{K} )  \right) \nu^{K}_{t}({\rm d} x) \right| \\
  & \hspace*{10mm} \leq K C  \int_{\hat{\mathcal{X}}}  \int_{\hat{\mathcal{X}}}  \left| f \left( \nu_{t}^{K}-\frac{1}{K} \delta_x \right)-  f ( \nu_{t}^{K} )  \right|  \nu^{K}_{t}({\rm d} x)  \nu^{K}_{t}({\rm d} y) \\
 & \hspace*{10mm} \leq  C  \langle \nu^{K}_{t}, \mathbf{1} \rangle,
\end{align*}

\noindent where $C$ is a nonnegative constant (not depending on $K$ and whose value changes from line to line). 
Thus, we obtain
\begin{eqnarray*}
\sup_{t \in [0, T]} |\mathscr{L}^{K}f( \nu_{t}^{K} )| \leq C  \sup_{t \in [0, T]} \langle \nu_{t}^{K}, \mathbf{1} \rangle,
\end{eqnarray*}

\noindent for a positive constant $C$ (that does not depend on $K$). Note that Lemma \ref{lemma1} implies that
\begin{eqnarray*}
\sup_{K} \mathbf{E}^{K} \big[  \sup_{t \in [0, T]} |\mathscr{L}^{K}f( \nu_{t}^{K} )|  \big]  < + \infty
\end{eqnarray*}

\noindent and that \cite[Theorem I.51]{Pro2005} implies also that for each $f \in \mathscr{E}$, the process $M^{K}(f) = (M^{K}_{t}(f), t \in [0,T])$ given by
\begin{eqnarray*}
M^{K}_{t}(f) = f(\nu_{t}^{K}) - f( \nu_{0}^{K}) - \int_{0}^{t} \mathscr{L}^{K}f( \nu_{s}^{K} ) {\rm d} s
\end{eqnarray*}

\noindent is a martingale. It then follows from \cite[Theorems 3.2.2 and 3.9.4]{Ethier1986} with $p = \infty$ and Lemma \ref{lemma1} that the laws of $(f( \nu_{t}^{K}): t \in [0,T])$ under $\mathbf{Q}^{K}$ are tight. This shows (b) and our claim follows from \cite[Theorem 3.1]{Jaku1986}.
\end{proof}

\begin{remark}
Alternative to show tightness of the sequence of laws $(\mathbf{Q}^{K})_{K}$ in $\mathcal{P}(\mathbb{D}([0,T], \mathcal{M}(\hat{\mathcal{X}})))$, it suffices, following \cite[Theorem 3.7.1, Chapter 3]{Dawson1991}, to show that for any continuous function $\phi: \hat{\mathcal{X}} \rightarrow \mathbb{R}$, the sequence of laws of the processes $(\langle \nu_{t}^{K}, \phi \rangle_{t}, t \in [0,T])$ is tight in $\mathbb{D}([0,T], \mathbb{R})$. From the moment estimate in Lemma \ref{lemma1}, we believe that this can be achieved by using Aldous criterion \cite{Aldous1978}. We refer to \cite[Proof of Theorem 5.5.3]{Dawson19922}, where this approach is implemented in a similar setting. 
\end{remark}

\subsection{Identifying the limit} \label{Ident}

Recall that $\mathbf{Q}^{K} = \mathbf{L}(\mu^{K})$ denotes the law of the process $\nu^{K}$ such that $\mathbf{Q}^{K} (\nu_{0}^{K} = \mu^{K})=1$, and denote by $\mathbf{Q}_{\mu}$ a limiting value of the tight sequence $(\mathbf{Q}^{K})_{K}$. Recall also that $\mathbb{D}([0,T], \mathcal{M}(\hat{\mathcal{X}}))$ is a separable space; see for example \cite{Bill1999}. By Skorokhod's representation (see \cite[p. 102]{Ethier1986}), we may assume that the c\`adl\`ag processes $(\nu^{K}_{t}, t \in [0,T])$ and $(\nu_{t}, t \in [0,T])$ with distributions $\mathbf{Q}^{K}$ and $\mathbf{Q}_{\mu}$ respectively are defined on the same probability space and that the sequence $(\nu^{K}_{t}, t \in [0,T])$ converges almost surely to $(\nu_{t}, t \in [0,T])$ on $
\mathbb{D}([0,T], \mathcal{M}(\hat{\mathcal{X}}))$. Define $D(\nu) = \{ t \in [0,T]: \mathbb{P}(\nu_{t} = \nu_{t-}) = 1 \}$. Then the complement in $[0,T]$ of
$D(\nu)$ is at most countable by
\cite[Lemma 3.7.7]{Ethier1986}. It follows from \cite[Proposition 3.5.2]{Ethier1986} that for each $t \in D(\nu)$ we have
$\lim_{K \rightarrow \infty} \nu_{t}^{K} = \nu_{t}$ almost surely.

In this section, we show that the limit point $\mathbf{Q}_{\mu}$ of the sequence $(\mathbf{Q}^{K})_{K}$ satisfies the properties stated in Theorem \ref{teo3} (b).

\paragraph{Proof of the moment bound.} The first moment bound of the limiting process $(\nu_{t}, t \in [0,T])$ follows from condition (\ref{eq13}) together with Fatou's Lemma and Lemma \ref{lemma1} (we have implicitly used that $D(v)$ is at most countable \cite[Lemma 3.7.7]{Ethier1986} and right continuity).

\paragraph{Proof of the martingale property.} Recall from Theorem  \ref{teo3} that
\begin{eqnarray*}
\mathscr {D}(\mathbf{M}) = C_{\partial}(\mathcal{X}, \mathbb{R}_{+})^{+} \cap \mathscr{D}(A),\end{eqnarray*}

\noindent where $\mathscr{D}(A)$ is defined in Assumption \ref{assum3}. Let $\phi \in\mathscr {D}(\mathbf{M})$, Theorem \ref{teo2} (b) implies that $E^{K}( \phi) = (E^{K}_{t}(\phi), t \in [0, T])$ given in equation (\ref{eq30}) is a martingale. More precisely, 
\begin{align} \label{exeq2}
 E^{K}_{t}( \phi) & = \exp(-  \langle \nu_{t}^{K}, \phi \rangle) - \exp(-  \langle \nu_{0}^{K},  \phi \rangle) \nonumber
\\
& \hspace*{3mm} - K\int_{0}^{t} \int_{\hat{\mathcal{X}}}
  \left( b_{K}(x) \left( g^{K}\left(x,x, e^{-\frac{ \phi(x)}{K}} \right) -1 \right)
+   d_{K}(x)\left( e^{\frac{\phi(x)}{K}} -1 \right) \right)
  \exp(-  \langle \nu_{s}^{K},  \phi \rangle) \nu_{s}^{K}({\rm d} x)  {\rm d} s \nonumber  \\
& \hspace*{3mm}  +
K \int_{0}^{t} \int_{\hat{\mathcal{X}}}  b_{K}(x) p(x) \Bigg( \int_{\hat{\mathcal{X}}} \bigg( g^{K} \left(x,x, e^{-\frac{\phi(x)}{K}} \right) - g^{K} \left(x,h, e^{-\frac{\phi(h)}{K}} \right) \bigg) m_{K}(x,h) \bar{m}({\rm d}h)  \Bigg) \nonumber \\
& \hspace*{120mm}  \exp(- \langle \nu_{s}^{K},  \phi \rangle)   \nu_{s}^{K}({\rm d} x) {\rm d} s \nonumber
\\
& \hspace*{3mm} -
 K^{2} \int_{0}^{t} \int_{\hat{\mathcal{X}}}  \left( \int_{\hat{\mathcal{X}}} c_{K}(x,y) \nu_{s}^{K}({\rm d} y) \right) \left( e^{\frac{\phi(x)}{K}} -1 \right) \exp(-  \langle \nu_{s}^{K},  \phi \rangle) \nu_{s}^{K}({\rm d} x)  {\rm d} s,
\end{align}

\noindent for $t \in [0,T]$.  Assumption \ref{assum3} (by applying condition (iv) to the second line, (vii) to the third to fourth line and (i) as well as (\ref{eq8}) to the last line) now shows that if a subsequence $(\nu_{K_n})_{n}$ converges to a $\nu$ then we obtain that   $E^{K_n}_{t}( \phi)$ converges weakly to
$E_{t}( \nu, \phi)$ where
\begin{align} \label{eq7}
 E_{t}( \nu, \phi) & = \exp(- \langle \nu_{t}, \phi \rangle) - \exp(-  \langle \nu_{0},  \phi \rangle)   + \int_{0}^{t} \int_{\hat{\mathcal{X}}}
 b(x) \phi(x)
  \exp(- \langle \nu_{s},  \phi \rangle) \nu_{s}({\rm d} x)  {\rm d} s \nonumber  \\
  & \hspace*{10mm} -  \int_{0}^{t} \int_{\hat{\mathcal{X}}}
  \sigma(x) \phi^{2}(x)
  \exp(-  \langle \nu_{s},  \phi \rangle) \nu_{s}({\rm d} x)  {\rm d} s \nonumber  \\
 & \hspace*{10mm}  - \int_{0}^{t} \int_{\hat{\mathcal{X}}} \left(\int_{0}^{\infty} (e^{-u  \phi(x)}-1+u \phi(x)) \Pi(x, {\rm d}u) \right) \exp(- \langle \nu_{s},  \phi \rangle) \nu_{s}({\rm d}x) {\rm d}s \nonumber  \\
& \hspace*{10mm}  +
  \int_{0}^{t} \int_{\hat{\mathcal{X}}}  p(x)  A\phi(x)   \exp(-  \langle \nu_{s},  \phi \rangle)   \nu_{s}({\rm d} x) {\rm d} s \nonumber
\\
& \hspace*{10mm} -
  \int_{0}^{t} \int_{\hat{\mathcal{X}}}  \left( \int_{\hat{\mathcal{X}}} c(x,y) \nu_{s}({\rm d} y) \right) \phi(x)  \exp(-  \langle \nu_{s},  \phi \rangle) \nu_{s}({\rm d} x)  {\rm d} s.
\end{align}
\noindent
This also suggests that  $E_{t}( \nu, \phi)$ for any limit point $\nu$ of the sequence $(\nu^{K})_{K}$ should be a martingale.
To justify this conclusion we need to know that the martingale property was preserved under passage to the limit.
Therefore, it is enough to check that for each $l \in {\mathbb N}, (s_{j})_{j=1}^{l} \subset D(\nu)$, $s,t \in D(\nu)$ with $0 \leq s_{1} \leq \cdots \leq s_{n} < s < t \leq T$,  some continuous and bounded maps $h_{1}, h_{2}, \dots, h_{l}$ on $\mathcal{M}(\hat{\mathcal{X}})$,
\begin{eqnarray} \label{exeq3}
\mathbf{E} \left[ (E_{t}(\nu, \phi) - E_{s}(\nu, \phi)) h_{1}(\nu_{s_{1}}) \cdots h_{l}(\nu_{s_{l}}) \right] = 0.
\end{eqnarray}

\noindent It follows from Theorem \ref{teo2} (b) that
\begin{eqnarray*}
0 = \mathbf{E} \left[ (E_{t}^{K}( \phi) - E_{s}^{K}( \phi)) h_{1}(\nu_{s_{1}}^{K}) \cdots h_{l}(\nu_{s_{l}}^{K}) \right] = \mathbf{E} \left[ \left(E_{t}(\nu^{K}, \phi) - E_{s}(\nu^{K}, \phi)- R_{K}(t,s) \right) h_{1}(\nu_{s_{1}}^{K}) \cdots h_{l}(\nu_{s_{l}}^{K}) \right],
\end{eqnarray*}

\noindent where
\begin{eqnarray*}
R_{K}(t,s) = E_{t}^{K}( \phi) - E_{s}^{K}( \phi) - E_{t}(\nu^{K}, \phi) + E_{s}(\nu^{K}, \phi).
\end{eqnarray*}
We notice that Assumption \ref{assum3} implies due to $\phi \geq 0$ that 
\begin{align} \label{exeq4}
| E_{t}(\nu^{K}, \phi) - E_{s}(\nu^{K}, \phi) | & \leq e^{- \langle \nu_{t}^{K}, \phi \rangle} + e^{-  \langle \nu_{s}^{K},  \phi \rangle}  +   \int_{s}^{t} e^{- \langle \nu_{u}^{K},  \phi \rangle} \bigg(  \parallel b \parallel_{\infty} 
\langle \nu_{u}^{K}, \phi \rangle   \nonumber  \\
  & \hspace*{10mm} +
  \parallel \sigma \parallel_{\infty} \parallel \phi \parallel_{\infty}
  \langle \nu_{u}^{K}, \phi \rangle \nonumber  +  \left(\sup_{x \in \hat{\mathcal{X}}}\int_{0}^{\infty} (u \wedge u^{2}) \Pi(x, {\rm d}u) \right) \parallel \phi \parallel_{\infty}
  \langle \nu_{u}^{K}, \phi \rangle  \nonumber  \\
& \hspace*{10mm}  +
    \parallel A\phi \parallel_{\infty}  \langle \nu_{u}^{K}, \mathbf{1} \rangle +
 \parallel c \parallel_{\infty} \langle \nu_{u}^{K}, \mathbf{1} \rangle \langle \nu_{u}^{K}, \phi \rangle  \bigg) {\rm d} u.
\end{align}

\noindent Recall that $\phi$ is bounded away from zero such that $ \langle \nu_{u}^{K}, \mathbf{1} \rangle \leq C \langle \nu_{u}^{K}, \phi \rangle$
for some constant $C>0.$ Using that $(x\vee x^2) e^{-x}$ is bounded with $x= \langle \nu_{u}^{K}, \phi \rangle$ we thus obtain
\begin{eqnarray*}
\sup_{s,t \in [0,T]}
| E_{t}(\nu^{K}, \phi) - E_{s}(\nu^{K}, \phi) | \leq C_{1},
\end{eqnarray*}
\noindent for some positive constant $C_{1}$ that does not depend on $K$. By using (\ref{exeq2}) above and a little bit of extra effort (similar computations have been done in (\ref{exeq1})), one can check that
\begin{eqnarray*}
\sup_{s,t \in [0,T]}| R_{K}(s,t) | \leq C_{2},
\end{eqnarray*}

\noindent for some positive constant $C_{2}$ that does not depend on $K$. Now we see that for a subsequence $(\nu_{K_n})_{n}$ that converges to a $\nu$ we have, by the Dominated Convergence Theorem, since all the $h_{j}$ are also bounded that
\begin{eqnarray*}
0 = \lim_{n \rightarrow \infty} \mathbf{E} \left[ (E_{t}^{K_n}( \phi) - E_{s}^{K_n}( \phi)) h_{1}(\nu_{s_{1}}^{K}) \cdots h_{l}(\nu_{s_{l}}^{K}) \right]
=  \mathbf{E} \left[ (E_{t}(\nu, \phi) - E_{s}(\nu, \phi)) h_{1}(\nu_{s_{1}}) \cdots h_{l}(\nu_{s_{l}}) \right]
\end{eqnarray*}
which shows (\ref{exeq3}). Thus, $E(\phi) = (E_{t}(\nu, \phi), t \in [0, T])$ defined in (\ref{eq7}) is a $\mathbf{Q}_{\mu}$-martingale.

Next, we show that $M(\phi)$ defined in (\ref{eq10}) is a martingale with the desired decomposition. Our proof follows similar ideas to those of \cite[Theorem 6.1.3]{Dawson1991} or \cite[Theorem 7.13]{Li2011}. We
show that $Z = (Z_{t}(\phi), t \in [0, T])$ given by $Z_{t}(\phi) := e^{- \langle \nu_{t}, \phi \rangle}$ is a special semimartingale, i.e., it has a representation
\begin{eqnarray*}
Z_{t}(\phi) = Z_{0}( \phi) + W_{t} + V_{t}, \hspace*{5mm} t \in [0,T],
\end{eqnarray*}

\noindent where $W = (W_{t}, t \in [0,T])$ is a local martingale with $W_{0} = 0$, $V = (V_{t}, t \in [0,T])$ is a process of locally bounded variation that has locally integrable variation; see, e.g. \cite[p. 85]{Lip1989}. In the following, we abbreviate
\begin{align*}
I_{t}(\phi) & :=  \langle \nu_{t}, p A\phi \rangle  - \langle \nu_{t}, \psi(\phi) \rangle - \int_{\hat{\mathcal{X}}} \left( \int_{\hat{\mathcal{X}}} c(x,y) \nu_{t}({\rm d} y) \right) \phi(x) \nu_{t}({\rm d} x),
\end{align*}

\noindent where $\psi(\phi) := \psi(x, \phi(x))$ for $x \in \hat{\mathcal{X}}$ with $\psi$ defined as in Assumption \ref{assum3}.
We now consider the processes $Y = (Y_{t}, t \in [0,T])$  and $H = (H_{t}, t \in [0,T])$  given by
\begin{equation*}
Y_{t}(\phi)  := \exp \left(  \int_{0}^{t}  I_{s}(\phi) {\rm d} s \right) \quad \text{ and }  \quad
H_{t}(\phi)  := Z_{t}(\phi) \, Y_{t}(\phi).
\end{equation*}
We note that $E_{t}(\nu, \phi)=  Z_{t}(\phi) - Z_{0}( \phi) + \int_0^t I_{s}(\phi) Z_{s}( \phi)  {\rm d} s.$
Using this and integration by parts together with the fact that $Y(\phi)$ is a process of locally bounded variation  we obtain that
\begin{align*}
\int_{0}^{t} Y_{s}(\phi) {\rm d} E_{s}(\nu, \phi) & = \int_{0}^{t} Y_{s}(\phi)  {\rm d} Z_{s}(\phi)
+ \int_{0}^{t} Y_{s}(\phi) I_{s}(\phi)  Z_{s}(\phi){\rm d} s \\
& = \int_{0}^{t} Y_{s}(\phi)  {\rm d} Z_{s}(\phi)
+ \int_{0}^{t}  Z_{s}(\phi){\rm d} Y_{s}(\phi) \\
& = H_{t}(\phi) - H_{0}(\phi)= H_{t}(\phi) - Z_{0}(\phi)
\end{align*}

\noindent is a $\mathbf{Q}_{\mu}$-local martingale. We have $Z_{t}(\phi) = H_{t}(\phi) Y_{t}(\phi)^{-1}$ with $Y_{t}(\phi)^{-1}$ of locally bounded variation and so, again by integration by parts,
\begin{align} \label{eq7ex}
{\rm d} Z_{t}(\phi) & = Y_{t}(\phi)^{-1} {\rm d} H_{t}(\phi) + H_{t}(\phi) d Y_{t}(\phi)^{-1}  \nonumber \\
& = Y_{t}(\phi)^{-1} {\rm d} H_{t}(\phi) - I_{t}(\phi)Z_{t-}(\phi)  {\rm d} t.
\end{align}

\noindent  Then $Z =(Z_{t}( \phi), t \in [0, T])$ is a special semi-martingale (one can follow a similar estimation  procedure as in (\ref{exeq4}) in order to check that the locally bounded variation term of $Z$ has locally integrable variation). On the other hand, using It\^{o}'s formula \cite[Theorem II.32]{Pro2005} 
 we conclude that $\langle \nu_{\cdot}, \phi \rangle = - \log Z_{\cdot}(\phi)$ is also a semi-martingale. Let $\mathcal{M}_{\pm}(\hat{\mathcal{X}})$ denote the space of signed Borel measures on $\hat{\mathcal{X}}$ endowed with the $\sigma$-algebra generated by the mappings $\mu \mapsto \mu(B)$ for all $B \in \mathbb{B}(\hat{\mathcal{X}})$. Let $\mathcal{M}_{\pm}^{\circ}(\hat{\mathcal{X}}) = \mathcal{M}_{\pm}(\hat{\mathcal{X}}) \setminus \{\mathbf{0} \}$ (where $\mathbf{0}$ denotes the null measure). We
define the optional random measure $N({\rm d}s, {\rm d}\mu)$ on $[0, \infty) \times \mathcal{M}_{\pm}^{\circ}(\hat{\mathcal{X}})$ by
\begin{eqnarray*}
N({\rm d}s, {\rm d}\mu) = \sum_{s  > 0} \mathds{1}_{\{\Delta \nu_{s} \neq \mathbf{0}  \}} \delta_{(s, \Delta \nu_{s})} ({\rm d}s, {\rm d}\mu)
\end{eqnarray*}

\noindent where $\Delta \nu_{s} = \nu_{s} - \nu_{s-} \in \mathcal{M}_{\pm}^{\circ}(\hat{\mathcal{X}})$. Let
$\hat{N}({\rm d}s, {\rm d}\mu)$ denote the predictable compensator of $N({\rm d}s, {\rm d}\mu)$ and let $\tilde{N}({\rm d}s, {\rm d}\mu)$ denote the compensated random measure; see  \cite[p. 172]{Lip1989}. It follows that
\begin{eqnarray} \label{eq11}
\langle \nu_{t}, \phi \rangle = \langle \nu_{0}, \phi \rangle + U_{t}(\phi) + M_{t}^{\text{c}}(\phi) + M_{t}^{\text{d}}(\phi),
\end{eqnarray}

\noindent where $U(\phi)  = (U_{t}(\phi), t  \in [0, T])$ is a predictable process with locally bounded variation,  $M^{\text{c}}(\phi) = (M^{\text{c}}_{t}(\phi), t  \in [0, T])$ is a continuous local martingale with increasing process $C(\phi) = (C_{t}(\phi), t  \in [0, T])$ and $M^{\text{d}}(\phi) = (M^{\text{d}}_{t}(\phi), t  \in [0, T])$ given by
\begin{eqnarray*}
M_{t}^{\text{d}}(\phi) = \int_{0}^{t} \int_{\mathcal{M}_{\pm}^{\circ}(\hat{\mathcal{X}})}  \langle \mu , \phi \rangle \tilde{N}({\rm d}s, {\rm d}\mu), \hspace*{5mm} t \in [0, T],
\end{eqnarray*}

\noindent is a purely discontinuous local martingale; see \cite[p. 85]{JaSh1987}. We apply It\^{o}'s formula \cite[Theorem II.32]{Pro2005} to $\exp(-\langle \nu_{\cdot}, \phi \rangle)$, with $\langle \nu_{\cdot}, \phi \rangle$ given by (\ref{eq11}), and we get that
\begin{align} \label{eq12}
 \bar{E}_{t}(\phi) & =  Z_{t}(\phi) - Z_{0}(\phi)  \nonumber \\
 & \hspace*{4mm} +  \int_{0}^{t} Z_{s-}(\phi) \left( {\rm d} U_{s}(\phi)  - \frac{1}{2} {\rm d} C_{s}(\phi) - \int_{\mathcal{M}_{\pm}(\hat{\mathcal{X}})} \left( e^{-\langle \mu, \phi \rangle} -1 +\langle \mu, \phi \rangle
 \right) N({\rm d}s, {\rm d} \mu)  \right)
\end{align}

\noindent is a local martingale. Note that
\begin{eqnarray*}
0 \leq Z_{s-}(\phi) \left( e^{-\langle \Delta \nu_{s}, \phi \rangle} -1 +\langle \Delta \nu_{s}, \phi \rangle  \right) \leq C \left( |\langle \Delta \nu_{s}, \phi \rangle | \wedge |\langle \Delta \nu_{s}, \phi \rangle^{2} | \right)
\end{eqnarray*}

\noindent for some constant $C \geq 0$. According to Theorem I.4.47 of \cite{JaSh1987}, $\sum_{s \leq t} \langle \Delta \nu_{s}, \phi \rangle^{2} < \infty$. Thus the second term in (\ref{eq12}) has finite variation over each finite interval $[0, T]$. Since $Z$ is a special semi-martingale, Proposition I.4.23 of \cite{JaSh1987} implies that
\begin{eqnarray*}
\int_{0}^{t}  Z_{s-}(\phi) \left( e^{-\langle \mu, \phi \rangle} -1 +\langle \mu, \phi \rangle  \right) N({\rm d}s, {\rm d} \mu)
\end{eqnarray*}

\noindent is of locally integrable variation. Thus it is locally integrable. According to Proposition II.1.28 of \cite{JaSh1987},
\begin{eqnarray*}
\int_{0}^{t}  Z_{s-}(\phi) \left( e^{-\langle \mu, \phi \rangle} -1 +\langle \mu, \phi \rangle  \right) \tilde{N}({\rm d}s, {\rm d} \mu)
\end{eqnarray*}

\noindent is a purely discontinuous local martingale. Therefore,
\begin{align} \label{eq12ex}
 \bar{E}_{t}^{\prime}(\phi) & =  Z_{t}(\phi) - Z_{0}(\phi)  \nonumber \\
 & \hspace*{4mm} +  \int_{0}^{t} Z_{s-}(\phi) \left( {\rm d} U_{s}(\phi)  - \frac{1}{2} {\rm d} C_{s}(\phi) - \int_{\mathcal{M}_{\pm}(\hat{\mathcal{X}})} \left( e^{-\langle \mu, \phi \rangle} -1 +\langle \mu, \phi \rangle
 \right) \hat{N}({\rm d}s, {\rm d} \mu)  \right)
\end{align}

\noindent is a local martingale. The uniqueness of canonical decomposition of special semi-martingales (see, e.g. \cite[p. 85]{Lip1989}) allows us to identify the predictable components of locally integrable variation in the two decompositions (\ref{eq7ex}) and (\ref{eq12ex}) to obtain that
\begin{align*}
- Z_{t-}(\phi) I_{t}(\phi){\rm d}s & = Z_{t-}(\phi) \left( -{\rm d} U_{t}(\phi)  + \frac{1}{2} {\rm d} C_{t}(\phi) + \int_{\mathcal{M}_{\pm}(\hat{\mathcal{X}})} \left( e^{-\langle \mu, \phi \rangle} -1 +\langle \mu, \phi \rangle
 \right) \hat{N}({\rm d}t, {\rm d} \mu)  \right).
\end{align*}


\noindent Then
\begin{align} \label{eq300}
- I_{t}(\phi)  & = - U_{t}(\phi)  + \frac{1}{2}  C_{t}(\phi) + \int_{0}^{t}\int_{\mathcal{M}_{\pm}(\hat{\mathcal{X}})} \left( e^{-\langle \mu, \phi \rangle} -1 +\langle \mu, \phi \rangle
 \right) \hat{N}({\rm d}s, {\rm d} \mu).
\end{align}

\noindent It is not difficult to deduce that $U_{t}(\lambda \phi) = \lambda U_{t}(\phi)$ and $C_{t}(\lambda \phi) = \lambda^{2} C_{t}(\phi)$, for $\lambda \in \mathbb{R}_{+}$. Replacing $\phi$ by $\lambda \phi$ in (\ref{eq300}), we have
\begin{align*}
- I_{t}(\lambda \phi) & = - \lambda U_{t}(\phi)  + \frac{\lambda^{2}}{2}  C_{t}(\phi) + \int_{0}^{t}\int_{\mathcal{M}_{\pm}(\hat{\mathcal{X}})} \left( e^{- \lambda \langle \mu, \phi \rangle} -1 + \lambda\langle \mu, \phi \rangle
 \right) \hat{N}({\rm d}s, {\rm d} \mu).
\end{align*}

\noindent This allows to conclude (in the semimartingale representation of $\langle \nu_{t}, \lambda \phi \rangle$):
\begin{eqnarray} \label{eq19}
C_{t}(\phi) = 2\int_{0}^{t} \langle  \nu_{s}, \sigma \phi^{2} \rangle {\rm d}s,
\end{eqnarray}
\begin{align*}
U_{t}(\phi) & =  \int_{0}^{t} \int_{\hat{\mathcal{X}}} \phi(x) \left( b(x) - \int_{\hat{\mathcal{X}}} c(x,y) \nu_{s}({\rm d}y) + p(x)  A\phi(x)  \right) \nu_{s}({\rm d}x){\rm d} s  
\end{align*}
\noindent and
\begin{eqnarray*}
\int_{0}^{t}  \int_{\mathcal{M}_{\pm}(\hat{\mathcal{X}})} \left( e^{-\langle \mu, \phi \rangle} -1 +\langle \mu, \phi \rangle
 \right) \hat{N}({\rm d}s, {\rm d} \mu) = \int_{0}^{t} \int_{\hat{\mathcal{X}}} \int_{0}^{\infty} (e^{-u \phi(x)}-1+u\phi(x)) \Pi(x, {\rm d}u) \nu_{s}({\rm d}x) {\rm d}s.
\end{eqnarray*}

\noindent That is, the jump measure of the process $\nu$ has compensator given by (\ref{eq:n}). In particular, this implies that the jumps of the $\nu$ are almost surely in $\mathcal{M}(\hat{\mathcal{X}})$.

Finally, from the identity (\ref{eq11}), we observe that $M(\phi)= M^{\text{c}}(\phi) + M^{\text{d}}(\phi)$. Therefore, it is enough to show that $M^{\text{c}}(\phi)$ and  $M^{\text{d}}(\phi)$ are actually martingales to conclude that $M(\phi)$ defined in (\ref{eq10}) is a martingale. Following the argument in Section 2.3 of \cite{LeMy2005} we obtain the martingale property of $M^{\text{d}}(\phi)$. We consider $M^{\text{d,1}}(\phi) = (M^{\text{d,1}}_{t}(\phi), t \in [0,T])$ and $M^{\text{d,2}}(\phi) = (M^{\text{d,2}}_{t}(\phi), t \in [0,T])$ given by
\begin{eqnarray*}
M^{\text{d,1}}_{t}(\phi) = \int_{0}^{t} \int_{\mathcal{M}(\hat{\mathcal{X}})}  \langle \mu , \phi \rangle \mathds{1}_{\{\langle \mu , \phi \rangle  \geq 1 \}} \tilde{N}({\rm d}s, {\rm d}\mu) \hspace*{5mm} \text{and} \hspace*{5mm} M^{\text{d,2}}_{t}(\phi) = \int_{0}^{t} \int_{\mathcal{M}(\hat{\mathcal{X}})}  \langle \mu , \phi \rangle  \mathds{1}_{\{\langle \mu , \phi \rangle  < 1 \}} \tilde{N}({\rm d}s, {\rm d}\mu),
\end{eqnarray*}

\noindent for $t \in [0,T]$. We observe that $M^{\text{d}}(\phi) = M^{\text{d,1}}(\phi) + M^{\text{d,2}}(\phi)$ and that
\begin{eqnarray*}
\mathbf{E}\left[ \int_{0}^{t} \int_{\mathcal{M}(\hat{\mathcal{X}})}  \langle \mu , \phi \rangle \mathds{1}_{\{\langle \mu , \phi \rangle  \geq 1 \}} \hat{N}({\rm d}s, {\rm d}\mu) \right] < \infty \hspace*{5mm} \text{and} \hspace*{5mm} \mathbf{E}\left[ \int_{0}^{t} \int_{\mathcal{M}(\hat{\mathcal{X}})}  \langle \mu , \phi \rangle^{2}  \mathds{1}_{\{\langle \mu , \phi \rangle  < 1 \}} \hat{N}({\rm d}s, {\rm d}\mu) \right] <  \infty
\end{eqnarray*}

\noindent due to Assumption \ref{assum3} (iv) for $\phi \in\mathscr {D}(\mathbf{M})$ and where $\hat{N}$ is given by (\ref{eq:n}). Hence, Proposition II.1.28 and Theorem II.1.33 in \cite{JaSh1987} show that $M^{\text{d,1}}(\phi)$ is a martingale and $M^{\text{d,2}}(\phi)$ is a square-integrable martingale with quadratic variation process given by
\begin{eqnarray*}
\left[ M^{\text{d,2}}(\phi) \right]_{t} = \int_{0}^{t} \int_{\mathcal{M}(\hat{\mathcal{X}})}  \langle \mu , \phi \rangle^{2}  \mathds{1}_{\{\langle \mu , \phi \rangle  < 1 \}} \hat{N}({\rm d}s, {\rm d}\mu),
\end{eqnarray*}

\noindent which implies that $M^{\text{d}}(\phi)$ is a martingale. On the other hand, recall that the continuous local martingale $M^{\text{c}}(\phi)$ possesses an increasing process $C(\phi)$ given by (\ref{eq19}), and  such that
\begin{eqnarray*}
\mathbf{E}[C_{t}(\phi)] = \mathbb{E}\left[2\int_{0}^{t} \int_{\hat{\mathcal{X}}} \sigma(x) \phi^{2}(x) \nu_{s}({\rm d}x) {\rm d}s \right] < \infty
\end{eqnarray*}

\noindent by the moment property in Theorem \ref{teo3} (b). Hence Corollary 1.25 in \cite{ReYo1999} implies that $M^{\text{c}}(\phi)$ is a square-integrable martingale. This conclude the proof of Theorem \ref{teo3} (b).

\section{Proof of Theorem \ref{teo4}} \label{Uni}

In this subsection, we will prove that uniqueness holds for solutions of the martingale problem (\ref{eq10}). Our approach is based on the use of a Girsanov type transform and the localization method introduced by Stroock \cite{Stroock1975} in the measure-valued context (see He \cite{He2009}). More precisely, we first introduce in Section \ref{Uni1} the ``killed'' martingale problem associated with the martingale problem (\ref{eq10}). The ``killed'' martingale problem may be seen as the martingale problem (\ref{eq10}) where the randomness is eliminated from the big jumps. Secondly, we develop a Girsanov type theorem in Section \ref{Uni2} for the ``killed'' martingale problem in order to get rid of the non-linearities (caused by the competition) which allows us to deduce uniqueness for the ``killed'' martingale problem. Finally, we develop a localization argument to show that uniqueness of the ``killed'' martingale problem implies uniqueness for the martingale problem (\ref{eq10}). In this section, we always assume that Assumption \ref{assum1} and \ref{assum3} are fulfilled.

It is important to mention that the use of Girsanov type transforms was first applied in the measure-valued diffusions setting by Dawson \cite[Section 5]{Da1978} (or \cite[Theorem 2.3]{Evans1994}). However, in our case Dawson's Girsanov Theorem is not applicable since the measure-valued process possesses jumps. Thus, we shall extend Dawson's result to our setting.

\subsection{The killed martingale problem} \label{Uni1}

In this section, we introduce the killed martingale problem.

The measure-valued process  $\nu^{\prime} \in \mathbb{D}([0,T], \mathcal{M}(\hat{\mathcal{X}}))$, or equivalently its law $\mathbf{Q}^{\prime}_{\mu}$, solves the killed martingale problem at the fixed level $1 < l < + \infty$ if for $\mu \in
\mathcal{M}(\mathcal{X})$, $\mathbf{Q}^{\prime}_{\mu}(\nu_{0}^{\prime} = \mu)=1$,  and for any $\phi \in \mathscr {D}(\mathbf{M})$, the process $M^{\prime}(\phi) = (M_{t}^{\prime}(\phi), t \in [0,T])$ given by
\begin{align} \label{eq20}
M_{t}^{\prime}(\phi) & =  \langle \nu_{t}^{\prime}, \phi \rangle - \langle \nu_{0}^{\prime}, \phi \rangle - \int_{0}^{t} \int_{\hat{\mathcal{X}}}  p(x) A\phi(x)  \nu_{s}^{\prime}({\rm d} x) {\rm d} s \nonumber \\
& \hspace*{5mm} - \int_{0}^{t} \int_{\hat{\mathcal{X}}}  \phi(x)  \left( b(x) - \int_{[l, \infty)} u \Pi(x, {\rm d} u) - \int_{\hat{\mathcal{X}}} c(x,y) \nu_{s}^{\prime}({\rm d}y) \right) \nu_{s}^{\prime}({\rm d}x){\rm d} s  \tag{$\mathbf{M}^{\prime}$}
\end{align}

\noindent is a $\mathbf{Q}^{\prime}_{\mu}$-martingale that admits the decomposition $M^{\prime}(\phi) = M^{\text{c}^{\prime}}(\phi) + M^{\text{d}^{\prime}}(\phi)$, where $M^{\text{c}^{\prime}}(\phi)$ is a continuous martingale with increasing process as in (\ref{eq:quadrativ-var})
%
\noindent
and $M^{{\rm d}^{\prime}}(\phi)$ is a purely discontinuous martingale, defined as in ({\ref{eq:discrete_mg})  by
\begin{eqnarray} \label{revision2}
M^{{\rm d}^{\prime}}_{t}(\phi) = \int_{0}^{t} \int_{\mathcal{M}(\hat{\mathcal{X}})}  \langle \mu , \phi \rangle \tilde{N}^{\prime}({\rm d}s, {\rm d}\mu).
\end{eqnarray}
where $\tilde{N}^{\prime}$ is the compensated random measure of an optional random measure $N^{\prime}$ defined as in (\ref{eq:N}) and where $\tilde{N}^{\prime}$ is characterized analogously to (\ref{eq:n}). The only difference is that in the definition of the compensator 
\begin{eqnarray}\label{eq:n'}
\hat{N}^{\prime}({\rm d}s, {\rm d} \mu) = {\rm ds} \, \hat{n}^{\prime}(\nu_{s}^{\prime}, {\rm d} \mu)
\end{eqnarray}
\noindent in  equation (\ref{eq:n}) is replaced by
\begin{eqnarray*}
\int_{\mathcal{M}(\hat{\mathcal{X}})} f(\mu) \hat{n}^{\prime}(\nu_{s}^{\prime}, {\rm d} \mu) = \int_{\hat{\mathcal{X}}} \int_{(0,l)} f(u \delta_{x}) \Pi(x, {\rm d}u) \nu_{s}^{\prime}({\rm d}x) \hspace*{5mm} \text{for} \hspace*{5mm} f \in B(\mathcal{M}(\hat{\mathcal{X}}), \mathbb{R}).
\end{eqnarray*}
Let us note that for the rest of the section $1 < l < \infty$ is a fixed number and we  will thus not specifically mention 
the dependence on $l.$ For example, we will refer  to  the killed martingale problem at level $l$ simply as the (\ref{eq20}) martingale problem.
Heuristically, we observe from the decomposition of the martingale (\ref{eq20}) that the randomness of ``big''   
jumps of size  larger than $l$ are suppressed by replacing them by the immigration term added in the drift. The latter yields
an additional drift in the compensator of such ``big jumps'' in 
(\ref{eq20}).

It is important to point out that the killed martingale problem also arises as a limit of a sequence of interacting particle systems as described in Section \ref{model}, where Assumption \ref{assum3} (iv) is satisfied with
 \begin{eqnarray*}
\psi^{\prime}(x, z) =  \left(-b(x) + \int_{[l, \infty)} u \Pi(x, {\rm d} u) \right)z + \sigma(x)z^{2} + \int_{(0, l)} (e^{-zu} -1 +zu) \Pi(x, {\rm d}u), \hspace*{5mm} x \in \hat{\mathcal{X}} \hspace*{5mm} \text{and} \hspace*{5mm} z \geq 0;
\end{eqnarray*}

\noindent see proof of Theorem \ref{teo3}.
We now show that each solution of the killed martingale problem has bounded moments of any order.  We write $\mathbf{E}^{\prime}_{\mu}$ for the expectation with respect $\mathbf{Q}^{\prime}_{\mu}$.

\begin{lemma} \label{lemma2}
Suppose that the $\mathcal{M}(\hat{\mathcal{X}})$-valued c\`adl\`ag process $\nu^{\prime} = (\nu_{t}^{\prime}, t \in [0,T])$ with law $\mathbf{Q}^{\prime}_{\mu}$ is a solution of the killed martingale problem. Then for $n \geq 1$,  we have
\begin{eqnarray*}
\mathbf{E}^{\prime}_{\mu}[\langle \nu_{t}^{\prime}, \mathbf{1} \rangle^{n}]  \leq   (\langle \mu, \mathbf{1} \rangle^{n} +1)  e^{C_{n}t}-1, \hspace*{5mm} t \in [0,T],
\end{eqnarray*}

\noindent where $0 < C_{n}< + \infty$ is a constant which depends on $n$.
\end{lemma}

\begin{proof}
We stress that the value of the non-negative constants $C_{n}$ appearing in the proof may change from line to line. Moreover, $C_{n}$ denotes a constant depending only on $n$.   Let $\tau_{m} =\inf \{ t \in [0,T]: \langle \nu^{\prime}_{t}, \mathbf{1} \rangle \geq m \}$, for $m \geq 1$, and note that $\tau_{m} \rightarrow \infty$ as $m \rightarrow \infty$ 
$\mathbf{Q}^{\prime}_{\mu}$-almost surely. We also observe that $(M_{t \wedge \tau_{m}}^{\prime}(\phi), t \in [0,T])$ is a $\mathbf{Q}^{\prime}_{\mu}$-martingale.  Setting $\phi \equiv 1$, It\^{o}'s formula implies that for $n \geq 1$ and $t \in [0,T],$
\begin{align*}
& \mathbf{E}_{\mu}^{\prime} \left[ \langle \nu_{t \wedge \tau_{m}}^{\prime}, \mathbf{1} \rangle^{n} \right ] 
\leq \langle \mu, \mathbf{1} \rangle^{n} +
n \mathbf{E}_{\mu}^{\prime} \left[  \int_{0}^{t} \int_{\hat{\mathcal{X}}}  \left( b(x) - \int_{[l, \infty)} u \Pi(x, {\rm d} u) \right)  \langle \nu_{s \wedge \tau_{m}}^{\prime}, \mathbf{1} \rangle^{n-1} \nu_{s \wedge \tau_{m}}^{\prime}({\rm d}x){\rm d} s \right] \\
& \hspace*{5mm} +   n(n-1) \mathbf{E}_{\mu}^{\prime}\left[  \int_{0}^{t} \int_{\hat{\mathcal{X}}}  \sigma(x)  \langle \nu_{s \wedge \tau_{m}}^{\prime}, \mathbf{1} \rangle^{n-2} \nu_{s \wedge \tau_{m}}^{\prime}({\rm d}x) {\rm d} s \right] \\
& \hspace*{5mm} + \mathbf{E}_{\mu}^{\prime} \left[  \int_{0}^{t}  \int_{\hat{\mathcal{X}}}
\int_{(0, l)} \Big( (\langle \nu_{s \wedge \tau_{m}}^{\prime}, \mathbf{1} \rangle + u)^{n} - \langle \nu_{s \wedge \tau_{m}}^{\prime}, \mathbf{1} \rangle^{n}  - n u \langle \nu_{s \wedge \tau_{m}}^{\prime}, \mathbf{1} \rangle^{n-1} \Big) \Pi(x, {\rm d} u) \nu_{s \wedge \tau_{m}}^{\prime}({\rm d}x)  {\rm d} s \right],
\end{align*}

\noindent  where we have used that $A\phi \equiv 0$ and omitted the negative competition term on the right hand side. Since $b \in C_{\partial}(\mathcal{X}, \mathbb{R})$, $\sigma \in C_{\partial}(\mathcal{X}, \mathbb{R}_{+})^{+}$, $c \in C_{\partial}(\mathcal{X} \times \mathcal{X}, \mathbb{R}_{+})$ and $\sup_{x \in \hat{\mathcal{X}}} \int_{0}^{\infty} (u \wedge u^{2}) \Pi(x, {\rm d} u) < + \infty$ (by Assumptions \ref{assum3}), we conclude with the  binomial formula 
that there is a positive constant $C_{n}$ such that
\begin{eqnarray*}
\mathbf{E}^{\prime}_{\mu}[\langle \nu_{t \wedge \tau_{m}}^{\prime}, \mathbf{1} \rangle^{n}] & \leq  & \langle \mu, \mathbf{1} \rangle^{n} + C_{n} \left(  \sum_{k=1}^{n} \int_{0}^{t} \mathbf{E}^{\prime}_{\mu}\left[ \langle \nu_{s \wedge \tau_{m}}^{\prime}, \mathbf{1} \rangle^{k} \right] {\rm d} s \right) \\
& \leq & \langle \mu, \mathbf{1} \rangle^{n} + C_{n} \int_{0}^{t} \left(1 + \mathbf{E}^{\prime}_{\mu}\left[ \langle \nu_{s \wedge \tau_{m}}^{\prime}, \mathbf{1} \rangle^{n} \right]  \right){\rm d} s ,
\end{eqnarray*}

\noindent for $t \in [0,T]$ . Then, the moment estimate follows by Gronwall's Lemma, first for $\langle \nu_{t \wedge \tau_{m}}^{\prime}, \mathbf{1} \rangle$ and by letting $m \rightarrow \infty$ and using monotone convergence for $\langle \nu_{t}^{\prime}, \mathbf{1} \rangle$.
\end{proof}

\subsection{Dawson's Girsanov type Theorem} \label{Uni2}

We next develop a Dawson's Girsanov type theorem. Recall that for each $0 < T  < +\infty$, the process $\nu^{\prime}=(\nu_{t}^{\prime}, t \in [0, T]) \in \mathbb{D}([0,T], \mathcal{M}(\hat{\mathcal{X}}))$, with law $\mathbf{Q}_{\mu}^{\prime}$, denotes a solution of the (\ref{eq20}) martingale problem. Informally, we want to find a measure $\mathbb{Q}_{\mu}$ under which, for all $\phi \in \mathscr {D}(\mathbf{M})$,
\begin{eqnarray*}
M_{t}^{\prime}(\phi) - \int_{0}^{t} \int_{\hat{\mathcal{X}}}   \int_{\hat{\mathcal{X}}} \phi(x) c(x,y) \nu_{s}^{\prime}({\rm d}y) \nu_{s}^{\prime}({\rm d}x){\rm d} s
\end{eqnarray*}

\noindent is a $\mathbb{Q}_{\mu}$-martingale.

To achieve this we will use the fact that the continuous part $M^{\text{c}^{\prime}}(\phi) = (M^{\text{c}^{\prime}}_{t}(\phi), t \in [0, T])$ of the martingale $M^{\prime}(\phi)$ can be expressed as an integral with respect to an orthogonal martingale measure (see \cite[Section 7.1]{Da1978}). As in Walsh \cite[Chapter 2]{Wa1984}, we write
\begin{eqnarray*}
M^{ \text{c}^{\prime}}_{t}(\phi) = \int_{0}^{t} \int_{\hat{\mathcal{X}}}  \phi(x) W({\rm d}s, {\rm d}x), \hspace*{5mm} t \geq 0,
\end{eqnarray*}

\noindent where $W({\rm d}s, {\rm d}x)$ is an orthogonal continuous martingale measure with covariance given by
\begin{eqnarray*}
{\rm d}[W({\rm d}x), W({\rm d}y)]_{t} = R(\nu_{t}^{\prime}, {\rm d}x, {\rm d}y){\rm d}t,
\end{eqnarray*}

\noindent and $R$ is defined by $R(\mu, {\rm d}x, {\rm d}y) = 2 \sigma(x) \delta_{y}({\rm d}x) \mu({\rm d}y)$, for $\mu \in \mathcal{M}(\hat{\mathcal{X}})$.

We consider the continuous local martingale $L = (L_{t}, t \in [0, T])$ given by
\begin{eqnarray} \label{eq14}
L_{t} = \int_{0}^{t} \int_{\hat{\mathcal{X}}}
a( \nu_{s}^{\prime},x) W({\rm d}s, {\rm d}x),
\end{eqnarray}
\noindent where
\begin{eqnarray*}
a( \nu_{s}^{\prime},x) = \left( \int_{\hat{\mathcal{X}}}  c(x,y) \nu_{s}^{\prime}({\rm d}y) \right) (2 \sigma(x))^{-1} \hspace*{5mm} \text{for} \hspace*{5mm} x \in \hat{\mathcal{X}}
\end{eqnarray*}
\noindent Recall here that the function $\sigma$ is bounded away from zero and that the competition kernel $c(x,y)$ is bounded such that
\begin{eqnarray} \label{revision1}
a( \nu_{s}^{\prime},x) \leq C \langle \nu_{s}^{\prime}, \mathbf{1} \rangle \hspace*{5mm} \text{uniformly over} \hspace*{3mm} x \in \hat{\mathcal{X}}. 
\end{eqnarray}

\noindent Then, the stochastic linear equation
\begin{eqnarray*}
\mathfrak{z}_{t} = 1 + \int_{0}^{t} \mathfrak{z}_{s}
{\rm d} L_s
\end{eqnarray*}

\noindent has a unique nonnegative solution (see for example \cite{Do1970}) known as the Dol\'ean-Dade exponential,
\begin{eqnarray*}
\mathfrak{z}_{t} = \exp \left( \int_{0}^{t} \int_{\hat{\mathcal{X}}} a( \nu_{s}^{\prime}, x) W({\rm d}s, {\rm d}x) -  \frac{1}{2}\int_{0}^{t} \int_{\hat{\mathcal{X}}} \int_{\hat{\mathcal{X}}}  a(\nu_{s}^{\prime}, x) a(\nu_{s}^{\prime}, y) R(\nu_{s}^{\prime}, {\rm d} x, {\rm d} y){\rm d} s \right).
\end{eqnarray*}

\noindent It is well-known that $\mathfrak{z}$ is a nonnegative local martingale (see \cite{Do1970}), and therefore it is a supermartingale with $\mathbb{E}[\mathfrak{z}_{t}] \leq 1$ for any $t \geq 0$. Moreover, if there exists $T >0$ such that $\mathbb{E}[\mathfrak{z}_{T}] = 1$ then $\mathfrak{z} = (\mathfrak{z}_{t}, t \in [0,T])$ is a martingale. The martingale property plays an important role in many applications. In particular,
$\mathfrak{z}_{T}$ usually plays the role of the Random-Nykodym derivative of one probability measure with respect to another, and thus, this will allow us to
generalize Dawson's Girsanov Theorem \cite{Da1978} in our setting.

\begin{theorem} \label{teo5}
Under Assumption \ref{assum3} such that $\sigma \in C_{\partial}(\mathcal{X}, \mathbb{R}_{+})^{+}$, we have that
\begin{itemize}
\item[(a)] The process $\mathfrak{z} = (\mathfrak{z}_{t}, t \in [0,T])$ is a martingale for any $T >0$.

\item[(b)] Moreover, under the probability measure $\mathbb{Q}_{\mu}$ defined by
\begin{eqnarray*}
\frac{{\rm d} \mathbb{Q}_{\mu}}{{\rm d} \mathbf{Q}_{\mu}^{\prime}} = \mathfrak{z}_{T},
\end{eqnarray*}

\noindent the process $\nu^{\prime}=(\nu_{t}^{\prime}, t \in [0, T])$ solves the following martingale problem: For $\mu \in \mathcal{M}(\mathcal{X})$, \\
$\mathbb{Q}_{\mu}(\nu_{0}^{\prime} = \mu)=1$ and for any $\phi \in \mathscr {D}(\mathbf{M})$, the process $\bar{M}(\phi) = (\bar{M}_{t}(\phi), t \in [0,T])$ given by
\begin{eqnarray*}
\bar{M}_{t}(\phi) = M_{t}^{\prime}(\phi) - \int_{0}^{t} \int_{\hat{\mathcal{X}}}   \int_{\hat{\mathcal{X}}} \phi(x) c(x,y) \nu_{s}^{\prime}({\rm d}y) \nu_{s}^{\prime}({\rm d}x){\rm d} s
\end{eqnarray*}

\noindent is a $\mathbb{Q}_{\mu}$-martingale.

\item[(c)] Furthermore, $\bar{M}(\phi)$ admits the decomposition $\bar{M}(\phi) = \bar{M}^{\text{c}}(\phi) + \bar{M}^{\text{d}}(\phi)$, where $\bar{M}^{\text{c}}(\phi)$ is a continuous martingale with increasing process as in (\ref{eq:quadrativ-var}) and
%
$\bar{M}^{{\rm d}}(\phi)$ is a purely discontinuous martingale defined as in (\ref{eq:discrete_mg}). That is, $\bar{M}^{\text{c}}(\phi)$ has increasing process
\begin{eqnarray*}
 2 \int_{0}^{t} \int_{\hat{\mathcal{X}}} \sigma(x) \phi^{2}(x) \nu_{s}^{\prime}({\rm d}x) {\rm d}s,
\end{eqnarray*}
\noindent and $\bar{M}^{{\rm d}}(\phi)$ has optional random measure given by
$ \bar{N}({\rm d}s, {\rm d}\mu) :=  N^{\prime}({\rm d}s, {\rm d}\mu)$ on $[0, \infty) \times \mathcal{M}(\hat{\mathcal{X}})$, where its compensator and compensated random measure are
\begin{eqnarray*}
\hat{\bar{N}}({\rm d}s, {\rm d} \mu) :=  \hat{N}^{\prime}({\rm d}s, {\rm d} \mu) \hspace*{5mm} \text{and} \hspace*{5mm} \tilde{\bar{N}}({\rm d}s, {\rm d} \mu) := \tilde{N}^{\prime}({\rm d}s, {\rm d} \mu),
\end{eqnarray*}
\noindent respectively, with $N^{\prime}$, $\hat{N}^{\prime}$ and $\tilde{N}^{\prime}$ defined in  (\ref{revision2}) and (\ref{eq:n'}).
%
%
%
%
\end{itemize}
\end{theorem}

\begin{proof}
First, we prove point (a). We fix $T>0$ and let $\tau_{n} = \inf \left\{ t \in [0,T]:  \langle \nu_{t}^{\prime}, \mathbf{1} \rangle  \geq n   \right \}$
for $n \geq 1$. Note that $\tau_{n} \rightarrow \infty$ as $n \rightarrow \infty$,
$\mathbf{Q}_{\mu}'$-almost surely.
We now note that by Assumption \ref{assum3} and (\ref{revision1}) we have that
\begin{eqnarray*}
\mathbb{E} \left[\exp \left( \frac{1}{2} \int_{0}^{T \wedge \tau_{n}} \int_{\hat{\mathcal{X}}} \sigma(x) a(\nu_{s}^{\prime}, x)^{2}  \nu_{s}^{\prime}({\rm d} x) {\rm d} s \right) \right] < \infty.
\end{eqnarray*}
This allows us to conclude by the so-called Novikov condition \cite{Nov1973}
 that $(\mathfrak{z}_{(t \wedge \tau_{n})}, t \in [0, T])$ is a $\mathbf{Q}_{\mu}'$-martingale.

We write $\mathbf{E}_{\mu}^{\prime}$ and $\mathbb{E}_{\mu}$ for the expectation with respect $\mathbf{Q}_{\mu}^{\prime}$ and $\mathbb{Q}_{\mu}$. Our aim is to show that $\mathbf{E}_{\mu}^{\prime}[\mathfrak{z}_{T}] = 1$ for which it is enough to prove that  $ \mathbb{Q}_{\mu}( \tau_{n} \leq T) \rightarrow 0$ as $n \rightarrow \infty$. By \cite[Theorem III.39, p. 134]{Pro2005}, we get that under $\mathbb{Q}_{\mu}$, the process $(\bar{M}_{t \wedge \tau_{n}}(\phi), t \in [0,T])$ is a martingale. By the same argument as in the proof of Lemma \ref{lemma2}, we have that
 \begin{eqnarray} \label{eq15}
\mathbb{E}_{\mu} \left[ \langle \nu_{t \wedge \tau_{n}}^{\prime}, \mathbf{1} \rangle^{m} \right] \leq   ( \langle \mu, \mathbf{1} \rangle^{m} +1)   e^{C_{m}t}-1, \hspace*{5mm} m \geq 1,
\end{eqnarray}

 \noindent for $t \in [0,T]$. This implies that  $(\bar{M}_{t \wedge \tau_{n}}(\phi), t \in [0,T])$ is a square integrable $\mathbb{Q}_{\mu}$-martingale. \\

\noindent {\bf Claim.} $\bar{M}_{\cdot \wedge \tau_{n}}(\phi)$ can under $\mathbb{Q}_{\mu}$ furthermore  be decomposed into a continuous and purely discontinuous martingale $\bar{M}_{\cdot \wedge \tau_{n}}(\phi) = \bar{M}^{\text{c}}_{\cdot \wedge \tau_{n}}(\phi) + \bar{M}^{\text{d}}_{\cdot \wedge \tau_{n}}(\phi)$ which are up to the stopping time $\tau_{n}$ characterized as in (c) of Theorem \ref{teo5}. That is, $\bar{M}^{\text{c}}_{\cdot \wedge \tau_{n}}(\phi)$ has increasing process
\begin{eqnarray*}
 2 \int_{0}^{t \wedge \tau_{n}} \int_{\hat{\mathcal{X}}} \sigma(x) \phi^{2}(x) \nu_{s}^{\prime}({\rm d}x) {\rm d}s,
\end{eqnarray*}

\noindent and $\bar{M}^{{\rm d}}_{\cdot \wedge \tau_{n}}(\phi)$ has optional random measure given by
$ \bar{N}_{n}({\rm d}s, {\rm d}\mu) := \mathds{1}_{\{ s \leq \tau_{n} \}} \bar{N}({\rm d}s, {\rm d}\mu)$
on $[0, \infty) \times \mathcal{M}(\hat{\mathcal{X}})$, where
its compensator and compensated random measure are
\begin{eqnarray*}
\hat{\bar{N}}_{n}({\rm d}s, {\rm d} \mu) := \mathds{1}_{\{ s \leq \tau_{n} \}} \hat{\bar{N}}({\rm d}s, {\rm d} \mu) \hspace*{5mm} \text{and} \hspace*{5mm} \tilde{\bar{N}}_{n}({\rm d}s, {\rm d} \mu) := \mathds{1}_{\{ s \leq \tau_{n} \}} \tilde{\bar{N}}({\rm d}s, {\rm d} \mu),
\end{eqnarray*}

\noindent respectively. We postpone the proof of this claim 
to the end of the proof.

Recall that we want to show that $ \mathbb{Q}_{\mu}( \tau_{n} \leq T) \rightarrow 0$ as $n \rightarrow \infty$. For this we take $m =1$ in the inequality (\ref{eq15}) and we integrate to obtain that
  \begin{eqnarray} \label{eq17}
\int_{0}^{T} \mathbb{E}_{\mu} \left[ \langle \nu_{t \wedge \tau_{n}}^{\prime}, \mathbf{1} \rangle \right]  {\rm d} t \leq   \frac{1}{C_{1}} (\langle \mu, \mathbf{1} \rangle +1)  (e^{C_{1}T}-1) - T.
\end{eqnarray}

\noindent The Burkholder-Davis-Gundy inequality applied to the $\mathbb{Q}_{\mu}$-martingale $(\bar{M}_{t \wedge \tau_{n}}(\mathbf{1}), t \in [0,T])$ gives that there is a constant $C^{\prime} > 0$ (the value of $C^{\prime}$ changing from line to line) such that
\begin{eqnarray*}
\mathbb{E}_{\mu}  \left[ \left( \sup_{t \in [0,T]} \bar{M}_{t \wedge \tau_{n}}(\mathbf{1}) \right)^{2} \right] \leq C^{\prime} \mathbb{E}_{\mu}\left[ \langle\bar{M}(\mathbf{1})\rangle_{T \wedge \tau_{n}} \right].
\end{eqnarray*}
On the other hand, it is not difficult to see that
\begin{eqnarray*}
\langle\bar{M}(\mathbf{1})\rangle_{T \wedge \tau_{n}}=  2 \int_{0}^{T\wedge \tau_{n}} \int_{\hat{\mathcal{X}}} \sigma(x)  \nu_{s}^{\prime}({\rm d}x) {\rm d}s + \int_{0}^{T\wedge \tau_{n}} \int_{\hat{\mathcal{X}}} \int_{(0,l)} u^{2} \Pi(x, {\rm d}u) \nu_{s }^{\prime}({\rm d}x) {\rm d}s,
\end{eqnarray*}

\noindent under $\mathbb{Q}_{\mu}$. Then, Assumption \ref{assum3} implies that
\begin{eqnarray} \label{eq24}
\mathbb{E}_{\mu} \left[ \left( \sup_{t \in [0,T]} \bar{M}_{t \wedge \tau_{n}}(\mathbf{1}) \right)^{2} \right] \leq C^{\prime} \mathbb{E}_{\mu} \left[  \int_{0}^{T} \langle \nu_{s \wedge \tau_{n}}^{\prime}, \mathbf{1} \rangle {\rm d}s  \right]  \leq C^{\prime} \left( \frac{1}{C_{1}} (\langle \mu, \mathbf{1} \rangle +1)  (e^{C_{1}T}-1) - T \right).
\end{eqnarray}
We observe that
\begin{eqnarray*}
\langle \nu_{t \wedge \tau_{n}}^{\prime}, \mathbf{1} \rangle \leq \bar{M}_{t \wedge \tau_{n}}(\mathbf{1} ) +  \langle \mu, \mathbf{1} \rangle    + \int_{0}^{t\wedge \tau_{n}} \int_{\hat{\mathcal{X}}}    \left( b(x)  - \int_{[l, \infty)} u \Pi(x, {\rm d} u)\right) \nu_{s}^{\prime}({\rm d}x){\rm d} s
\end{eqnarray*}

\noindent Then, 
(\ref{eq17}) and (\ref{eq24}) as well as Assumption \ref{assum3} imply that
\begin{eqnarray} \label{eq16}
\mathbb{E}_{\mu}  \big[  \sup_{t \in [0,T]} \langle \nu_{t \wedge \tau_{n}}^{\prime}, \mathbf{1} \rangle \big] \leq \bar{C} \left( C^{\prime}, C_{1}, T, \langle \mu, \mathbf{1} \rangle  \right).
\end{eqnarray}

\noindent where $\bar{C} \left( C^{\prime}, C_{1}, T, \langle \mu, \mathbf{1} \rangle  \right)$ is a positive constant. We observe that
\begin{eqnarray*}
\mathbb{Q}_{\mu}( \tau_{n} \leq T)  \leq 
\mathbb{Q}_{\mu} \left( \sup_{t \in [0,T]} \langle \nu_{t \wedge \tau_{n}}^{\prime}, \mathbf{1} \rangle  \geq n \right),
\end{eqnarray*}

\noindent and we notice that the Markov inequality together with the estimation (\ref{eq16}) implies that
\begin{eqnarray*}
\mathbb{Q}_{\mu} \left( \sup_{t \in [0,T]} \langle \nu_{t \wedge \tau_{n}}^{\prime}, \mathbf{1} \rangle  \geq n \right) \rightarrow 0, \hspace*{5mm} \text{as} \hspace*{5mm} n \rightarrow \infty.
\end{eqnarray*}

Finally, we have shown that $\mathbb{Q}_{\mu}( \tau_{n} \leq T) \rightarrow 0$ as $n \rightarrow \infty$ and so, we deduce that $\mathbf{E}_{\mu}^{\prime}[\mathfrak{z}_{T}] = 1$. Therefore, $\mathfrak{z}$ is a martingale as required. \\

\noindent {\bf Proof of Claim.} We first check that the random measure $\hat{\bar{N}}_{n}$ is a $\mathbb{Q}_{\mu}$-compensator of the optional random measure $\bar{N}_{n}$. Let $\theta_{n}$ be a stopping time such that $\theta_{n} \leq \tau_{n}$ for $n \geq 1$. Let $B \subset \mathcal{M}(\hat{X}) \setminus \{ {\bf 0}\}$ be a measurable set. Then,
\begin{align*}
\mathbb{E}_{\mu} \left[ \int_{0}^{t} \int_{\mathcal{M}(\hat{X})} \mathds{1}_{\{ s \leq \theta_{n} \}} \mathds{1}_{\{ \eta \in B \}} \langle \eta, \phi \rangle \bar{N}_{n}({\rm d}s, {\rm d} \eta) \right]& =  \mathbb{E}_{\mu} \left[ \int_{0}^{t \wedge \tau_{n}} \int_{\mathcal{M}(\hat{X})} \mathds{1}_{\{ s \leq \theta_{n} \}} \mathds{1}_{\{ \eta \in B \}} \langle \eta, \phi \rangle \bar{N}({\rm d}s, {\rm d} \eta) \right] \\
& = \mathbf{E}_{\mu} \left[ \mathfrak{z}_{t \wedge \tau_{n}} \int_{0}^{t \wedge \tau_{n}} \int_{\mathcal{M}(\hat{X})} \mathds{1}_{\{ s \leq \theta_{n} \}} \mathds{1}_{\{ \eta \in B \}} \langle \eta, \phi \rangle N^{\prime}({\rm d}s, {\rm d} \eta) \right] \\
& = \mathbf{E}_{\mu} \left[ \mathfrak{z}_{t \wedge \tau_{n}} \int_{0}^{t \wedge \tau_{n}} \int_{\mathcal{M}(\hat{X})} \mathds{1}_{\{ s \leq \theta_{n} \}} \mathds{1}_{\{ \eta \in B \}} \langle \eta, \phi \rangle \tilde{N}^{\prime}({\rm d}s, {\rm d} \eta) \right] \\
 & \hspace*{10mm} + \mathbb{E}_{\mu} \left[ \int_{0}^{t \wedge \tau_{n}} \int_{\mathcal{M}(\hat{X})} \mathds{1}_{\{ s \leq \theta_{n} \}} \mathds{1}_{\{ \eta \in B \}} \langle \eta, \phi \rangle \hat{N}^{\prime}({\rm d}s, {\rm d} \eta) \right] \\
 & = \mathbb{E}_{\mu} \left[ \int_{0}^{t} \int_{\mathcal{M}(\hat{X})} \mathds{1}_{\{ s \leq \theta_{n} \}} \mathds{1}_{\{ \eta \in B \}} \langle \eta, \phi \rangle \hat{\bar{N}}_{n}({\rm d}s, {\rm d} \eta) \right],
\end{align*}

\noindent where we have used the fact that
\begin{align*}
 & \mathbf{E}_{\mu} \left[ \mathfrak{z}_{t \wedge \tau_{n}} \int_{0}^{t \wedge \tau_{n}} \int_{\mathcal{M}(\hat{X})} \mathds{1}_{\{ s \leq \theta_{n} \}} \mathds{1}_{\{ \eta \in B \}} \langle \eta, \phi \rangle \tilde{N}^{\prime}({\rm d}s, {\rm d} \eta) \right] \\
& \hspace*{20mm} = \mathbf{E}_{\mu} \left[ \left \langle \mathfrak{z}_{ \cdot \wedge \tau_{n}}, \int_{0}^{ \cdot \wedge \tau_{n}} \int_{\mathcal{M}(\hat{X})} \mathds{1}_{\{ s \leq \theta_{n} \}} \mathds{1}_{\{ \eta \in B \}} \langle \eta, \phi \rangle \tilde{N}^{\prime}({\rm d}s, {\rm d} \eta) \right \rangle_{t} \right] = 0.
\end{align*}

\noindent This follows from well-known results of square integrable martingales (recall Lemma \ref{lemma2}), and the fact that $\mathds{1}_{ \cdot \wedge \tau_{n}} \tilde{N}^{\prime}({\rm d}s, {\rm d} \eta)$ has bounded variation while $\mathfrak{z}_{ \cdot \wedge \tau_{n}}$ does not. Then our first claim follows in view of the arbitrariness of $\theta_{n}$ and $B$; see \cite[Chapter 4, Section 5, p. 222]{Lip1989} and \cite[Proof of Theorem III.3.17]{JaSh1987}.

Recall that $\mathfrak{z}_{\cdot \wedge \tau_{n}}$,
$M^{\text{c}^{\prime}}_{\cdot \wedge \tau_{n}}(\phi)$ and $M^{\text{d}^{\prime}}_{\cdot \wedge \tau_{n}}(\phi)$ are square integrable martingales (recall also Lemma \ref{lemma2}). Moreover, their quadratic characteristic are defined as
\begin{eqnarray*}
 \left \langle \mathfrak{z}_{\cdot \wedge \tau_{n}}, M^{\text{c}^{\prime}}_{\cdot \wedge \tau_{n}}(\phi) \right \rangle_{t} & = & \int_{0}^{t \wedge \tau_{n}} \int_{\hat{\mathcal{X}}} \phi(x) \mathfrak{z}_{s} a(\nu_{s}^{\prime}, x) R(\nu_{s}^{\prime}, {\rm d} x, {\rm d} x ) {\rm d} s \\
 & = & \int_{0}^{t \wedge \tau_{n}} \int_{\hat{\mathcal{X}}} \phi(x) \mathfrak{z}_{s} \left( \int_{\hat{\mathcal{X}}} c(x,y) \nu_{s}^{\prime}({\rm d} y) \right) \nu_{s}^{\prime}({\rm d}x) {\rm d} s,
\end{eqnarray*}

\noindent and $\left \langle  \mathfrak{z}_{\cdot \wedge \tau_{n}}, M^{\text{d}^{\prime}}_{\cdot \wedge \tau_{n}}(\phi) \right \rangle_{t} = 0$. By \cite[Theorem 2, Chapter 4, Section 5]{Lip1989} (see also \cite[Theorem 39, p. 134]{Pro2005}), we know that
\begin{eqnarray*}
\bar{M}^{\text{c}^{\prime}}_{\cdot \wedge \tau_{n}}(\phi) = M^{\text{c}^{\prime}}_{\cdot \wedge \tau_{n}}(\phi) - \int_{0}^{\cdot} \mathfrak{z}_{(s \wedge \tau_{n})-}^{-1} {\rm d} \left \langle \mathfrak{z}_{\cdot \wedge \tau_{n}}, M^{\text{c}^{\prime}}_{\cdot \wedge \tau_{n}}(\phi) \right \rangle_{s}
\end{eqnarray*}

\noindent and
\begin{eqnarray*}
\bar{M}^{\text{d}^{\prime}}_{\cdot \wedge \tau_{n}}(\phi) = M^{\text{d}^{\prime}}_{\cdot \wedge \tau_{n}}(\phi) - \int_{0}^{\cdot} \mathfrak{z}_{(s \wedge \tau_{n})-}^{-1} {\rm d} \left \langle \mathfrak{z}_{\cdot \wedge \tau_{n}}, M^{\text{d}^{\prime}}_{\cdot \wedge \tau_{n}}(\phi) \right \rangle_{s} = M^{\text{d}^{\prime}}_{\cdot \wedge \tau_{n}}(\phi),
\end{eqnarray*}

\noindent are continuous and purely discontinuous $\mathbb{Q}_{\mu}$-martingales. Therefore, we conclude that under $\mathbb{Q}_{\mu}$ the martingales $\bar{M}^{\text{c}^{\prime}}_{\cdot \wedge \tau_{n}}(\phi)$ and $\bar{M}^{\text{d}^{\prime}}_{\cdot \wedge \tau_{n}}(\phi)$ obey the desired representations and it shows our {\bf Claim}.

Finally, the points (b) and (c) follow from \cite[Theorem 2, Chapter 4, Section 5]{Lip1989} and a similar argument as the {\bf Claim}.
\end{proof}

The following proposition shows uniqueness for the killed martingale problem.

\begin{proposition} \label{prop3}
Under Assumption \ref{assum1} and \ref{assum3}, there is a unique solution of the killed martingale problem, or equivalently, for the (\ref{eq20}) martingale problem.
\end{proposition}

\begin{proof}
The result is a direct consequence of Theorem \ref{teo5} (see for example \cite[Section 5]{Da1978} or \cite[Theorem 2.3]{Evans1994} in a similar setting). More precisely, under the measure $\mathbb{Q}_{\mu}$ (that is equivalent to the measure $\mathbf{Q}^{\prime}_{\mu}$), the martingale problem $\bar{M}(\phi) = (\bar{M}_{t}(\phi), t \in [0,T])$ corresponds to the one of general measure-valued Markov branching processes studied by Dawson \cite[Theorem 6.1.3]{Dawson1991} and Fitzsimmons \cite{Fitz1992}.
\end{proof}

\subsection{Localization method} \label{Uni3}

In this section, we will consider the localization procedure introduced by Stroock \cite{Stroock1975} and generalized to the measure-valued context by He \cite{He2009} in order to show that uniqueness for the martingale problem (\ref{eq10}) follows from the uniqueness of the martingale problem (\ref{eq20}). We briefly describe the idea. First, we show that each solution of the martingale problem (\ref{eq10}) behaves in the same way as the solution of the killed martingale problem until it has a ``big jump''. Since we have proven uniqueness for the
killed martingale problem, the solution of the martingale problem (\ref{eq10}) is uniquely determined before it has a ``big jump''. Furthermore, we show that when a ``big jump'' event happens, the jump size is also uniquely determined. Finally, we prove by induction that the distribution of the branching particle system corresponding to the martingale problem (\ref{eq10}) is uniquely determined, since after the first ``big jump'' event happens, the system also behaves in the same way as the solution of the killed martingale problem until the second ``big jump'' event happens. We point out that the arguments and results of this section are similar to those of  \cite{He2009, Stroock1975}, thus we are going to provide as many details as necessary for clarity and the convenience of the reader but leave out cumbersome steps that are entirely analogous. In order for this to be possible, we translate our set-up to the notation used in \cite{He2009}. We first start with some preliminary notation and definitions.

\begin{definition} \label{def2}
For $\mu \in \mathcal{M}(\hat{\mathcal{X}})$ and $0< T< +\infty$, we say that a stochastic process $\nu \in \mathbb{D}([0,T], \mathcal{M}(\hat{\mathcal{X}}))$, or equivalently its law $\mathbf{P}_{\mu}$, solves the $(\mathscr{L}, \mathscr{D}(\mathscr{L}), \mu)$-martingale problem if $\mathbf{P}_{\mu}(\nu_{0} = \mu) = 1$ and
\begin{eqnarray*}
 F(\nu_{t}) - F(\nu_{0}) - \int_{0}^{t} \mathscr{L} F(\nu_{s}) {\rm d}s, \hspace*{6mm} t \in [0,T],
\end{eqnarray*}

\noindent is a  $\mathbf{P}_{\mu}$-martingale, for all $F$ in some appropriate domain of functions on $ \mathscr{D}(\mathscr{L}) \subset B(\mathcal{M}(\hat{\mathcal{X}}), \mathbb{R})$.
\end{definition}

We consider the following two operators,
\begin{align*}
\mathscr{L}F(\mu) & =   \int_{\mathcal{\hat{X}}} p(x) A \left( \frac{\delta F(\mu)}{\delta \mu (x)} \right) \mu({\rm d} x) 
 +  \int_{\mathcal{\hat{X}}}   \left( b(x) - \int_{\mathcal{\hat{X}}} c(x,y)  \mu({\rm d}y) \right)\frac{\delta F(\mu)}{\delta \mu (x)} \mu({\rm d}x) \nonumber \\
& \hspace*{2mm} + \int_{\mathcal{\hat{X}}} \int_{0}^{\infty} \left( F(\mu + u \delta_{x})  - F(\mu) - \frac{\delta F(\mu)}{\delta \mu (x)} u \right) \Pi(x, {\rm d}u) \mu({\rm d}x)  +  \int_{\mathcal{\hat{X}}}   \sigma(x) \frac{\delta^{2} F(\mu)}{\delta \mu (x)^{2}} \mu({\rm d}x)
\end{align*}

\noindent and
\begin{align*}
\mathscr{L}^{\prime}F(\mu) & =   \int_{\mathcal{\hat{X}}} p(x) A \left( \frac{\delta F(\mu)}{\delta \mu (x)} \right) \mu({\rm d} x)   + \int_{\mathcal{\hat{X}}}   \left( b(x)  -\int_{[l, \infty)} u \Pi(x, {\rm d}u) - \int_{\mathcal{\hat{X}}} c(x,y)  \mu({\rm d}y)\right) \frac{\delta F(\mu)}{\delta \mu (x)} \mu({\rm d}x) \nonumber \\
& \hspace*{2mm} + \int_{\mathcal{\hat{X}}} \int_{(0,l)} \left( F(\mu + u \delta_{x})  - F(\mu) - \frac{\delta F(\mu)}{\delta \mu (x)} u \right) \Pi(x, {\rm d}u) \mu({\rm d}x)  +  \int_{\mathcal{\hat{X}}}   \sigma(x) \frac{\delta^{2} F(\mu)}{\delta \mu (x)^{2}} \mu({\rm d}x)
\end{align*}

\noindent defined for some appropriate functions $F$ in $B(\mathcal{M}(\mathcal{\hat{X}}), \mathbb{R})$, where the so-called variational derivatives are defined by
\begin{eqnarray*}
\frac{\delta F(\mu)}{\delta \mu (x)} := \lim_{h \downarrow 0} \frac{F(\mu + h \delta_{x}) - F(\mu)}{h} = \frac{\partial}{\partial h} F(\mu + h \delta_{x}) |_{h = 0}, \hspace*{5mm} x \in \hat{\mathcal{X}}
\end{eqnarray*}

\noindent and
\begin{eqnarray*}
\frac{\delta^{2} F(\mu)}{\delta \mu (x) \delta \mu(y)} := \frac{\partial^{2}}{\partial h_{1} \partial h_{2}} F(\mu + h_{1} \delta_{x} + h_{2} \delta_{y}) |_{h_{1} = h_{2} = 0}, \hspace*{5mm} x,y \in \hat{\mathcal{X}}.
\end{eqnarray*}

\noindent For $n \geq 0$, we define the function $F_{n}^{\phi, \lambda, \theta}$ on $\mathcal{M}(\hat{\mathcal{X}})$ by
\begin{eqnarray*}
F_{n}^{\phi, \lambda, \theta}(\mu) := \sum_{i=1}^{n} \theta_{i} e^{- \lambda_{i} \langle \mu, \phi_{i} \rangle},
\end{eqnarray*}

\noindent where $\mu \in \mathcal{M}(\hat{\mathcal{X}})$, $\lambda_{i} \in \mathbb{R}_{+}$, $\theta_{i} \in \mathbb{R}$ and $\phi_{i} \in C_{\partial}(\mathcal{X}, \mathbb{R}_{+})$, for $i = 1, \dots, n$. We also define the family of functions
\begin{eqnarray*}
 \mathscr{E} := \bigcup_{n =1}^{\infty}\left \{ F_{n}^{\phi, \lambda, \theta}(\mu): \mu \in \mathcal{M}(\hat{\mathcal{X}}), \hspace*{2mm} \lambda_{i} \in \mathbb{R}_{+}, \hspace*{2mm} \theta_{i} \in \mathbb{R} \hspace*{3mm}  \text{and} \hspace*{3mm} \phi_{i} \in C_{\partial}(\mathcal{X}, \mathbb{R}_{+}), \hspace*{3mm}  \text{for} \hspace*{2mm} i = 1, \dots, n   \right \}
\end{eqnarray*}

\noindent that is a subset of $B(\mathcal{M}(\mathcal{\hat{X}}), \mathbb{R})$.

\begin{remark} \label{Rem2}
 Recall that $\hat{\mathcal{X}}$ is the one-point compactification of $\mathcal{X}$. Thus $\mathcal{M}(\hat{\mathcal{X}})$ equipped with the weak topology is also compact. On the other hand, recall from the proof of tightness (Proposition \ref{prop2}) that $\mathscr{E}$ separates points on $\mathcal{M}(\hat{\mathcal{X}})$ (this follows from Dynkin's $\pi$-$\lambda$ Theorem). Moreover, $\mathscr{E}$ has the non-vanishing property, i.e, for every $\mu \in \mathcal{M}(\hat{\mathcal{X}})$ there exists $F_{n}^{\phi, \lambda, \theta} \in \mathscr{E}$ such that $F_{n}^{\phi, \lambda, \theta}(\mu) \neq 0$. Therefore the Stone-Weierstrass Theorem (see for example \cite[Appendix A7, Theorem 5, p. 393]{Ash1972}) implies that $\mathscr{E}$ is dense in $C_{0}(\mathcal{M}(\mathcal{\hat{X}}), \mathbb{R})$.
\end{remark}

We make the link between the martingale problems (\ref{eq10}), (\ref{eq20}) and Definition \ref{def2}.

\begin{corollary} \label{cor1}
Let
\begin{eqnarray*}
\mathscr{D}(\mathscr{L}) = \bigcup_{n=0}^{\infty} \left \{ F_{n}^{\phi, \lambda, \theta} \in \mathscr{E}: \phi_{i} \in \mathscr {D}(\mathbf{M}), \hspace*{3mm} \text{for} \hspace*{2mm} i=1, \dots, n \right \}.
\end{eqnarray*}

\noindent Then, $\mathbf{Q}_{\mu}$ (resp. $\mathbf{Q}_{\mu}^{\prime}$) solves the $(\mathscr{L}, \mathscr{D}(\mathscr{L}), \mu)$-martingale problem (resp. $(\mathscr{L}^{\prime}, \mathscr{D}(\mathscr{L}), \mu)$-martingale problem) if and only if it solves the martingale problem (\ref{eq10}) (resp. (\ref{eq20})).
\end{corollary}

\begin{proof}
The result is a consequence of Theorem \ref{teo3} (b) and its proof as well as an application of It\^{o}'s formula.
\end{proof}

By Proposition \ref{prop3}, we henceforth assume throughout this section that for
$\mu \in \mathcal{M}(\mathcal{X})$, there is a unique solution to the martingale problem (\ref{eq20}) which is the killed martingale problem.

Let $\omega = (\omega_{t}, t \in [0,T])$ denote the coordinate process of $\mathbb{D}([0,T], \mathcal{M}(\hat{\mathcal{X}}))$ and let $\mathbf{Q}^{\prime}$ denote the unique solution of the killed martingale problem. For $0 \leq s < T < + \infty$ and $\mu \in \mathcal{M}(\mathcal{X})$ let $\mathbf{Q}^{\prime}_{s, \mu} = \mathbf{Q}^{\prime}(\cdot | \omega_{s} = \mu)$. Hence $\mathbf{Q}^{\prime}_{s, \mu}$ is also a unique solution of the killed martingale problem starting from time $s$ at the value $\mu$. We set
\begin{eqnarray*}
\Omega = \mathbb{D}([0,T], \mathcal{M}(\hat{\mathcal{X}})), \hspace*{6mm} \mathcal{F}_{t} = \sigma(\omega_{s}: 0 \leq s \leq t \leq T) \hspace*{6mm} \text{and} \hspace*{6mm} \mathcal{F}^{t} = \sigma \left( \bigcup_{t \leq s \leq T}\mathcal{F}_{s}  \right),
\end{eqnarray*}

\noindent for $0 \leq t \leq T$. In particular, we write $\mathcal{F} = \mathcal{F}^{0}$. For $\mu \in \mathcal{M}(\mathcal{X})$  let $\mathbf{Q}_{\mu}$ be a solution of the martingale problem (\ref{eq10}). Then, for  $\phi \in \mathscr {D}(\mathbf{M})$ the process $M(\phi) = (M_{t}(\phi), t \in [0, T])$ defined in (\ref{eq10}) (with $\omega$ instead of $\nu$) is a $\mathbf{Q}_{\mu}$-martingale. We set
\begin{eqnarray*}
\omega_{t}^{l} := \omega_{t} - \int_{0}^{t} \int_{\mathcal{M}(\hat{\mathcal{X}})} \mu \cdot \mathds{1}_{\{ \langle \mu, \mathbf{1} \rangle \geq l \} } N({\rm d}s, {\rm d} \mu), \hspace*{5mm} t \in [0,T],
\end{eqnarray*}

\noindent for $ 1 < l < \infty$ and where $N$ is the optional random measure on $[0, \infty) \times \mathcal{M}(\hat{\mathcal{X}})$ associated with the purely discontinuous part of $M$ in Theorem \ref{teo3}. Recall that $\hat{N}$ and $\tilde{N}$ denote the compensator and compensated random measure of $N$, respectively. Recall also that $M^{\text{c}}(\phi)$ denotes the continuous martingale part of $M$. Then, Theorem \ref{teo3} implies that
\begin{align*}
\langle \omega_{t}^{l}, \phi \rangle & = \langle \omega_{0}, \phi \rangle + \int_{0}^{t} \int_{\hat{\mathcal{X}}}  p(x) A \phi(x) \omega_{s}({\rm d} x) {\rm d} s  + \int_{0}^{t} \int_{\hat{\mathcal{X}}}  \phi(x)
 \left( b(x)  - \int_{\hat{\mathcal{X}}} c(x,y) \omega_{s}({\rm d}y) \right) \omega_{s}({\rm d}x){\rm d} s  \\
& \hspace*{5mm} + M_{t}^{\text{c}}(\phi) +  \int_{0}^{t} \int_{\mathcal{M}(\hat{\mathcal{X}})} \langle \mu, \phi \rangle  \mathds{1}_{\{ \langle \mu, \mathbf{1} \rangle < l \} } \tilde{N}({\rm d}s, {\rm d} \mu) - \int_{0}^{t} \int_{\mathcal{M}(\hat{\mathcal{X}})} \langle \mu, \phi \rangle  \mathds{1}_{\{ \langle \mu, \mathbf{1} \rangle \geq l \} } \hat{N}({\rm d}s, {\rm d} \mu),
\end{align*}

\noindent for $t \in [0,T]$. Thus, by It\^{o}'s formula (see also the computation in (\ref{eq7})) the process $I^{n} = (I_{t}^{n}, t \in [0,T])$ for $n \geq 0$ an integer given by
\begin{align*}
I_{t}^{n} & = F_{n}^{\phi, \lambda, \theta}(\omega_{t}^{l}) - F_{n}^{\phi, \lambda, \theta}(\omega_{0}^{l})  -  \int_{0}^{t} \mathscr{L}^{\prime} F_{n}^{\phi, \lambda, \theta}(\omega_{s}^{l}) {\rm d}s
\end{align*}

\noindent is a local martingale under $\mathbf{Q}_{\mu}$, where $\lambda_{i} \in \mathbb{R}_{+}$, $\theta_{i} \in \mathbb{R}$ and $\phi_{i} \in \mathscr{D}(\mathbf{M})$, for $i = 1, \dots, n$.

Let $\tau^{1}(\omega) = \inf \{t \geq 0: \langle \omega_{t}, \mathbf{1} \rangle \geq l + \langle \omega_{0}, \mathbf{1} \rangle \} \wedge T$ and $\tau^{2}(\omega) = \inf \{t \geq 0: | \langle \omega_{t}, \mathbf{1} \rangle - \langle \omega_{t-}, \mathbf{1} \rangle| \geq l \}$. Set $\tau(\omega) = \tau^{1}(\omega) \wedge \tau^{2}(\omega)$. The following lemma gives another martingale characterization for $\omega^{l}$.

\begin{lemma} \label{lemma4}
For $\mu \in \mathcal{M}(\mathcal{X})$, let $\mathbf{P_{\mu}}$ be a probability measure on
$(\Omega, \mathcal{F})$ such that $\mathbf{P}_{\mu}(\omega_{0} = \mu) = 1$. Then
the process $I(\phi)= (I_{t}(\phi), t \in [0,T])$ given by
\begin{align*}
I_{t}(\phi) & = \exp \bigg(- \langle \omega_{t \wedge \tau(\omega)}^{l}, \phi  \rangle +
\int_{0}^{t \wedge \tau(\omega)} \int_{\hat{\mathcal{X}}}   p(x) A \phi(x) \omega_{s}({\rm d} x) {\rm d} s \nonumber \\
& \hspace*{4mm} + \int_{0}^{t \wedge \tau(\omega)}  \int_{\hat{\mathcal{X}}} \left( b(x)  - \int_{\hat{\mathcal{X}}} c(x,y)  \omega_{s}({\rm d}y) - \int_{l}^{\infty} u  \Pi(x, {\rm d} u )\right) \phi(x) \omega_{s}({\rm d}x) {\rm d} s  \nonumber \\
& \hspace*{4mm} - \int_{0}^{t \wedge \tau(\omega)}  \int_{\hat{\mathcal{X}}} \int_{0}^{l} \left( e^{-u \phi(x)} - 1 + u \phi(x)  \right) \phi(x) \Pi(x, {\rm d} u ) \omega_{s}({\rm d}x) {\rm d} s  \nonumber \\
& \hspace*{4mm} - \int_{0}^{t \wedge \tau(\omega)}  \int_{\hat{\mathcal{X}}}  \sigma(x) \phi^{2}(x) \omega_{s}({\rm d}x) {\rm d} s \bigg) \nonumber
\end{align*}

\noindent is a $\mathbf{P}_{\mu}$-martingale for every $\phi \in \mathscr{D}(\mathbf{M})$ if and only if $(I^{n}_{t \wedge \tau}, t \in [0, T])$ is a $\mathbf{P}_{\mu}$-martingale for each $n \geq 1$.
\end{lemma}

\begin{proof}
The result follows from the formula for integration by parts and the same argument as in the proof of Theorem 7 in \cite{Elk1991}. Here, it is used that up to time $\tau(\omega)$ we have $\langle \omega_t, \mathbf{1} \rangle$ bounded almost surely.
\end{proof}

The next two theorems correspond to \cite[Theorem 2.4 and Theorem 2.5]{He2009}. 
The first result shows that the solution of the martingale problem (\ref{eq10}) is determined by the martingale problem (\ref{eq20}) before it has a jump of size larger than $1 < l < + \infty$. Define
\begin{eqnarray*}
\mathcal{F}_{\tau(\omega)} = \{E \in \mathcal{F}: \, E \cap \{\tau \leq t\} \in \mathcal{F}_{t} \, \, \text{for all} \, \, t \in [0,T] \}.
\end{eqnarray*}

\noindent It is not hard to see that $\mathcal{F}_{\tau(\omega)} = \sigma(\omega(t \wedge \tau): t \in [0,T])$. We also define
\begin{eqnarray*}
\mathcal{F}_{\tau(\omega)-} = \sigma(\{E \in \mathcal{F}: \, E \cap \{\tau > t\} \in \mathcal{F}_{t} \, \, \text{for all} \, \, t \in [0,T] \}).
\end{eqnarray*}

\begin{theorem} \label{teo7}
For $\mu \in \mathcal{M}(\mathcal{X})$, let $\mathbf{P}_{\mu}$ be a probability measure on  $(\Omega, \mathcal{F})$ such that $\mathbf{P}_{\mu}(\omega_{0} = \mu) = 1$ and $(I_{t \wedge \tau(\omega)}^{n}, t \in [0,T])$ is a $\mathbf{P_{\mu}}$-martingale for each $n \geq 1$. We define the measure $\mathbf{S}_{\omega} = \delta_{\omega} \otimes \mathbf{Q}_{\tau(\omega), \omega_{\tau(\omega)}^{l}}^{\prime}$ on $(\Omega, \mathcal{F})$ that satisfies
\begin{eqnarray*}
\mathbf{S}_{\omega}(E_{1} \cap E_{2}) = \mathds{1}_{ E_{1}}(\omega) \mathbf{Q}_{\tau(\omega), \omega_{\tau(\omega)}^{l}}^{\prime}(E_{2}), \hspace*{5mm} \text{for} \hspace*{3mm} E_{1} \in \mathcal{F}_{\tau(\omega)-} \hspace*{3mm} \text{and} \hspace*{3mm} E_{2} \in \mathcal{F}_{\tau(\omega)}.
\end{eqnarray*}

\noindent We set
\begin{eqnarray*}
\mathbf{P}_{\mu}^{\prime}(E) = \mathbf{E}_{\mathbf{P}_{\mu}}(\mathbf{S}_{\omega}(E)), \hspace*{5mm} E \in \mathcal{F},
\end{eqnarray*}

\noindent and define $\mathcal{F}_{\tau(\omega)-}^{l} = \sigma(\omega_{t \wedge \tau(\omega)}^{l}: t \in [0, T])$. Then $\mathbf{P}_{\mu}^{\prime}$ is a solution of the killed martingale problem and $\mathbf{P}_{\mu} = \mathbf{Q}_{\mu}^{\prime}$ on $\mathcal{F}_{\tau(\omega)-}^{l}$. In particular, we can take $\mathbf{P}_{\mu} = \mathbf{Q}_{\mu}$.
\end{theorem}

\begin{proof}
The statement is obtained along exactly the same lines as the proof of \cite[Theorem 2.4]{He2009}. 
\end{proof}

We now see that uniqueness of the killed martingale problem implies uniqueness for the solution $\mathbf{Q}_{\mu}$ of the martingale problem (\ref{eq10})
on $\mathcal{F}_{\tau(\omega)-}^{l}$. Our next step is to show that uniqueness of the killed martingale problem implies uniqueness of $\mathbf{Q}_{\mu}$ on $\mathcal{F}_{\tau(\omega)}$. The next theorem shows that when a jump of size larger than $1 < l < + \infty$ happens, the jump size is uniquely determined by $\mathcal{F}_{\tau(\omega)-}^{l}$. We denote by $\mathbf{E}_{\mu}$ the expectation with respect to $\mathbf{Q}_{\mu}$.

\begin{theorem} \label{teo8}
For $1 < l < + \infty$ let $\mathcal{M}_{l}(\hat{\mathcal{X}}) = \{ \mu \in \mathcal{M}(\hat{\mathcal{X}}): \langle \mu, \mathbf{1} \rangle \geq l \}$. There is an $\mathcal{F}_{\tau(\omega)-}^{l}$-measurable function $\tau^{\prime}: \Omega \rightarrow  [0,T]$ such that for $E \in \mathbb{B}(\mathcal{M}_{l}(\hat{\mathcal{X}}))$
\begin{align*}
& \mathbf{E}_{\mu} \left[ N((0, \tau(\omega)], E) \Big | \mathcal{F}_{\tau(\omega)-}^{l} \right] = \\
& \hspace*{5mm} \int_{0}^{\tau^{\prime}} \int_{\hat{\mathcal{X}}} \int_{0}^{\infty} \exp \left(-\int_{0}^{t} \int_{\hat{\mathcal{X}}} \int_{[l, \infty)} \Pi(x, {\rm d} u) \omega_{s \wedge \tau(\omega)}^{l}({\rm d} x) {\rm d} s \right) \mathds{1}_{E}(v\delta_{y})\Pi(y, {\rm d} v) \omega_{s \wedge \tau(\omega)}^{l}({\rm d} y) {\rm d} t
\end{align*}

\noindent holds for any solution $\mathbf{Q}_{\mu}$ of the martingale problem (\ref{eq10}). In particular, given $\mathcal{F}_{\tau(\omega)-}^{l} $ the distribution of the random measure $N$ up to time $\tau(\omega)$ is
uniquely determined.
\end{theorem}

\begin{proof}
The formula for the conditional expectation follows from \cite[Theorem 2.5]{He2009} by using Theorem \ref{teo7} and Lemma \ref{lemma4} to show that the requirements of \cite[Theorem 3.2]{Stroock1975} are satisfied. Since the distribution of the random measure $N$ up to time $\tau(\omega)$ is characterized by its intensity the result follows.
\end{proof}
Since
\begin{eqnarray*}
\omega_{\tau(\omega)} = \omega_{\tau(\omega)}^{l} + \int_{(0, \tau(\omega)]} \int_{\mathcal{M}(\hat{\mathcal{X}})} \mu \cdot \mathds{1}_{\{ \langle \mu, \mathbf{1} \rangle \geq l \} } N({\rm d}s, {\rm d} \mu),
\end{eqnarray*}

\noindent we see that the distribution of $\omega_{\tau(\omega)}$ under $\mathbf{Q}_{\mu}$ given $\mathcal{F}_{\tau(\omega)-}^{l}$ is uniquely determined, and therefore Theorem \ref{teo7} implies that the measure $\mathbf{Q}_{\mu}$ is uniquely determined on $\mathcal{F}_{\tau(\omega)}$.

\begin{lemma} \label{lemma3}
Let $\mathbf{Q}_{\mu}$ be a solution of the martingale problem (\ref{eq10}). Let $\beta(\omega)$ be a finite stopping time and let $\mathcal{Q}_{\omega}$ be a regular conditional probability distribution of $\mathbf{Q}_{\mu} | \mathcal{F}_{\beta(\omega)}$. Then, there is a set $E \in \mathcal{F}_{\beta(\omega)}$ such that $\mathbf{Q}_{\mu}(E) = 0$ and when $\omega \not \in E$,
\begin{eqnarray*}
F_{n}^{\phi, \lambda, \theta}( \omega_{t \vee \beta(\omega)} ) - F_{n}^{\phi, \lambda, \theta}( \omega_{ \beta(\omega)} ) - \int_{\beta(\omega)}^{t \vee \beta(\omega)} \mathscr{L} F_{n}^{\phi, \lambda, \theta}( \omega_{s} ) {\rm d}s, \hspace*{5mm} t \in [0,T]
\end{eqnarray*}

\noindent is a $\mathcal{Q}_{\omega}$-martingale for $F_{n}^{\phi, \lambda, \theta} \in \mathscr{D}(\mathscr{L})$.
\end{lemma}

\begin{proof}
The result is proved in the same way as \cite[Theorem 1.2.10]{Stroock1979}.
\end{proof}

We now state the main theorem in this section.

\begin{theorem}
Suppose that for $1 < l < + \infty$ there is a unique solution of the martingale problem (\ref{eq20}), or equivalently the killed martingale problem. Then there is a unique solution of the martingale problem (\ref{eq10}).
\end{theorem}

\begin{proof}
The argument of this proof is exactly the same as that in \cite[Theorem 2.6]{He2009}. Thus, we only sketch it here. Suppose that $\mathbf{Q}_{\mu}$ is a solution of the martingale problem (\ref{eq10}) and observe that $(I_{t \wedge \tau(\omega)}^{n}, t \in [0,T])$ is a $\mathbf{Q}_{\mu}$-martingale for each $n \geq 1$. Define the following sequence of stopping times, $\beta_{0} = 0$ and
\begin{eqnarray*}
\beta_{n+1} = \left(\inf\{t \geq \beta_{n}: |\langle \omega_{t}, \mathbf{1} \rangle - \langle \omega_{t-}, \mathbf{1} \rangle | \geq l \hspace*{3mm} \text{or} \hspace*{3mm} \langle \omega_{t}, \mathbf{1} \rangle - \langle \omega_{\beta_{n}}, \mathbf{1} \rangle  \geq l \} \right) \wedge (\beta_{n} +1).
\end{eqnarray*}

\noindent Notice that for each $n \geq 1$, $\beta_{n}$ is bounded by $nl$. By Lemma \ref{lemma3} and Theorem \ref{teo8}, we can prove by induction that $\mathbf{Q}_{\mu}$ is uniquely determined on $\mathcal{F}_{\beta_{n}}$ for all $n \geq 1$. Therefore, it is enough to check that $\mathbf{Q}_{\mu}(\beta_{n} \leq T) \rightarrow 0$ as $n \rightarrow \infty$ for each $T > 0,$ which follows along exactly the same lines as in \cite[Theorem 2.6]{He2009}.
\end{proof}

Finally, the previous result together with Proposition \ref{prop3} concludes the proof of Theorem \ref{teo4}.

\section{Proof of Theorem \ref{teo9}} \label{lastT}

In this section, we just check that no mass escapes for the unique solution to the martingale problem (\ref{eq10}). The proof follows exactly as in \cite[Theorem 3.1 and Theorem 3.2]{He2009}. Specifically, one first shows that the solution to the killed martingale problem (\ref{eq20}) is actually the law of a measure valued process in $\mathbb{D}([0,T], \mathcal{M}(\mathcal{X}))$. From this one builds a solution to the martingale problem (\ref{eq10}) that is the law of a measure valued process in $\mathbb{D}([0,T], \mathcal{M}(\mathcal{X}))$ and concludes by uniqueness.

Recall that we are assuming that

\begin{assumption} \label{assum4}
There exists a sequence $(\phi_{n})_{n \geq 1} \subset\mathscr {D}(\mathbf{M})$ such that
\begin{itemize}
\item[(i)] for $n \geq 1$, $\lim_{d(x, \partial) \rightarrow 0} \phi_{n}(x) = 1$, $\lim_{d(x, \partial) \rightarrow 0} A\phi_{n}(x)  = 0$ (where $d$ is some metric in the Polish space $\mathcal{X}$),
\item[(ii)] $\sup_{n \geq 1} \sup_{x \in \mathcal{X}} \phi_{n}(x) < \infty$ and $\sup_{n \geq 1} \sup_{x \in \mathcal{X}} A\phi_{n}(x) < \infty$. Furthermore, $ \phi_{n} \rightarrow \mathds{1}_{\{ \partial\}}$ and $ A\phi_{n} \rightarrow 0$, as $n \rightarrow \infty$, pointwise.
\end{itemize}
\end{assumption}

\begin{theorem}  \label{teo10}
Suppose that Assumptions \ref{assum3} and \ref{assum4} are satisfied with $\sigma \in C_{\partial}(\mathcal{X}, \mathbb{R}_{+})$.
 For a non random $\mu \in \mathcal{M}(\mathcal{X})$ let
 $\nu^{\prime} \in \mathbb{D}([0,T], \mathcal{M}(\hat{\mathcal{X}}))$, or equivalently its law $\mathbf{Q}^{\prime}_{\mu}$, solve the martingale problem (\ref{eq20}) for some $1<l<+ \infty.$  Then,
 \begin{eqnarray*}
 \mathbf{Q}_{\mu}^{\prime} ( \nu_{t}(\{\partial \}) = 0 \hspace*{3mm} \text{for all} \hspace*{3mm} t \in [0,T] ) = 1.
 \end{eqnarray*}
\end{theorem}

\begin{proof}
Let $\phi \in \mathscr{D}(\mathbf{M})$. By Lemma \ref{lemma2}, the process $M^{\prime}(\phi) = (M_{t}^{\prime}(\phi), t \in [0,T])$ given by
\begin{align}  \label{exeq6}
M_{t}^{\prime}(\phi) & =  \langle \nu_{t}^{\prime}, \phi \rangle - \langle \nu_{0}^{\prime}, \phi \rangle - \int_{0}^{t} \int_{\hat{\mathcal{X}}}  p(x) A \phi(x)  \nu_{s}^{\prime}({\rm d} x) {\rm d} s \nonumber \\
& \hspace*{5mm}  - \int_{0}^{t} \int_{\hat{\mathcal{X}}}  \phi(x) 
\left( b(x)  - \int_{[l, \infty)} u \Pi(x, {\rm d} u) - \int_{\hat{\mathcal{X}}} c(x,y) \nu_{s}^{\prime}({\rm d}y)  \right) \nu_{s}^{\prime}({\rm d}x){\rm d} s
\end{align}

\noindent is a square-integrable $\mathbf{Q}_{\mu}^{\prime}$-martingale with quadratic variation process given by
\begin{eqnarray*}
\langle M^{\prime}(\phi) \rangle_{t }=  2 \int_{0}^{t} \int_{\hat{\mathcal{X}}} \sigma(x)  \phi^{2}(x) \nu_{s}^{\prime}({\rm d}x) {\rm d}s + \int_{0}^{t} \int_{\hat{\mathcal{X}}} \int_{(0,l)} u^{2} \phi^{2}(x) \Pi(x, {\rm d}u) \nu_{s}^{\prime}({\rm d}x) {\rm d}s, \hspace*{5mm} t \in [0,T],
\end{eqnarray*}

Let $(\phi_{n})_{n \geq 1} \subset\mathscr {D}(\mathbf{M})$ be a sequence that fulfills Assumption \ref{assum4}. Then, using Doob's inequality, we obtain
\begin{align*}
& \mathbf{E}_{\mu}^{\prime} \left[ \sup_{t \in [0,T]} \left| M_{t}^{\prime}(\phi_{i}) - M_{t}^{\prime}(\phi_{j}) \right|^{2}  \right] \leq 4 \int_{0}^{T} \mathbf{E}_{\mu}^{\prime} \left[  \int_{\hat{\mathcal{X}}}   (\phi^{2}_{i}(x) - \phi^{2}_{j}(x)) \left( 2\sigma(x) + \int_{(0,l)} u^{2}  \Pi(x, {\rm d}u) \right)  \nu_{s}^{\prime}({\rm d}x) \right] {\rm d}s ,
\end{align*}

\noindent for $i,j \geq 1$ Therefore, Assumptions \ref{assum3}, \ref{assum4} and the moment bound
$\sup_{t \in [0,T]} \mathbf{E}^{\prime}_{\mu}[\langle \nu_{t}^{\prime}, \mathbf{1} \rangle^{n}] < \infty, n \geq 1$ (from Lemma  \ref{lemma2}) together with the Dominated Convergence Theorem imply that
\begin{eqnarray*}
\lim_{i,j \rightarrow \infty} \mathbf{E}_{\mu}^{\prime} \left[ \sup_{t \in [0,T]} \left| M_{t}^{\prime}(\phi_{i}) - M_{t}^{\prime}(\phi_{j}) \right|^{2}  \right] = 0.
\end{eqnarray*}

\noindent So $M_{t}^{\prime}(\phi_{n})$ converges uniformly on $[0,T]$  in mean square as $n \rightarrow \infty$ to a limit that we denote by $M^{\partial} = (M_{t}^{\partial}, t \in [0,T]) $ and with probability one along an appropriate subsequence.
By \cite[Lemma 2.1.2]{Ike1981} we obtain that $M^{\partial}$ is a c\`adl\`ag square-integrable
martingale with $\mathbf{E}_{\mu}^{\prime}  [M^{\partial}_t]= \lim_{n \rightarrow \infty} \mathbf{E}_{\mu}^{\prime}  [M_{t}^{\prime}(\phi_{n})]= \lim_{n \rightarrow \infty} \mathbf{E}_{\mu}^{\prime}  [M_{0}^{\prime}(\phi_{n})]=0.$ Then from (\ref{exeq6}) we deduce by Lebesgue's convergence theorem that
\begin{align*}
M_{t}^{\partial} & =  \nu_{t}^{\prime}(\{ \partial\})    - \int_{0}^{t}  
 \left( b(\partial)  - \int_{[l, \infty)} u \Pi(\partial, {\rm d} u) - \int_{\hat{\mathcal{X}}} c(\partial,y) \nu_{s}^{\prime}({\rm d}y) \right) \nu_{s}^{\prime}(\{ \partial\}){\rm d} s.
\end{align*}

\noindent Taking expectations in the last equality and using Gronwall's inequality (with Assumption \ref{assum3}) yields
$\mathbf{E}_{\mu}^{\prime} [\nu_{t}^{\prime}(\{ \partial\})  ]  =0$, for all $t \in [0,T]$. Hence $\nu_{t}^{\prime}(\{ \partial\})   = 0$ with probability one, for all  $t \in [0,T]$ and now the conclusion follows from the right continuity of  $(\nu_{t}^{\prime}(\{ \partial\}), t \in [0,T])$.
\end{proof}

\begin{proof}[Proof of Theorem \ref{teo9}]

By Proposition \ref{prop3} and Theorem \ref{teo10}, there is an unique solution to the martingale problem (\ref{eq20}). Then one may follow exactly the same argument as in the proof of \cite[Theorem 3.2]{He2009} to conclude that there is a process $\nu^{\infty} \in \mathbb{D}([0,T], \mathcal{M}(\mathcal{X}))$ that is a solution to the martingale problem (\ref{eq10}). Therefore our claim is a consequence of Theorem \ref{teo4}. It is important to point out that in order to be in the framework of \cite[Theorem 3.2]{He2009} one needs to put the martingale problems (\ref{eq10}) and (\ref{eq20}) in the form of  Definition \ref{def2} by Corollary \ref{cor1} (recall also Remark \ref{Rem2}).
\end{proof}

\paragraph{Acknowledgements.}
This work was supported by the DFG-SPP Priority Programme 1590, {\sl Probabilistic Structures in Evolution}. 

\providecommand{\bysame}{\leavevmode\hbox to3em{\hrulefill}\thinspace}
\providecommand{\MR}{\relax\ifhmode\unskip\space\fi MR }
\providecommand{\MRhref}[2]{%
  \href{http://www.ams.org/mathscinet-getitem?mr=#1}{#2}
}
\providecommand{\href}[2]{#2}


\end{document}